\documentclass[11pt,reqno]{amsart}

\usepackage{amsmath}
\usepackage{amsthm}
\usepackage{graphicx}
\usepackage{amsfonts}
\usepackage{amssymb}
\usepackage{hyperref}
\usepackage{bm}
\usepackage[margin=1.25in]{geometry}
\usepackage{enumerate}
\numberwithin{equation}{section}

\newtheorem{theorem}{Theorem}[section]
\newtheorem{lemma}[theorem]{Lemma}
\newtheorem{proposition}[theorem]{Proposition}
\newtheorem{corollary}[theorem]{Corollary}

\theoremstyle{definition}
\newtheorem{example}[theorem]{Example}
\newtheorem{definition}[theorem]{Definition}
\newtheorem{remark}[theorem]{Remark}

\def\E{{\mathbb E}}

\def\R{{\mathbb R}}
\def\N{{\mathbb N}}
\def\PP{{\mathbb P}}
\def\FF{{\mathbb F}}
\def\P{{\mathcal P}}
\def\B{{\mathcal B}}

\def\W{{\mathcal W}}

\def\F{{\mathcal F}}
\def\tr{{\mathrm{Tr}}}

\def\C{{\mathcal C}}
\def\transpconst{\gamma}
\def\tailfn{\mathrm{S}}

\title[Hierarchies, entropy, and propagation of chaos]{Hierarchies, entropy, and quantitative propagation of chaos for mean field diffusions}
\author{Daniel Lacker}
\address{Department of Industrial Engineering \& Operations Research, Columbia University}
\email{daniel.lacker@columbia.edu}
\thanks{This work was partially supported by the Air Force Office of Scientific Research Grant FA9550-19-1-0291.}

\begin{document} 

\begin{abstract}
This paper develops a non-asymptotic, \emph{local} approach to quantitative propagation of chaos for a wide class of mean field diffusive dynamics. For a system of $n$ interacting particles, the relative entropy between the marginal law of $k$ particles and its limiting product measure is shown to be $O((k/n)^2)$ at each time, as long as the same is true at time zero. A simple Gaussian example shows that this rate is optimal. The main assumption is that the limiting measure obeys a certain functional inequality, which is shown to encompass many potentially irregular but not too singular finite-range interactions, as well as some infinite-range interactions. This unifies the previously disparate cases of Lipschitz versus bounded measurable interactions, improving the best prior bounds of $O(k/n)$ which were deduced from \emph{global} estimates involving all $n$ particles. We also cover a class of models for which qualitative propagation of chaos and even well-posedness of the McKean-Vlasov equation were previously unknown. At the center of a new approach is a differential inequality, derived from a form of the BBGKY hierarchy, which bounds the $k$-particle entropy in terms of the $(k+1)$-particle entropy.
\end{abstract}

\maketitle


\section{Introduction}

Consider the interacting stochastic dynamics 
\begin{align}
dX^{n,i}_t = \bigg(b_0(X^{n,i}_t) + \frac{1}{n-1}\sum_{j \neq i}b(X^{n,i}_t,X^{n,j}_t)\bigg)dt + dW^i_t, \qquad i=1,\ldots,n, \label{def:introSDE}
\end{align}
where $W^1,\ldots,W^n$ are independent $d$-dimensional Brownian motions. Precise assumptions on the confinement and interaction functions $b_0 : \R^d \to \R^d$ and $b : \R^d \times \R^d \to \R^d$ will be given in Section \ref{se:intro-dyn}, and we also treat higher-order (non-pairwise) interactions in Section \ref{se:intro-infrange}. In this introduction, assume for simplicity that all relevant stochastic differential equations (SDEs) admit unique solutions.

The $n\to\infty$ behavior of the SDE systems \eqref{def:introSDE} is typically described in terms of the McKean-Vlasov equation
\begin{align}
dX_t =  \big(b_0(X_t) + \langle \mu_t, b(X_t,\cdot)\rangle \big)\,dt +  dW_t, \qquad \mu_t = \mathrm{Law}(X_t), \label{def:introMV}
\end{align}
or by the corresponding nonlinear Fokker-Planck equation
\begin{align*}
\partial_t\mu_t(x) = -\mathrm{div}\big( \big(b_0(x)+  \langle \mu_t, b(x,\cdot)\rangle\big)\,\mu_t(x)\big) + \tfrac12 \Delta \mu_t(x). 
\end{align*}

The precise connection between the $n$-particle system \eqref{def:introSDE} and the McKean-Vlasov equation \eqref{def:introMV} was first studied in the context of interacting diffusions by McKean \cite{mckean1967propagation}, using Kac's formalism of (molecular) \emph{chaos} \cite{kac1956foundations}, also called \emph{the Boltzmann property}.
One expects to have \emph{propagation of chaos}, in the following sense.
Write $\P(E)$ for the set of probability measures on a Polish space $E$, always equipped with the topology of weak convergence. 
Let $P^{(n,k)}_t \in \P((\R^d)^k)$ denote the law of $(X^{n,1}_t,\ldots,X^{n,k}_t)$, for $1 \le k \le n$ and $t \ge 0$, where $(X^{n,1},\ldots,X^{n,n})$ solves \eqref{def:introSDE}.
 Suppose the initial distribution $P^{(n,n)}_0$ is exchangeable and \emph{$\mu_0$-chaotic}, in the sense that
\begin{align}
P^{(n,k)}_0 \to \mu^{\otimes k}_0 \text{ as } n\to\infty, \quad \text{for each } k \in \N,\label{eq:intro-propagation-0}
\end{align}
where $^{\otimes k}$ denotes the $k$-fold tensor product.
Then the same is true at any later time:
\begin{align}
P^{(n,k)}_t  \to \mu^{\otimes k}_t \text{ as } n\to\infty, \quad \text{for each } k \in \N, \ \ t > 0. \label{eq:intro-propagation}
\end{align}
This notion of chaos admits a well known reformulation in terms of the empirical measures $L^n_t := \frac{1}{n}\sum_{i=1}^n\delta_{X^{n,i}_t}$. For any $t \ge 0$, the following are equivalent \cite[Proposition 2.2]{sznitman1991topics}:
\begin{enumerate}[{\ \ }(a)]
\item $P^{(n,k)}_t \to \mu^{\otimes k}_t$ as $n\to\infty$, for each $k \in \N$.
\item $L^n_t \to \mu_t$ in probability as $n\to\infty$.
\end{enumerate}

Interacting SDE systems of the form \eqref{def:introSDE} and their $n\to\infty$ behavior arise in a wide range of theoretical and applied contexts, reaching far beyond their origins in kinetic theory. 
A classic reference is \cite{sznitman1991topics}, and see \cite{jabinwang-review} for a survey of more recent developments with an eye toward models arising in the physical sciences.
This short introduction can hardly do justice to the vast literature, but we mention two notable and active recent directions beyond those discussed in \cite{sznitman1991topics,jabinwang-review}:
McKean-Vlasov models feature prominently in recent work on neural networks, as scaling limits of stochastic gradient descent algorithms \cite{mei2018mean,rotskoff2018neural,sirignano2018mean}, and the recent literature on mean field game theory \cite{lasry-lions,huang2006large,carmona-delarue} has inspired new theoretical developments and applications in engineering and economics.

\emph{Qualitative} propagation of chaos, i.e., the implication \eqref{eq:intro-propagation-0} $\Rightarrow$ \eqref{eq:intro-propagation}, has been established for many different kinds of dynamics. The main contribution of this paper is to develop a new, non-asymptotic approach to \emph{quantitative} propagation of chaos, for which the literature is more limited.  
The key assumptions and the method are summarized in Sections \ref{se:intro:mainresults} and \ref{se:ideas} below, with precise definitions and statements given in Section \ref{se:mainresults}.
Informally, we show that
\begin{align}
H(P^{(n,k)}_0 \,|\, \mu^{\otimes k}_0) = O((k/n)^2) \ \ \Longrightarrow \ \ H(P^{(n,k)}_t \,|\, \mu^{\otimes k}_t) = O((k/n)^2) \ \ \ \forall t > 0, \label{eq:intro-entropy-new}
\end{align}
where $H$ denotes relative entropy.
By Pinsker's inequality, this implies a convergence rate in a genuine metric,
\begin{align}
\|P^{(n,k)}_t  - \mu^{\otimes k}_t\|_{\mathrm{TV}} = O(k/n), \label{eq:intro-TV-new}
\end{align}
where $\|\cdot\|_{\mathrm{TV}}$ denotes total variation.
This is contrary to an apparently widely held belief that $\|P^{(n,k)}_t  - \mu^{\otimes k}_t\|_{\mathrm{TV}} = O(\sqrt{k/n})$ is optimal \cite[Remark 4]{jabin-wang-W1inf}.
An explicitly solvable Gaussian case, Example \ref{ex:gaussian-dyn}, shows that the rate $(k/n)^2$ in \eqref{eq:intro-entropy-new} cannot be improved, even if $P^{(n,n)}_0=\mu_0^{\otimes n}$.
In fact, our main results are stronger, stated in terms of relative entropy between measures on path space, but we avoid introducing the requisite additional notation here.

First, we stress that the new method is purely \emph{local}, working directly with the distance between $P^{(n,k)}_t$ and $\mu^{\otimes k}_t$. No non-trivial global estimates are needed.
The equivalence (a) $\Leftrightarrow$ (b) mentioned above is \emph{qualitative} and does not appear to admit a natural \emph{quantitative} version, even once one fixes a preferred metric with which to quantify the distance between $P^{(n,k)}_t$ and $\mu^{\otimes k}_t$ in (a), or between $L^n_t$ and $\mu_t$ in (b). For this reason, we will distinguish (a) as \emph{local chaos}, as it pertains to finitely many coordinates at a time, versus (b) as \emph{global chaos}, as it involves the full configuration. 
In contrast to our purely local approach, prior work has
deduced \emph{local} estimates from \emph{global} ones, relying on estimating either a distance between the full configurations $P^{(n,n)}_t$ and $\mu^{\otimes n}_t$, or expected distances between $L^n_t$ and $\mu_t$. But, as we will see, local estimates deduced from global ones are often suboptimal.

For example, for Lipschitz coefficients $(b_0,b)$, the popular (synchronous) coupling method introduced by Sznitman \cite{sznitman1991topics} (used also in \cite{meleard1996asymptotic,malrieu2001logarithmic,
malrieu2003convergence} and many other works) allows one to prove a global estimate of the form $\W_2(P^{(n,n)}_t,\mu^{\otimes n}_t) = O(1)$, as long as the same is true at time $t=0$, where $\W_2$ denotes the quadratic Wasserstein distance (defined in Section \ref{se:mainresults}).
The same coupling argument also yields the local estimate
\begin{align}
\W_2^2(P^{(n,k)}_t,\mu^{\otimes k}_t) = O(k/n). \label{ineq:intro:W2}
\end{align}
Alternatively, this follows from the aforementioned global estimate and a well known subadditivity inequality \cite[Proposition 2.6(i)]{hauray2014kac}
\begin{align}
\W_2^2(P^{(n,k)}_t,\mu^{\otimes k}_t) \le \frac{2k}{n} \W_2^2(P^{(n,n)}_t,\mu^{\otimes n}_t). \label{def:subadditivity-W}
\end{align}
Although our main result \eqref{eq:intro-entropy-new} is for relative entropy, we will see that it can be used along with transport inequalities to improve the Wasserstein estimate \eqref{ineq:intro:W2} to $O((k/n)^2)$ in the case of Lipschitz coefficients.

In a different direction, for bounded coefficients $(b_0,b)$, the (techniques of the) more recent papers \cite{benarous-zeitouni,jabin-wang-bounded,jabir2019rate} establish  global relative entropy estimates of the form
\begin{align}
H(P^{(n,n)}_t \,|\, \mu^{\otimes n}_t) = O(1), \label{def:intro:entropy-global}
\end{align}
as long as the same is true at $t=0$.
See also \cite[Section 5.2.1]{malrieu2001logarithmic} for a derivation of \eqref{def:intro:entropy-global} in the Lipschitz case.
From the well known subadditivity inequality  \cite[Lemma 3.9]{del2001genealogies},
\begin{align}
H(P^{(n,k)}_t\,|\,\mu^{\otimes k}_t) \le \frac{2k}{n} H(P^{(n,n)}_t\,|\,\mu^{\otimes n}_t), \label{def:subadditivity-H}
\end{align}
and Pinsker's inequality, one deduces the local estimates
\begin{align}
H(P^{(n,k)}_t \,|\, \mu^{\otimes k}_t) = O(k/n), \qquad \|P^{(n,k)}_t - \mu^{\otimes k}_t \|_{\mathrm{TV}} = O(\sqrt{k/n}), \label{eq:intro-H/TV-old}
\end{align}
which are of course worse than \eqref{eq:intro-entropy-new} and \eqref{eq:intro-TV-new}.

Notably, these quantitative estimates all reveal that the distance between $P^{(n,k)}_t$ and $\mu^{\otimes k}_t$ vanishes even when $k$ increases with $n$, as long as $k=o(n)$. This is sometimes called \emph{increasing propagation of chaos} \cite{benarous-zeitouni,del2001genealogies} or the \emph{size of chaos} \cite{paul2019size}. Further discussion of related literature on quantitative propagation of chaos is postponed to Section \ref{se:intro:literature}.

\subsection{Discussion of main results} \label{se:intro:mainresults}

The key ingredient in our method is an a priori functional inequality for the solution $\mu$ of the McKean-Vlasov equation \eqref{def:introMV}:
\begin{align}
\exists \transpconst < \infty \ \ \text{s.t.} \ \ \sup_{x \in \R^d}\big|\langle \nu-\mu_t,b(x,\cdot)\rangle\big|^2 \le \transpconst H(\nu\,|\,\mu_t), \quad \forall t \in (0,T), \ \nu \in \P(\R^d). \label{intro:transport}
\end{align}
This is equivalent (see Lemma \ref{le:integ-transp-equiv}) to the exponential integrability condition
\begin{align}
\exists c > 0 \ \text{ s.t. } \sup_{(t,x) \in (0,T) \times \R^d}\int_{\R^d} \exp\left( c|b(x,y)-\langle\mu_t, b(x,\cdot)\rangle|^2\right)\,\mu_t(dy) < \infty. \label{intro:expintegrability}
\end{align}
Our main Theorem \ref{th:intro-dynamic} imposes \eqref{intro:transport} as an assumption, and various  tractable sufficient conditions are given in Section \ref{se:intro-dyn}, including cases for which not even qualitative propagation of chaos was previously known (see Theorem \ref{th:sublinear}).

Kantorovich duality gives rise to an enlightening sufficient condition for \eqref{intro:transport}, which is to assume that there exists a metric $\rho$ on $\R^d$ for which the following two properties hold:
\begin{enumerate}[(i)]
\item There exists $\transpconst < \infty$ such that $\mu$ satisfies the transport inequality $\W_{1,\rho}^2(\cdot,\mu_t) \le \transpconst H(\cdot\,|\,\mu_t)$ for all $t$, where the 1-Wasserstein distance $\W_{1,\rho}$ is defined using $\rho$.
\item $b(x,\cdot)$ is Lipschitz with respect to $\rho$, uniformly in $x$.
\end{enumerate}
For coefficients $(b_0,b)$ which are Lipschitz with respect to the usual Euclidean metric  $\rho(x,y)=|x-y|$, it can be shown that $\mu$ satisfies a \emph{quadratic} transport inequality, which is stronger than (i). This leads to Corollary \ref{co:Lipschitz-dyn}.
Moreover, the $k$-tensorized form of this quadratic transport inequality combines with \eqref{eq:intro-entropy-new} to yield 
\begin{align}
\W_2^2(P^{(n,k)}_t,\mu^{\otimes k}_t) = O((k/n)^2), \label{intro:rate:W2}
\end{align}
which improves \eqref{ineq:intro:W2}.
Alternatively, if $b$ is bounded, then it is Lipschitz with respect to the trivial metric $\rho(x,y)=1_{x \neq y}$. In this case. $\W_{1,\rho}$ becomes the total variation distance, and the transport inequality (i) is just Pinsker's inequality. See Corollary \ref{co:bounded}.

This new technique seems versatile enough to adapt to a variety of situations.
In Section \ref{se:rev} we illustrate similar estimates on the \emph{reversed} entropy $H(\mu^{\otimes k}_t\,|\,P^{(n,k)}_t)$, which admits a simpler proof but under more restrictive assumptions. Section \ref{se:intro-infrange} adapts the approach to  settings with more general interactions, not necessarily pairwise. We discuss in Remarks \ref{re:vlasov1} and \ref{re:vlasov2} how to handle \emph{kinetic} Fokker-Planck equations, with no additional effort.
In the companion paper \cite{lacker-static}, we develop similar optimal rates of local chaos for the invariant (Gibbs) measures of diffusions like \eqref{def:introSDE}.
It is challenging  but possible to adapt the method to obtain estimates which are \emph{uniform in time}, and will be the subject of future work.

The optimal rate in \eqref{eq:intro-entropy-new} is interesting from a theoretical standpoint, and it also has the concrete implication of leading to an optimal rate of convergence for the law of large numbers for linear statistics. Indeed, Pinsker's inequality and \eqref{eq:intro-entropy-new} imply $\|P^{(n,2)}_t - \mu^{\otimes 2}_t \|_{\mathrm{TV}} = O(1/n)$, and one can then easily show that, for bounded measurable functions $f$,
\begin{align*}
\int\bigg(\frac{1}{n}\sum_{i=1}^n f(x_i) - \langle \mu_t,f\rangle\bigg)^2 P^{(n,n)}_t(dx) &= O(1/n).
\end{align*}
The weaker estimate $\|P^{(n,2)}_t - \mu^{\otimes 2}_t \|_{\mathrm{TV}} = O(1/\sqrt{n})$ would merely yield $O(1/\sqrt{n})$ here.

\subsection{Outline of the proof} \label{se:ideas}

The starting point of the new technique is the (finite) BBGKY hierarchy. The marginal flow $(P^{(n,k)}_t)_{t \ge 0}$ can be characterized as the solution of its own Fokker-Planck equation, or alternatively as the law of a solution of a $k$-particle SDE system, for each $1 \le k \le n$. The drift $\widehat{b}^k=(\widehat{b}^k_1,\ldots,\widehat{b}^k_k)$ of this $k$-particle SDE is obtained as a conditional expectation,
\begin{align*}
\widehat{b}^k_i(t,x_1,\ldots,x_k) &= \E\bigg[b_0(X^{n,i}_t) +  \frac{1}{n-1}\sum_{j \neq i} b(X^{n,i}_t,X^{n,j}_t)\,\Big|\,X^{n,1}_t=x_1,\ldots,X^{n,k}_t=x_k\bigg] \\
	&= b_0(x_i) + \frac{1}{n-1}\sum_{j \le k, \, j \neq i}b(x_i,x_j) + \frac{n-k}{n-1} \langle P^{(k+1|k)}_{t,(x_1,\ldots,x_k)}, b(x_i,\cdot)\rangle,
\end{align*}
where $P^{(k+1|k)}_{t,x}$ denotes the conditional law of $X^{k+1}_t$ given $(X^1_t,\ldots,X^k_t)=x$, and the simplification in the last step owes to exchangeability. This is a \emph{hierarchy} in the sense that the equation for of $P^{(n,k)}_t$ depends on $P^{(n,k+1)}_t$ for each $k$.
The drift $\widehat{b}^k$ can be derived via PDE computations, or from a stochastic perspective using the \emph{mimicking theorem} for It\^o processes \cite{gyongy1986mimicking,brunick2013mimicking}. 

We then employ a well known relative entropy estimate for diffusions related to Girsanov's theorem, used in many prior studies \cite{malrieu2001logarithmic,jabin-wang-bounded,jabin-wang-W1inf,jabir2019rate,lim2020quantitative}  to estimate the \emph{global} quantity $H(P^{(n,n)}_t\,|\,\mu^{\otimes n}_t)$ but apparently not the \emph{local} quantities $H^k_t := H (P^{(n,k)}_t \,|\, \mu^{\otimes k}_t )$. The entropy estimate yields essentially
\begin{align*}
\frac{d}{dt}H^k_t &\le \frac12\sum_{i=1}^k \int_{(\R^d)^k} \left| \widehat{b}^k_i(t,x) - b_0(x_i) - \langle \mu_t,b(x_i,\cdot)\rangle\right|^2\,P^{(n,k)}_t(dx).
\end{align*}
(In fact, this  becomes equality when one works with measures on path space.)
Simple manipulations using exchangeability lead to 
\begin{align}
\frac{d}{dt}H^k_t &\le \sum_{i=1}^k \int_{(\R^d)^k} \bigg| \frac{1}{n-1}\sum_{j \le k, \, j \neq i} \big(b(x_i,x_j) - \langle \mu_t,b(x_i,\cdot)\rangle \big)\bigg|^2\,P^{(n,k)}_t(dx) \label{intro:term1} \\
	&\quad +   \sum_{i=1}^k \int_{(\R^d)^k} \left|\frac{n-k}{n-1}\langle P^{(k+1|k)}_{t,x} - \mu_t, b(x_i,\cdot)\rangle \right|^2 \,P^{(n,k)}_t(dx) \nonumber \\
	&\le \frac{k(k-1)^2}{(n-1)^2}M + k \int_{(\R^d)^k} \left|\langle P^{(k+1|k)}_{t,x} - \mu_t, b(x_1,\cdot)\rangle \right|^2 \, P^{(n,k)}_t(dx), \label{intro:term2}
\end{align}
where we make the assumption that
\begin{align*}
M := \sup_{n,t}\int_{(\R^d)^2}|b(x_1,x_2) - \langle\mu_t, b(x_1,\cdot)\rangle|^2\,P^{(n,2)}_t(dx_1,dx_2) < \infty.
\end{align*}
Using the key assumption \eqref{intro:transport}, the last term in \eqref{intro:term2} is bounded by
\begin{align*}
\transpconst k \int_{(\R^d)^k} H\big( P^{(k+1|k)}_{t,x} \,|\, \mu_t\big) \, P^{(n,k)}_t(dx) = \transpconst k \big(H^{k+1}_t - H^k_t\big),
\end{align*}
with the last step following from the chain rule for relative entropy. Combining the last two estimates yields the following differential inequality, which is the heart of the argument:
\begin{align}
\frac{d}{dt}H^k_t &\le \frac{k(k-1)^2}{(n-1)^2}M + \transpconst k \big(H^{k+1}_t - H^k_t\big). \label{idea:diffeq-entropy}
\end{align}
Applying Gronwall's inequality, this implies
\begin{align*}
H^k_t &\le e^{- \transpconst k t}H^k_0 + \int_0^t e^{-\transpconst k (t - s)} \left( \frac{k(k-1)^2}{(n-1)^2}M + \transpconst k  H^{k+1}_s \right) \, ds. 
\end{align*}
The rest of the proof is a careful iteration of this inequality in the spirit of Picard, bounding the final term with the crude global estimate $H^n_t \le H^n_0 + nMt$. The resulting estimate of $H^k_t$ involves nested exponential integrals, which require some care to estimate effectively.

Actually, this argument at first yields only the suboptimal $H^k_t=O(k^3/n^2)$, because of the suboptimal factor $k(k-1)^2$ in  the variance-like $M$ term. The $(k-1)^2$ comes from bounding both the diagonal and cross-terms crudely by $O(1)$ upon expanding the square in  \eqref{intro:term1}, ignoring any cancellations. However, once we know that $H^3_t=O(1/n^2)$, we can show that the cross-terms are each $O(1/n)$, and we ultimately improve the factor in front of $M$ to $O(k^2/n^2)$. Repeating the rest of the argument leads to the optimal $O(k^2/n^2)$. This is where there is an advantage in working with the \emph{reversed} entropy $H(\mu^{\otimes k}_t\,|\,P^{(n,k)}_t)$, discussed in Section \ref{se:rev}: The argument is mostly the same, except that the integrals are with respect to $\mu^{\otimes k}_t$ instead of $P^{(n,k)}_t$, and thus the cross-terms vanish immediately by independence.

\subsection{Related literature} \label{se:intro:literature}

Quantitative propagation of chaos has advanced significantly in recent years. The classical coupling approach of \cite{sznitman1991topics} discussed above still sees useful refinements, such as the more sophisticated \emph{reflection coupling} technique of \cite{durmus2020elementary} which yields a form of \eqref{ineq:intro:W2} which is uniform in time. Coupling methods have an advantage of easily accommodating higher-order interactions (even in the diffusion coefficients), which our approach does not yet handle as seamlessly.

Relative entropy methods  appear to be more recent  in the analysis of mean field limits.
As discussed above, prior work  \cite{jabin-wang-bounded,jabin-wang-W1inf,jabir2019rate,lacker2018strong,lim2020quantitative}
has focused on establishing \emph{global} entropy estimates like \eqref{def:intro:entropy-global}, with local estimates like \eqref{eq:intro-H/TV-old} deduced from subadditivity.
The main novelty of our approach is to develop \emph{local} relative entropy estimates, which allow us to improve \eqref{eq:intro-H/TV-old} to \eqref{eq:intro-entropy-new}. It is worth highlighting that the aforementioned papers all involve establishing some form of a priori exponential integrability similar in spirit to \eqref{intro:expintegrability}, some of which also incorporate a tradeoff between integrability of the interaction $b$ and of derivatives of the solution $\mu_t$ of the McKean-Vlasov equation.
Most notably, the breakthrough work \cite{jabin-wang-W1inf} deepened the global entropy approach to cover a broad class of singular interactions appearing in important physical models, using an exponential integrability result that takes full advantage of the smoothing effects of the noise. See \cite{bresch2019mean} for further developments of these ideas, as well as \cite{jabir2018mean,tomasevic2020propagation} which use an interesting \emph{local} Girsanov argument to prove \emph{qualitative} propagation of chaos for certain singular interactions.
All but the most mild singularities are apparently out of reach of our main results because of the nature of our exponential integrability condition \eqref{intro:expintegrability}.
It is natural to wonder, though, if our general strategy of comparing the $k$-particle ``energy" to the $(k+1)$-particle ``energy" via the BBGKY hierarchy could be adapted for different kinds of dynamics and perhaps better-tailored notions of ``energies," rather than entropy, such as the modulated (free) energy used in \cite{serfaty2020mean,duerinckx2016mean,bresch2019mean}.

Yet another powerful approach to quantitative propagation of chaos emerged recently, based on semigroup analysis, stemming from the important recent work \cite{mischler2013kac} (see also \cite{kolokoltsov2010nonlinear} for related ideas). The rough idea is to compare the evolutions of the $n$-particle empirical measure $(L^n_t)_{t \ge 0}$ and its candidate limit $(\mu_t)_{t \ge 0}$ using stability properties of the associated (nonlinear) semigroups, which act on functionals of $\P(\R^d)$.
This approach is extremely versatile, applying to a wide variety of dynamics.
Closest to our diffusion setting is \cite[Section 6]{mischler2015new}, which gives rates of convergence for quantities like $\langle P^{(n,k)}_t-\mu^{\otimes k}_t, f\rangle$ for sufficiently smooth test functions $f$. However, the estimate is limited by the mean rate of convergence of i.i.d.\ empirical measures in Wasserstein distance, which is well known (see \cite{fournier2015rate}) to deteriorate with the dimension $d$.

Using arguments which are similar in spirit to the semigroup approach but more tailored to the diffusive setting, a few recent papers have identified the optimal $O(1/n)$ rate of so-called \emph{weak error} in the case that $k$ is fixed as $n\to\infty$.
For sufficiently smooth coefficients, it is shown in \cite[Theorem 1.2]{bencheikh2019bias} that $|\langle P^{(n,1)}_t-\mu_t, f\rangle|=O(1/n)$ for $C^2_b$ test functions $f$.
For Lipschitz coefficients, this is extended in \cite[Theorem 2.17]{chassagneux2019weak} to $|\E[F(L^n_t)]-F(\mu_t)| = O(1/n)$ for nonlinear but sufficiently smooth functionals $F : \P(\R^d) \to \R$, by taking advantage of an appropriate differential calculus on $\P(\R^d)$;
the very recent paper \cite{delarue-tse} deepens these ideas to derive estimates which are uniform in time.
Closest to our setting, with merely H\"older-continuous coefficients, \cite[Theorem 3.5]{frikha2019backward} obtains $\|P^{(n,1)}_t - \mu_t\|_{\mathrm{TV}} = O(1/n)$, in fact with pointwise estimates between the densities, by exploiting the regularizing effect of noise on the backward Kolmogorov equation associated with $\mu$.
Compared to these works, the present paper has the advantages of obtaining the optimal rate of chaos not only in $n$ but also in $k$, using the stronger ``distance" relative entropy,  applying to more irregular coefficients (except that there can be no interaction in the diffusion coefficient), and being arguably simpler to implement.

\subsubsection{From global to local}
Aside from subadditivity estimates like \eqref{def:subadditivity-W} and \eqref{def:subadditivity-H}, there is an alternative approach to deducing local estimates from global ones employed in \cite{mischler2013kac,hauray2014kac,mischler2015new}, which also appears to fall inevitably short of our optimal local estimates. The idea, instead of comparing $P^{(n,k)}_t$ directly to $\mu^{\otimes k}_t$, is to instead compare  $P^{(n,k)}_t$ to the measure $\widetilde{P}^{(n,k)}_t$ defined as the mean of the $k$-fold product of the empirical measure:
\begin{align}
\widetilde{P}^{(n,k)}_t := \int_{(\R^d)^n} L_n^{\otimes k} \,dP^{(n,n)}_t, \qquad L_n(x) := \frac{1}{n}\sum_{i=1}^n\delta_{x_i}. \label{intro:Ln}
\end{align}
One may interpret $\widetilde{P}^{(n,k)}_t$ (resp.\  $P^{(n,k)}_t$) as the joint law of $k$ particles drawn at random from the $n$ \emph{with} (resp.\ \emph{without}) replacement.
A combinatorial argument due to Diaconis-Freedman \cite{diaconis1980finite} shows that $\|P^{(n,k)}_t-\widetilde{P}^{(n,k)}_t\|_{\mathrm{TV}} \le k(k-1)/n$, which is sharp in general but fails to achieve our desired result of $O(k/n)$.
Our result \eqref{eq:intro-TV-new} thus suggests that $P^{(n,k)}_t$ is closer to $\mu^{\otimes k}_t$ than it is to $\widetilde{P}^{(n,k)}_t$, at least for large $k$.
More importantly, applying the triangle inequality, one must still estimate the distance between $\widetilde{P}^{(n,k)}_t$ and $\mu^{\otimes k}_t$, which is even more costly. The natural approach of \cite[Theorem 2.4]{hauray2014kac} works with the first-order Wasserstein distance $\W_1$, using convexity of $\W_1(\cdot,\mu_t)$ and a tensorization identity:
\begin{align*}
\W_1(\widetilde{P}^{(n,k)}_t,\mu^{\otimes k}_t) &\le \int_{(\R^d)^n} \W_1(L_n^{\otimes k},\mu^{\otimes k}_t)\,dP^{(n,n)}_t = k \int_{(\R^d)^n} \W_1(L_n,\mu_t)\,dP^{(n,n)}_t.
\end{align*}
The right-hand side in general is no better than $O(k/n^{1/(d \vee 2)})$ by \cite{fournier2015rate}, which is worse than our estimate of $O(k/n)$ in \eqref{intro:rate:W2}.
Similar ideas appear in \cite{mischler2013kac,mischler2015new}, with an additional approximation between the generators of the $n$-particle and limiting dynamics.

It is worth emphasizing here that our main result \eqref{eq:intro-entropy-new} is \emph{dimension-free}, in the sense that the dimension of the underlying state space $\R^d$ does not appear in the exponent, but only potentially in the constant hidden by the big-O notation $O((k/n)^2)$. The hidden constant is typically $O(d)$, as can be seen in Corollaries \ref{co:bounded} and \ref{co:Lipschitz-dyn}.
Global estimates of the expected rate of convergence of empirical measures invariably deteriorate with $d$ as mentioned in the previous paragraph (unless perhaps one is more careful about the choice of metric \cite{han2021class}), and for this reason it seems unlikely that our optimal local estimates could be deduced from estimates on empirical measures.

\subsubsection{The BBGKY hierarchy}
The BBGKY hierarchy, described in Section \ref{se:ideas} and Remark \ref{re:BBGKY} below, has been the basis for \emph{qualitative} studies of propagation of chaos, in certain contexts, but apparently not many \emph{quantitative} studies.
It does not appear to be well-documented for diffusive models, but we refer to \cite[Section 1.5]{golse2016dynamics} and \cite{golse2013empirical} for thorough and lucid introduction to the noiseless setting. 
Using the BBGKY hierarchy for even a qualitative proof of propagation of chaos is no trivial task; see \cite{spohn1981vlasov} for the noiseless setting and \cite{lanford1975time} for Lanford's celebrated derivation of the Boltzmann equation.
Beyond this paragraph, we will make no attempt to summarize the extensive literature on propagation of chaos for Boltzmann-type dynamics, wherein hierarchical methods feature more prominently.
We highlight the recent works \cite{pulvirenti2017boltzmann,paul2019size}, which derive quantitative local propagation of chaos using a hierarchical approach based on the so-called \emph{correlation errors}, also known as \emph{v-functions} \cite{demasi2006mathematical}, which are the signed measures defined by $E^k_t := \int  (L_n -\mu_t)^{\otimes k} \,dP^{(n,n)}_t$ with $L_n$ as in \eqref{intro:Ln}. 
This approach bears a philosophical resemblance to ours, in the sense that both involve iterating a BBGKY-type hierarchy to estimate a quantity which controls the $k$-particle correlations.

\subsection{Outline of the paper}

Section \ref{se:mainresults} presents all of the main results in detail. Section \ref{se:gaussian} gives an explicit Gaussian example, which shows that the rate $(k/n)^2$ is optimal. The remaining Sections \ref{se:proofs-dyn}--\ref{se:proofs-inf} and appendices give the proofs.

\section{Main results} \label{se:mainresults}

This section states the general results of the paper in full detail. We begin in Section \ref{se:intro-dyn} with pairwise interactions, and Section \ref{se:intro-infrange} extends the ideas to more general interactions.

\subsection{Notation}
Recall that a probability measure on a product space $E^n$ is said to be \emph{exchangeable} if it is invariant with respect to permutations of the $n$ coordinates.
We use the usual notation $\langle \nu,f\rangle = \int f\,d\nu$ for integration of a $\nu$-integrable function $f$ with respect to a measure $\nu$. 
For a real-valued function $f$ on any set $E$ we write $\|f\|_\infty := \sup_{x \in E}|f(x)|$.

\subsubsection{Distances of probability measures}
The  relative entropy between two probability measures $(\nu,\nu')$ on a common measurable space is defined as usual by
\begin{align*}
H(\nu\,|\,\nu') = \int \frac{d\nu}{d\nu'}\log\frac{d\nu}{d\nu'}\,d\nu', \ \ \text{ if } \nu \ll \nu', \ \text{ and } \ \  H(\nu\,|\,\nu')= \infty \ \ \text{otherwise}.
\end{align*}
For a metric space $(E,\rho)$ and $q \ge 1$, define the Wasserstein distance $\W_q$ by
\begin{align*}
\W_{q,\rho}^q(\nu,\nu') := \inf_\pi \int_{E \times E} \rho^q(x,y)\,\pi(dx,dy) ,
\end{align*}
where the infimum is over all $\pi \in \P(E \times E)$ with marginals $\nu$ and $\nu'$.
We typically abbreviate $\W_q=\W_{q,\rho}$, as the metric will be understood from context.
Note that there is no normalization by the dimension in our definitions of $H$ and $\W_p$, in contrast to \cite{jabin-wang-bounded,jabin-wang-W1inf,lim2020quantitative,hauray2014kac}.

\subsubsection{Path space} \label{se:pathspacenotation}
We fix throughout the paper a dimension $d \in \N$ and time horizon $T > 0$. 
We will allow time-dependent and path-dependent coefficients, which does not complicate the arguments or weaken the results but requires some additional notation. 
For $t \ge 0$  let $\C_t^d:=C([0,t];\R^d)$ denote the continuous path space equipped with the uniform norm
\begin{align*}
\|x\|_t := \sup_{s \in [0,t]}|x_s|.
\end{align*}
We also view $\|\cdot\|_t$ as a seminorm on $\C^d_T$ for $t < T$.
Identify $\C_0^d$ with $\R^d$.
Here $|\cdot|$ always denote the Euclidean ($\ell_2$) norm on $\R^k$ for any $k$, or on $(\R^d)^k\cong \R^{dk}$. Similarly, $(\C_T^d)^k$ is equipped with the sup-norm induced by the identification $(\C_T^d)^k\cong \C_T^{dk}$.
For any $k \in \N$, $Q \in \P(\C_T^k)$, and $t \in [0,T]$,  denote by   $Q_t \in \P(\R^k)$ denote the time-$t$ marginal, i.e., the pushforward by the evaluation map $x \mapsto x_t$.

For $k \in \N$ and a topological space $E$, we say that a function $\overline{b} : [0,T] \times \C_T^k \to E$ is \emph{progressively measurable} if it is Borel measurable and non-anticipative, in the sense that $\overline{b}(t,x)=\overline{b}(t,x')$ for every $t \in [0,T]$ and $x,x' \in C_T^k$ satisfying $x|_{[0,t]}=x'|_{[0,t]}$.
For a function $\overline{b} : [0,T] \times \C_T^d  \to \R^d$ we define $\overline{b}^{\otimes k} : [0,T] \times (\C_T^d)^k \to (\R^d)^k$ by
\begin{align}
\overline{b}^{\otimes k}(t,x^1,\ldots,x^k) = (\overline{b}(t,x^1),\ldots,\overline{b}(t,x^k)). \label{def:tensor}
\end{align}

\subsection{Pairwise interactions} \label{se:intro-dyn}

For the rest of this section, fix progressively measurable functions $b_0 : [0,T] \times \C_T^d \to \R^d$ and $b : [0,T] \times \C_T^d \times \C_T^d \to \R^d$.
Consider the SDE system
\begin{align}
dX^i_t = \bigg(b_0(t,X^i) + \frac{1}{n-1}\sum_{j \neq i}b(t,X^i,X^j)\bigg)dt +  dW^i_t, \quad i=1,\ldots,n, \label{def:introSDE-nonMarkov}
\end{align}
where $W^1,\ldots,W^n$ are independent $d$-dimensional Brownian motions.
In stating that $(X^1,\ldots,X^n)$ is a solution of the SDE system \eqref{def:introSDE-nonMarkov}, we implicitly require that 
\begin{align*}
 \int_0^T \bigg(|b_0(t,X^i)| + \frac{1}{n-1}\sum_{j \neq i}|b(t,X^i,X^j)|\bigg) dt < \infty, \ \ \ a.s., \text{ for } i=1,\ldots,n.
\end{align*}
which ensures that the equation \eqref{def:introSDE-nonMarkov} makes sense.
We write $P^{(n)} \in \P((\C_T^d)^n)$ for the law of a weak solution $(X^1,\ldots,X^n)$ of \eqref{def:introSDE-nonMarkov}, and let $P^{(k)} \in \P((\C_T^d)^k)$ denote the marginal law of the first $k$ coordinates.
As the results in this section are non-asymptotic, we consider the size $n$ of the particle system as fixed, and we lighten the notation compared to the introduction by writing simply $P^{(k)}$ instead of $P^{(n,k)}$, and $X^k$ in place of $X^{n,k}$.

Consider a weak solution of the McKean-Vlasov equation
\begin{align}
dX_t = \left(b_0(t,X) + \langle \mu, b(t,X,\cdot)\rangle\right)dt +  dW_t, \quad \mu = \mathrm{Law}(X), \label{def:introMV-nonMarkov}
\end{align}
where $W$ is a $d$-dimensional Brownian motion. We
implicitly require that a solution satisfies
\begin{align*}
 \int_0^T \big(|b_0(t,X)| + \langle \mu, |b(t,X,\cdot)|\rangle \big) dt < \infty, \ \ \ a.s.,
\end{align*}
which ensures that the equation \eqref{def:introMV-nonMarkov} makes sense.

Our main theorems assume the existence of weak solutions of the equations \eqref{def:introSDE-nonMarkov} or \eqref{def:introMV-nonMarkov}, with the law $P^{(n)}$ of the former being exchangeable. Note that exchangeability of $P^{(n)}$  is automatic if  the SDE is unique in law and the initial distribution is exchangeable.
We do not need uniqueness for \eqref{def:introSDE-nonMarkov} or \eqref{def:introMV-nonMarkov}, but we do require that the \emph{linearized} version of \eqref{def:introMV-nonMarkov} to well-posed in a somewhat specialized sense in order to justify our entropy estimate in Lemma \ref{le:entropy-diffusions}; see Remarks \ref{re:assumptions} and \ref{re:wellposedness} for further discussion.

\begin{definition} \label{def:SDEwellposed}
Let $\ell \in \N$, and let $\overline{b} : [0,T] \times \C_T^\ell \to \R^\ell$ be progressively measurable. We say that the \emph{SDE$(\overline{b})$ is well-posed} if, for every $(t,z) \in [0,T) \times \C_T^\ell$, there exists  a unique in law weak solution of the SDE
\begin{align*}
dX_s &= \overline{b}(s,X)\,ds + dW_s, \text{ for } s \in (t,T], \quad \text{ and } \quad X_s = z_s, \text{ for } s \in [0,t],
\end{align*}
where $W$ is an $\ell$-dimensional Brownian motion.
\end{definition}

The following is the main general result of the paper.
Let $x_+=\max(x,0)$ denote the positive part of a real number $x$, and recall the notation $\overline{b}^{\otimes k}$ from \eqref{def:tensor}.

\begin{theorem} \label{th:intro-dynamic}
Assume there exists a weak solution of the SDE  \eqref{def:introSDE-nonMarkov} whose law $P^{(n)} \in \P((\C_T^d)^n)$ is exchangeable, and assume there exists a weak solution of the McKean-Vlasov equation \eqref{def:introMV-nonMarkov} with law $\mu \in \P(\C_T^d)$.
Assume the following:
\begin{enumerate}[(1)]
\item Well-posedness of the linearized McKean-Vlasov equation: For each $k =1,\ldots,n$,  SDE$(\overline{b}^{\otimes k})$ is well-posed in the sense of Definition \ref{def:SDEwellposed}, where $\overline{b}(t,x) := b_0(t,x) + \langle \mu,b(t,x,\cdot)\rangle$.
\item Square-integrability:
\begin{align}
M := \mathrm{ess\,sup}_{t \in [0,T]}\int_{\C_T^d \times \C_T^d} \left|  b(t,x,y) - \langle \mu,b(t,x,\cdot)\rangle \right|^2 \, P^{(2)}(dx,dy) < \infty. \label{asmp:dyn-moment}
\end{align}
\item Transport-type inequality: There exists $0 < \transpconst < \infty$ such that
\begin{align}
\begin{split}
&|\langle \mu - \nu, b(t,x,\cdot)\rangle|^2 \le  \transpconst H(\nu\,|\,\mu), \\
\text{for all }& t \in (0,T), \ x \in \C_T^d, \ \nu \in \P(\C_T^d) \text{ s.t. } b(t,x,\cdot) \in L^1(\nu).
\end{split} \label{asmp:dyn-transp}
\end{align}
\item Chaotic initial conditions: There exists $C_0 < \infty$ such that
\begin{align}
H(P^{(k)}_0\,|\,\mu^{\otimes k}_0) &\le C_0 k^2/n^2, \quad \text{for } k=1,\ldots,n. \label{asmp:dyn-init}
\end{align}
\end{enumerate}
If $n \ge 6 e^{\transpconst T }$, then for integers $k \in [1,n]$ we have 
\begin{align}
H(P^{(k)} \,|\,\mu^{\otimes k}) &\le 2C\frac{k^2}{n^2} + C\exp\Big(-2n\Big(e^{-\transpconst T} - \frac{k}{n}\Big)_+^2\Big), \label{eq:dyn-result}
\end{align}
where the constant $C$ is defined by
\begin{align}
C := 8\left( C_0 + (1+ \transpconst ) M T  \right)e^{6 \transpconst  T  } . \label{def:dyn-constC}
\end{align}
\end{theorem}

The proof is given in Section \ref{se:proofs-dyn}.
The bound \eqref{eq:dyn-result} is non-asymptotic but easily yields the asymptotic bound of $O((k/n)^2)$ when $n\to\infty$ with $k=o(n)$, as long as the stated assumptions hold with constants $(M,\transpconst,C_0)$ not depending on $n$.
The constant $C$ is not optimal, but the rate $(k/n)^2$ is, as Example \ref{ex:gaussian-dyn} will show.

Before stating corollaries, we first discuss the assumptions of Theorem \ref{th:intro-dynamic}. 
The most notable assumption in Theorem \ref{th:intro-dynamic} is the inequality \eqref{asmp:dyn-transp}. It is equivalent to a square-exponential integrability:

\begin{lemma} \label{le:integ-transp-equiv}
For any $\mu \in \P(\C_T^d)$ satisfying $\langle \mu, |b(t,x,\cdot)|\rangle < \infty$ for all $(t,x) \in [0,T] \times \C_T^d$, the following are equivalent:
\begin{enumerate}[(a)]
\item There exists $\transpconst  > 0$ such that \eqref{asmp:dyn-transp} holds. 
\item There exist $\kappa,R > 0$ such that
\begin{align*}
\sup_{(t,x) \in (0,T) \times \C_T^d} \log\int_{\C_T^d} \exp\left( \kappa |b(t,x,y) - \langle \mu,b(t,x,\cdot)\rangle|^2\right) \, \mu(dy) \le R.
\end{align*}
\end{enumerate}
If (b) holds, then (a) holds with $\transpconst = 2(1+R)/\kappa$.
\end{lemma}

The proof is given in Appendix \ref{ap:transp-proof}.  This is similar in spirit to the well known integral criteria for transport inequalities; see \cite[Section 6]{gozlan-leonard} and references therein. The implication (b) $\Rightarrow$ (a) is a consequence of the weighted Pinsker inequality of  \cite{bolleyvillani}. The constants $\kappa$ and $R$ in (b) could be estimated explicitly from the constant $\transpconst$ in (a), though not as simply, and this direction is not as important for our purposes.
We comment further on the assumptions of Theorem \ref{th:intro-dynamic}, before giving concrete examples below.

\begin{remark} \label{re:assumptions}
{\ } 
\begin{enumerate}[(i)]
\item Suppose the coefficients are Markovian, in the sense that $b_0(t,x)=\widetilde{b}_0(t,x_t)$ and  $b(t,x,y)=\widetilde{b}_0(t,x_t,y_t)$ for functions $\widetilde{b}_0$ and $\widetilde{b}$ defined on $[0,T] \times \R^d$ and $[0,T] \times \R^d \times \R^d$, respectively. Then a sufficient condition for (3) is that  there exists $\transpconst$ such that 
\begin{align*}
|\langle \mu_t - \nu, \widetilde{b}(t,x,\cdot)\rangle|^2 \le  \transpconst H(\nu\,|\,\mu_t), \ \ \forall t \in (0,T), \ x \in \R^d, \ \nu \in \P(\R^d).
\end{align*}
Indeed, the inequality \eqref{asmp:dyn-transp} follows from $H(\nu_t\,|\,\mu_t) \le H(\nu\,|\,\mu)$, valid for $\nu \in \P(\C_T^d)$.
\item The assumption (4) of chaotic initial conditions is natural if we are to expect \eqref{eq:dyn-result}. It is automatic if the initial distribution is i.i.d.\ in the sense that $P^{(n)}_0=\mu^{\otimes n}_0$. It also holds for a broad class of Gibbs measures, as shown in the companion paper \cite{lacker-static}. For instance, suppose $P^{(n)}_0 \in \P((\R^d)^n)$ is of the form
\begin{align*}
P^{(n)}_0(dx) = \frac{1}{Z_n}\exp\bigg( -  \sum_{i=1}^nU(x_i) - \frac{1}{ n-1 }\sum_{1 \le i < j \le n}V(x_i-x_j)  \bigg)dx, \ \ \ Z_n > 0,
\end{align*}
where $U$ and $V$ are convex even functions with $\nabla^2U \ge \kappa I$ and $0 \le \nabla^2 V \le LI$ in semidefinite order for some constants $\kappa, L > 0$. Then, if $L < \kappa$, there exists a unique $\mu_0 \in \P(\R^d)$ satisfying
\begin{align*}
\mu_0(dx) &= \frac{1}{Z}\exp\big( -  U(x) - V *\mu_0(x) \big)dx, \ \ \ Z > 0,
\end{align*}
and it holds that $H(P^{(k)}_0\,|\,\mu^{\otimes k}_0) \le C_0(k/n)^2$ for all $1\le k \le n$, where the constant $C_0 > 0$ depends only on $(L,\kappa,d)$ and not on $(n,k)$; see \cite[Corollary 2.7]{lacker-static}.
\item There is no assumption on $b_0$, other than progressive measurability and implicit restrictions imposed by the hypothesis that the SDEs \eqref{def:introSDE-nonMarkov} and \eqref{def:introMV-nonMarkov} admit solutions.
 Conditions for well-posedness and square-integrability of McKean-Vlasov equations of the form \eqref{def:introMV-nonMarkov}, and especially for standard SDEs of the form \eqref{def:introSDE-nonMarkov}, are well known and plentiful in the literature. This is why we prefer to state the assumptions (1) and (2) of Theorems \ref{th:intro-dynamic} in abstract terms. 
\item The assumptions (1) and (3) actually imply a form of uniqueness of the McKean-Vlasov equation, though we do not need this. See Remark \ref{re:MVuniqueness}.
\item Even if $P^{(n)}$ is not exchangeable, one could still apply Theorem \ref{th:intro-dynamic} to the symmetrized measure obtained by randomly permuting the coordinates.
\item The well-posedness assumption (1) is somewhat awkward. One might suspect that the well-posedness for general $k \in \N$ would follow from the well-posedness for the $k=1$ case. But this is not obvious, unless we assume pathwise uniqueness for $k=1$ instead of uniqueness in law.
See Remark \ref{re:wellposedness} for further discussion of relaxing (1).
\end{enumerate}
\end{remark}

The rest of this section presents corollaries of Theorem \ref{th:intro-dynamic} under more concrete assumptions, which are proven in Section \ref{se:proofs:corollaries}. The first two corollaries descend from a general principle mentioned in the introduction:

\begin{example}[Transport inequality] \label{ex:transport}
Let $L,\transpconst' < \infty$.
Suppose for each $t \in [0,T]$ that $\mu$ satisfies the transport inequality $\W_{1,\rho}^2(\cdot,\mu) \le \transpconst' H(\cdot\,|\,\mu)$, where $\W_{1,\rho}$ is defined relative to some Borel measurable pseudo-metric $\rho: \C_T^d \times \C_T^d \to [0,\infty)$ with respect to which $b(t,x,\cdot)$ is $L$-Lipschitz for all $(t,x) \in (0,T) \times \C_T^d$. Then \eqref{asmp:dyn-transp} holds with $\transpconst := L^2\transpconst'$.
\end{example}

The case of bounded $b$ fits into Example \ref{ex:transport}, using the discrete metric, for which the transport inequality reduces to Pinsker's inequality. 

\begin{corollary} \label{co:bounded}
Suppose $b$ is uniformly bounded.
Assume the SDE$(b_0^{\otimes k})$ is well-posed in the sense of Definition \ref{def:SDEwellposed}, for each $k =1,\ldots,n$.
Then the SDE \eqref{def:introSDE-nonMarkov} and the  McKean-Vlasov SDE \eqref{def:introMV-nonMarkov} admit unique in law weak solutions from any initial distribution.
Let $P^{(n)}$ and $\mu$, respectively, denote solutions whose initial distributions satisfy condition (4) of Theorem \ref{th:intro-dynamic}, with $P^{(n)}_0$ exchangeable.
Then the conclusion \eqref{eq:dyn-result} of Theorem \ref{th:intro-dynamic} holds, with $\transpconst := 2\||b|^2\|_\infty$ and with
\begin{align}
C \le 8\left( C_0 + (1+ \transpconst ) \transpconst T  \right)e^{6 \transpconst  T  } . \label{def:dyn-constC-bd}
\end{align}
\end{corollary}

The case of Lipschitz coefficients $(b_0,b)$ fits into Example \ref{ex:transport}, using the usual Euclidean metric. The associated transport inequality requires justification, though, especially to determine the constant $\transpconst$.
Recall that $(\C_t^d)^k \cong \C_t^{dk}$ is equipped with the supremum norm $\|\cdot\|_t$ associated with the usual Euclidean norm on $\R^{dk}$, which informs the definitions of the Wasserstein distances appearing in the following.

\begin{corollary} \label{co:Lipschitz-dyn}
Suppose $(b_0,b)$ is Lipschitz, in the sense that there exists $L < \infty$ such that
\begin{align}
|b_0(t,x)-b_0(t,x') + b(t,x,y)-b(t,x',y')| \le L(\|x-x'\|_t + \|y-y'\|_t), \label{asmp:Lipschitz-dyn}
\end{align}
for all $t \in [0,T]$ and $x,x',y,y' \in \C^d_T$. Assume also that
\begin{align}
M_0 := \int_0^T|b_0(t,0) + b(t,0,0)|\, dt < \infty. \label{asmp:Lipschitz-dyn2}
\end{align}
Then the SDE \eqref{def:introSDE-nonMarkov} admits a unique in law weak solution from any initial distribution, and the  McKean-Vlasov SDE \eqref{def:introMV-nonMarkov} admits a unique in law weak solution from starting from any initial distribution with finite first moment.
Let $P^{(n)}$ and $\mu$, respectively, denote solutions whose initial distributions satisfy condition (4) of Theorem \ref{th:intro-dynamic}, with $P^{(n)}_0$ exchangeable. Suppose also that $\mu_0$ has finite second moment and satisfies the quadratic transport inequality
\begin{align}
\W_2^2(\nu,\mu_0) \le \eta_0 H(\nu\,|\,\mu_0), \qquad \forall \nu \in \P(\R^d), \label{def:dyn-init-T2H}
\end{align}
for some $0 < \eta_0 < \infty$.
If $n \ge 6 e^{\transpconst T}$, then for integers $k \in [1,n]$ we have 
\begin{align}
\frac{L^2}{\transpconst}\W_2^2(P^{(k)}, \mu^{\otimes k}) &\le H(P^{(k)} \,|\,\mu^{\otimes k}) \label{ineq:Lipschitz} \\
	&\le 2C\frac{k^2}{n^2} +  C\exp\bigg(-2n\Big(e^{-\transpconst T} - \frac{k}{n}\Big)_+^2\bigg). \nonumber
\end{align}
where $C$ is as in \eqref{def:dyn-constC}, but with the constants $(M,\transpconst)$ defined by
\begin{align*}
M &= 8L^2e^{16TL^2}\left( \int_{\R^d} |x|^2\,\mu_0(dx) + \int_{\R^d} |x|^2\,P^{(1)}_0(dx) + 2M_0^2 + 8dT\right), \\
\transpconst &= 3L^2(\eta_0 \vee 2T)e^{3TL^2}.
\end{align*}
\end{corollary}

The proof, given in Section \ref{se:proofs:Lipschitz}, uses a known method based on Girsanov's theorem \cite{djellout2004transportation,ustunel2012transportation,pal2012concentration} to deduce that $\mu$ satisfies a quadratic transport inequality  with respect to the sup-norm; see also \cite{bartl2020functional,bahlali2020quadratic} for recent results on quadratic transport inequalities for SDEs with irregular coefficients, which may be useful in deriving additional corollaries.
Interestingly, the (tensorized) quadratic transport inequality is useful not only for checking assumption (3) of Theorem \ref{th:intro-dynamic}, but also for transferring the main estimate from entropy to Wasserstein distance, yielding the first inequality \eqref{ineq:Lipschitz}. 
It is natural to wonder if the same rate $\W_2^2(P^{(k)}, \mu^{\otimes k}) = O((k/n)^2)$ can be proven without assuming entropic chaos \eqref{asmp:dyn-init} for the initial condition, but rather its analogue in transport distance.
This is an open problem.

\begin{example}[Gaussian case] \label{ex:gaussian-dyn}
Suppose $b_0(t,x)=-ax_t$ and $b(t,x,y)=-by_t$, where $a,b > 0$ are constants. 
Consider the initial conditions $P^{(n)}_0=\delta_0$ and $\mu_0 = \delta_0$.
This clearly fits the assumptions of Corollary \ref{co:Lipschitz-dyn}.
In Section \ref{se:gaussian} we show that, if $n\to\infty$ and $k/n\to 0$, then 
\begin{align}
\liminf_{n\to\infty} \, (n/k)^2 \W_2^2(P^{(k)} \,|\,\mu^{\otimes k}) > 0, \label{ex:Gaussian-claim}
\end{align}
for each $t > 0$.
Thus, the rate $(k/n)^2$ obtained in Corollary \ref{co:Lipschitz-dyn} and Theorem \ref{th:intro-dynamic} cannot be improved in general, even when the initial positions are i.i.d.
\end{example}

\begin{remark} \label{re:transport-rate}
In Corollary \ref{co:Lipschitz-dyn}, the estimate \eqref{ineq:Lipschitz} comes from the \emph{dimension-free} tensorization property of the quadratic transport inequality.
However, even in the setting of a (weaker) first-order transport inequality as in Example \ref{ex:transport}, a dimension-dependent tensorization is possible: From the metric $\rho$ in Example \ref{ex:transport} we can define the corresponding $\ell_1$-metric on $(\C_T^d)^k$ for each $k \in \N$ by
\begin{align*}
\rho_{1,k}\big((x^1,\ldots,x^k),(y^1,\ldots,y^k)\big) := \sum_{i=1}^k \rho(x^i,y^i).
\end{align*}
Define the Wasserstein distance $\W_{1,\rho_{1,k}}$ on $\P((\C_T^d)^k)$ using this metric. Then, it is well known \cite[Proposition 1.9]{gozlan-leonard} that the inequality $\W_{1,\rho}^2(\cdot,\mu) \le \transpconst' H(\cdot\,|\,\mu)$ implies $\W_{1,\rho_{1,k}}^2(\cdot,\mu^{\otimes k}) \le k\transpconst' H(\cdot\,|\,\mu^{\otimes k})$ for each $k$.
In this case, the conclusion $H(P^{(k)} \,|\,\mu^{\otimes k})=O((k/n)^2)$ of Theorem \ref{th:intro-dynamic} along with this transport inequality yield $\W_{1,\rho_{1,k}}(P^{(k)},\mu^{\otimes k}) = O(k^{3/2}/n)$. 
When $\rho$ is the discrete metric, for instance, $\rho_{1,k}$ is the so-called Hamming metric.
When $\rho$ is the sup-norm, the $\ell_1$-norm $\rho_{1,k}$ bounds from above our default $\ell_2$-sup-norm on $(\C_T^d)^k\cong \C_T^{dk}$, so we may write $\W_{1,\rho_{1,k}} \ge \W_1$.
\end{remark}

As a final corollary of Theorem \ref{th:intro-dynamic}, we discuss a class of discontinuous linear growth interactions for which even \emph{qualitative} propagation of chaos was previously unknown. In fact, even existence and uniqueness for the McKean-Vlasov equation are new here. Because of this we label the following a \emph{theorem} rather than a \emph{corollary}, though the propagation chaos part follows quickly from Theorem \ref{th:intro-dynamic}.

\begin{theorem} \label{th:sublinear}
Suppose $b_0$ and $b$ satisfy, for some $0 < K < \infty$,
\begin{align}
\begin{split}
|b_0(t,x)| + |b(t,x,y)| &\le K\left(1 +  \|x\|_t +  \|y\|_t\right)  \\
|b(t,x,y)-b(t,x,y')| &\le K\left(1 +  \|y\|_t +  \|y'\|_t\right)
\end{split} \Bigg\}
\qquad \forall t \in [0,T], \ x,y,y' \in \C_T^d. \label{condition:lingrowth}
\end{align}
Then the SDE \eqref{def:introSDE-nonMarkov} admits a unique in law weak solution from any initial distribution, and the  McKean-Vlasov SDE \eqref{def:introMV-nonMarkov} admits a unique in law weak solution from any initial distribution $\mu_0 \in \P(\R^d)$ satisfying
\begin{align}
R := \int_{\R^d} e^{\kappa_0 |x|^2}\,\mu_0(dx) < \infty, \ \text{ for some } \kappa_0 > 0. \label{def:dyn-expinteg-time0}
\end{align}
Let $P^{(n)}$ and $\mu$, respectively, denote solutions whose initial distributions satisfy condition (4) of Theorem \ref{th:intro-dynamic}, with $P^{(n)}_0$ exchangeable and with $\mu_0$ satisfying \eqref{def:dyn-expinteg-time0}.
Then there exist constants $0 < C,\transpconst < \infty$, depending only on $(C_0,K,c,\kappa_0,R,T,d)$ and the second moments of $P^{(1)}_0$ and $\mu_0$ (but not on $n$ or $k$), such that the following holds: If $n \ge 6 e^{\transpconst T}$, then for integers $k \in [1,n]$ we have 
\begin{align*}
\frac{1}{Ck}\W_1^2(P^{(k)} ,\mu^{\otimes k}) \le H(P^{(k)} \,|\,\mu^{\otimes k}) \le C\frac{k^2}{n^2} +  C\exp\bigg(-2n\Big(e^{-\transpconst T} - \frac{k}{n}\Big)_+^2\bigg),
\end{align*}
\end{theorem}

The proof of propagation of chaos is given in Section \ref{se:proofs:sublinear}, and existence and uniqueness of the McKean-Vlasov equation are proven in Section \ref{se:wellposedness}.
The estimate $\W_1(P^{(k)} ,\mu^{\otimes k}) = O(k^{3/2}/n)$ is consistent with the discussion of Remark \ref{re:transport-rate}.
The condition \eqref{condition:lingrowth} covers all cases of uniformly continuous functions $(b_0,b)$, a class for which even \emph{qualitative} propagation of chaos was apparently previously unknown. In particular, Theorem \ref{th:sublinear}  yields a local rate of propagation of chaos which complements a global estimate recently given in \cite{holding2016propagation} for H\"older interactions.

\begin{remark} \label{re:vlasov1}
Our method can easily handle \emph{kinetic} models, where instead of the SDE \eqref{def:introSDE-nonMarkov} one considers
\begin{align*}
dX^i_t &= V^i_tdt \\
dV^i_t &= \bigg( b_0(t,X^i,V^i) + \frac{1}{n-1}\sum_{j \neq i}b(t,X^i,X^j)\bigg)dt +  dW^i_t.
\end{align*}
In this case, the McKean-Vlasov equation (in PDE form) is often called the \emph{Vlasov}-Fokker-Planck equation.
The main theorems and proofs adapt without change. See Remark \ref{re:vlasov2} for additional comments, and see \cite{bolley2011stochastic,jabin-wang-bounded,hauray2019propagation} for some recent results and references on quantitative propagation of chaos for this kind of model.
\end{remark}

\begin{remark}
We work with identity diffusion coefficient for simplicity, but analogous results for any constant non-degenerate matrix follow by a simple scaling argument. The methods could also handle time- and state-dependent diffusion coefficient $\sigma(t,X^i)$, as long as the spectrum of $\sigma\sigma^\top$ is uniformly bounded above and away from zero. The key entropy estimates, however, break down if there are interactions in $\sigma$. 
\end{remark}

\begin{remark} \label{re:wellposedness}
The well-posedness assumption (1) in Theorem \ref{th:intro-dynamic} could likely be relaxed. It is used only to ensure that the relative entropy between $P^{(n)}$ and $\mu^{\otimes n}$ takes the desired form (see Lemma \ref{le:entropy-diffusions}).  Girsanov's theorem yields  this desired form immediately for bounded coefficients, but the unbounded case is delicate.
If the well-posedness assumption is removed, it may still be possible to construct a particular solution $P^{(n)}$ satisfying a desirable entropy estimate, perhaps by taking advantage of the ideas of \cite{ruf2015martingale}, and Theorem \ref{th:intro-dynamic} could then be argued for this particular solution.
This speculation is inspired in part by the approach of \cite{jabin-wang-bounded,jabin-wang-W1inf}, which proves propagation of chaos only for (potentially non-unique) solutions $P^{(n)}$ which obey a certain entropy estimate. See also \cite[Remark 1.1]{jabir2019rate}.
\end{remark}

\subsection{Reversing the relative entropy} \label{se:rev}

A result similar to Theorem \ref{th:intro-dynamic} can be obtained when the order of arguments in the relative entropy is reversed.
We will treat only the case of bounded $b$ here, though see Remark \ref{re:relaxingboundedness} for some discussion of a generalization.

\begin{theorem} \label{th:intro-dynamic-rev}
Suppose $b$ is uniformly bounded.
Assume the SDE$(b_0^{\otimes k})$ is well-posed in the sense of Definition \ref{def:SDEwellposed}, for each $k \in \N$.
Then the SDE \eqref{def:introSDE-nonMarkov} and the  McKean-Vlasov SDE \eqref{def:introMV-nonMarkov} are well-posed (in the sense of Definition \ref{def:SDEwellposed}).
Let $P^{(n)}$ and $\mu$, respectively, denote solutions whose initial distributions satisfy
\begin{align}
H(\mu^{\otimes k}_0\,|\,P^{(k)}_0) &\le C_0 k^2/n^2, \quad \text{for } k=1,\ldots,n, \label{asmp:dyn-init-rev}
\end{align}
for some $C_0 < \infty$. Assume also that $P^{(n)}$ is exchangeable.
Then, for integers $k \in [1,n]$, we have 
\begin{align*}
H(\mu^{\otimes k} \,|\,P^{(k)}) &\le  2C \frac{k^2}{n^2} + C \exp\bigg(-2n\Big(e^{-\transpconst T} - \frac{k}{n}\Big)_+^2\bigg)
\end{align*}
where we define the constants $\transpconst := 2\||b|^2\|_\infty$ and $C := (C_0+ 2\transpconst T )e^{2 \transpconst T}$.
\end{theorem}

The proof is given in Section \ref{se:proofs-dyn}.
Of course, by Pinsker's inequality, the estimate \eqref{asmp:dyn-init-rev} on the ``reversed" entropy implies the same total variation bound $\|P^{(k)}-\mu^{\otimes k}\|_{\mathrm{TV}} = O(k/n)$ as one obtains from the ``non-reversed" entropy.
And although both share the same general strategy, Theorem \ref{th:intro-dynamic-rev} has the advantage of a somewhat simpler proof compared to Theorem \ref{th:intro-dynamic}, as mentioned at the end of Section \ref{se:ideas},
because the natural estimates of the reversed entropy $H(\mu^{\otimes k} \,|\,P^{(k)})$ involve integrals with respect to the \emph{product measure} $\mu^{\otimes k}$ as opposed to $P^{(k)}$.
In other ways, however, this reversed entropy is not as natural. For instance, the subadditivity inequality \eqref{def:subadditivity-H} fails when the order of arguments is reversed, i.e., when the second argument is not a product measure.

For examples of non-i.i.d.\ initial conditions satisfying the hypothesis \eqref{asmp:dyn-init-rev}, we again refer to the companion paper \cite{lacker-static}. Specifically, in the setting of Remark \ref{re:assumptions}(ii), it is shown in \cite[Corollary 2.10]{lacker-static} that \eqref{asmp:dyn-init-rev} holds as long as $L < \kappa/\sqrt{2}$.

\subsection{Beyond pairwise interactions} \label{se:intro-infrange}

In this section, we study models with more general interactions, described by smooth functionals of the empirical measure. We focus on bounded interaction functions for simplicity.
Again let $b_0 : \R_+ \times \C_T^d \to \R^d$.
Now, let $b : \R_+ \times \C_T^d \times \P(\C_T^d) \to \R^d$ be a bounded Borel function, assumed progressive in the sense that $b(t,x,m)=b(t,x',m')$ whenever $x|_{[0,t]}=x'|_{[0,t]}$ and the measures $m$ and $m'$ have the same image under the restriction map $\C_T^d \to \C_t^d$. 
Consider the SDE system
\begin{align}
dX^i_t &= \big( b_0(t,X^i) + b(t,X^i,L^n)\big) dt +  dW^i_t, \quad L^n = \frac{1}{n}\sum_{j=1}^n\delta_{X^j}. \label{def:mainSDE-finrange}
\end{align}
The corresponding McKean-Vlasov equation is
\begin{align}
dY_t = \big( b_0(t,Y) + b(t,Y,\mu) \big) dt +  dW_t, \quad \mu=\mathrm{Law}(Y). \label{def:MV-finrange}
\end{align}
Our main Theorem \ref{th:intro-infrange} below will assume, as in Section \ref{se:intro-dyn}, that solutions of \eqref{def:mainSDE-finrange} and \eqref{def:MV-finrange} exist, with certain a priori assumptions on the solutions.
The main new assumption in this section will be that $b$ \emph{admits a power series expansion around} $\mu$, which we will take to mean that there exist bounded progressively measurable functions $b_\ell : \R_+ \times \C_T^d \times (\C_T^d)^\ell \to \R^d$ with $\|b_\ell\|_\infty \le 2^{-\ell}$, and constants $s_\ell \ge 0$, such that $\sum_{\ell=1}^\infty s_\ell < \infty$ and
\begin{align}
b(t,x,m) - b(t,x,\mu) \!=\! \sum_{\ell=1}^\infty s_\ell \langle (m-\mu)^{\otimes \ell},b_\ell(t,x,\cdot)\rangle, \ \ \forall (t,x,m) \in \R_+ \times \C_T^d \times \P(\C_T^d). \label{def:powerseries}
\end{align}
Note that $s_\ell$ and $b_\ell$ may depend on $\mu$, and we may assume without loss of generality that $b_\ell(t,x,\cdot)$ is a symmetric function of its $\ell$ coordinates, for each $(t,x)$.

\begin{example} \label{ex:flatderivative}
The setting \eqref{def:powerseries} arises most naturally when, for each $(t,x)$, the function $b(t,x,\cdot)$ admits a bounded flat derivative of every order. Here, for $\ell \in \N$, a function $F : \P(\C_T^d) \to \R^d$ is said to \emph{admit a bounded flat derivative of order $\ell$} if there exists a bounded measurable function $\delta_m^\ell F : \P(\C_T^d) \times (\C_T^d)^\ell \to \R^d$ such that
\begin{align*}
\frac{d^\ell}{dh^\ell}\Big|_{h=0}^+F((1-h)m + hm') = \langle (m'-m)^{\otimes \ell}, \delta^\ell_m F(m,\cdot)\rangle, \quad \forall m,m' \in \P(\C_T^d),
\end{align*}
where the $+$ symbol indicates that the derivative is taken from the right. The flat derivative is also known as the \emph{linear functional derivative}; see \cite[Section 2]{chassagneux2019weak} for background. Note that if $F$ admits a bounded flat derivative of every order, and if $\|\delta_m^\ell F\|_\infty /\ell! \to 0$, then a Taylor expansion (cf.\ \cite[Lemma 2.2]{chassagneux2019weak}) yields
\begin{align*}
F(m')-F(m) = \sum_{\ell=1}^\infty \frac{1}{\ell!} \langle (m'-m)^{\otimes \ell}, \delta^\ell_m F(m,\cdot)\rangle.
\end{align*}
In particular, \eqref{def:powerseries} holds if $b(t,x,\cdot)$ admits a bounded flat derivative of every order, with $\|\delta_m^\ell b(t,x,m,\cdot)\|_\infty/\ell! \to 0$ as $\ell \to \infty$.
\end{example}

\begin{example} \label{ex:analyticfunc}
Consider a measurable function $f : \R^d \times \R^d \to [-1,1]$, and let
\begin{align*}
b(t,x,m) := G( \langle m_t,\, f(x_t,\cdot)\rangle), \quad m \in \P(\C_T^d), \ x \in \C_T^d,
\end{align*}
for some analytic function $G$ defined on an open neighborhood of $[-1,1]$.
As this is a special case of Example \ref{ex:flatderivative}, we have the power series
\begin{align*}
b(t,x,m) - b(t,x,\mu) = \sum_{\ell=1}^\infty \frac{1}{\ell!} G^{(\ell)}( \langle \mu_t,\, f(x_t,\cdot)\rangle) \prod_{i=1}^\ell \langle m_t -\mu_t, f(x_t,\cdot)\rangle.
\end{align*}
This fits the setting of \eqref{def:powerseries}, if we define $s_\ell = 2^{\ell}\|G^{(\ell)}\|_\infty/ \ell!$ and
\begin{align*}
b_\ell(t,x,x^1,\ldots,x^\ell) := 2^{-\ell}\|G^{(\ell)}\|_\infty^{-1} G^{(\ell)}( \langle \mu_t,\, f(x_t,\cdot)\rangle) \prod_{i=1}^\ell f(x,x^i).
\end{align*}
A noteworthy example comes from rank-based interactions, where $d=1$ and $f(x,y) = 1_{(-\infty,x]}(y)$, and so $b(t,m,x) = G(m_t(-\infty,x_t])$. This example arises in particle approximations of Burgers-type PDEs  \cite{shkolnikov2012large,jourdain2013propagation}.
\end{example}

The rate of decay of the coefficients $(s_\ell)_{\ell \in \N}$ will influence the rate of propagation of chaos that our method can achieve. For  $x, p \ge 0$ define 
\begin{align}
\bar{s}_p := \sum_{\ell=1}^\infty \ell^p s_\ell, \qquad \tailfn_0(x) := \frac{1}{\bar{s}_0} \sum_{\ell=\lfloor x\rfloor+1}^{\infty} s_\ell. \label{def:tailfn}
\end{align}
Note that $\lim_{x \to \infty} \tailfn_0(x)=0$, as we assumed that $\bar{s}_0 < \infty$. 
We state first a general theorem to illustrate the similar nature of the estimate to Theorem \ref{th:intro-dynamic}, and then Corollary \ref{co:infrange-subexp} gives a more concrete implication.

\begin{theorem} \label{th:intro-infrange}
Assume $b$ is a bounded progressive function as described above. 
Assume that there exists a weak solution of the SDE \eqref{def:mainSDE-finrange} whose law $P^{(n)} \in \P((\C_T^d)^n)$ is exchangeable, and that there exists a weak solution of the McKean-Vlasov equation \eqref{def:MV-finrange} with law $\mu \in \P(\C_T^d)$.
Assume $b$ satisfies \eqref{def:powerseries}.
Assume the SDE$(b_0^{\otimes k})$ is well-posed in the sense of Definition \ref{def:SDEwellposed}, for each $k =1,\ldots,n$.
Assume finally that  there exists $C_0 < \infty$ such that 
\begin{align}
H(P^{(k)}_0\,|\,\mu^{\otimes k}_0) &\le C_0 k^3/n^2, \quad \text{for } k=1,\ldots,n. \label{asmp:infrange-init}
\end{align}
Then, for $k,\ell \in \N$ satisfying $1 \le \ell + k \le n/2$, we have
\begin{align}
\begin{split}
H(P^{(k)} \,|\, \mu^{\otimes k}) &\le \big(24 C_0\bar{s}_0^2  +  4 (\bar{s}_2/\bar{s}_0)^2\big) e^{12\bar{s}_0^2 \ell T}\frac{k^3}{n^2}  +  e^{4\bar{s}_0^2 \ell T} \tailfn_0^2(\ell)\frac{k}{\ell} \\
	&\quad +  n\left(C_0 + \bar{s}_0^2 T  \right)\exp\bigg(-  \frac{n}{\ell} \bigg(e^{-4 \bar{s}_0^2 \ell T} - \frac{k}{2n } \bigg)_+^2 \bigg).
\end{split} \label{infrange-est}
\end{align}
\end{theorem}

The proof of Theorem \ref{th:intro-infrange} is given in Section \ref{se:proofs-inf}.
The last term of \eqref{infrange-est} is somewhat similar to the estimates of Theorems \ref{th:intro-dynamic} and \ref{th:intro-dynamic-rev}. The $k^3$ in the first term, though, is likely suboptimal; see Remark \ref{re:inf-selfimprovement}.
The term in \eqref{infrange-est} involving $\tailfn_0$ is new and depends on the decay rate of $(s_\ell)_{\ell \in \N}$.
Theorem \ref{th:intro-infrange} is clearly of no use unless this term can be sent to zero, or $\tailfn_0^2(\ell)e^{8\bar{s}_0^2\ell T}/\ell \to 0$ as $\ell \to \infty$.
Note that the variable $\ell$ is free in the sense that it does not appear on the left-hand side and may thus be optimized. But the final term in \eqref{infrange-est} is increasing in $\ell$, which precludes sending $\ell\to\infty$ too quickly relative to $n$. 
The tradeoff is easiest to handle in the finite-range case. If there exists $r \in \N$ such that $s_\ell=0$ for all $\ell > r$, then we may take $\ell=r$ in \eqref{infrange-est} to get $\tailfn_0(\ell)=0$. This yields $H(P^{(k)} \,|\, \mu^{\otimes k}) = O(k^3/n^2)$. The following corollary shows that we can get arbitrarily close to the same exponent even for infinite-range interactions, as long as $\tailfn_0(\ell)$ decays exponentially with $\ell$.

\begin{corollary} \label{co:infrange-subexp}
Grant the assumptions of Theorem \ref{th:intro-infrange}, and assume $\sum_{\ell=1}^\infty s_\ell e^{c\ell} < \infty$ for all $c > 0$. 
If $k=o(n)$, then $H(P^{(k)} \,|\, \mu^{\otimes k}) = O(k^{3-\epsilon}/n^{2-\epsilon} )$ for each $\epsilon > 0$.
\end{corollary}
\begin{proof}
Use a Chernoff bound to deduce
\begin{align*}
\tailfn_0(x) &\le e^{-cx}\sum_{\ell=1}^\infty s_\ell e^{c\ell}, \quad \forall c,x > 0.
\end{align*}
That is, for any $c > 0$, there exists $C_c$ such that $\tailfn_0(x) \le C_ce^{-cx}$ for all $x > 0$. Note also
that the finiteness of $\sum s_\ell e^{c\ell}$ for $c > 0$ implies $\bar{s}_p < \infty$ for all $p \ge 0$.

Let $0 < \epsilon < 1/2$, and take $\ell=\lfloor \frac{\epsilon}{12\bar{s}_0^2T}\log\frac{n}{k}\rfloor$. Apply Theorem \ref{th:intro-infrange}. Note that
\begin{align*}
n\exp\bigg(-  \frac{n}{\ell} \bigg(e^{-4 \bar{s}_0^2 \ell T} - \frac{k}{2n } \bigg)_+^2 \bigg) &\le n\exp\bigg(-  \frac{n}{\ell} \bigg(\left(\frac{k}{n}\right)^{\epsilon/3} - \frac{k}{2n } \bigg)_+^2 \bigg) \\
	&= n\exp\bigg(-  \frac{k^{2\epsilon}n^{1-2\epsilon/3}}{\ell} \bigg(1 - \frac12\bigg(\frac{k}{n}\bigg)^{1-\epsilon/3} \bigg)_+^2 \bigg)
\end{align*}
is clearly $O(n^{-p})$ for any $p > 0$. Bounding $\tailfn_0^2(\ell)$ by a constant times $e^{-4\bar{s}_0^2T(1+r) \ell}$, for $r > 1$ to be determined, the other terms in \eqref{infrange-est} are bounded by
\begin{align*}
e^{12\bar{s}_0^2\ell T}\frac{k^3}{n^2} &\le C_1 \frac{k^{3-\epsilon}}{ n^{2-\epsilon}},  \\
e^{4\bar{s}_0^2 \ell T} \tailfn_0^2(\ell) k &\le C_2 k e^{-4\bar{s}_0^2T r \ell} \le C_3 k(k/n)^{r \epsilon /3 }. 
\end{align*}
for constants $C_1,C_2,C_3>0$ not depending on $(k,n,\epsilon)$.
Choose $r=(6/\epsilon)-3$ to get $2- \epsilon=r\epsilon/3$, which makes each term $O(k(k/n)^{2-\epsilon})$.
\end{proof}

\begin{remark}
The condition $\sum_{\ell=1}^\infty s_\ell e^{c\ell} < \infty$ holds in the setting of Example \ref{ex:flatderivative} if $\|\delta_m^\ell b(t,x,m,\cdot)\|_\infty \le a_1e^{a_2\ell}$ for all $\ell$, for some $a_1,a_2 \ge 0$. Indeed, we may take $b_\ell(t,x,\cdot) := 2^{-\ell}\|\delta_m^\ell b\|_\infty^{-1}\delta_m^\ell b(t,x,\mu,\cdot)$ and $s_\ell := 2^\ell \|\delta_m^\ell b(t,x,m,\cdot)\|_\infty/\ell!$, noting that $s_\ell \le Ce^{c\ell -\ell\log\ell}$ by Stirling's formula. A similar remark applies to Example \ref{ex:analyticfunc}.
\end{remark}

\section{The Gaussian case} \label{se:gaussian}

This section gives the details of Example \ref{ex:gaussian-dyn}, to illustrate the optimality of the $O((k/n)^2)$ rate of propagation of chaos. Throughout the section, $J_n$ denotes the $n \times n$ matrix of all ones, and $I_n$ is the identity matrix.
The covariance matrix of an exchangeable Gaussian measure on $\R^n$ necessarily belongs to the span of $\{I_n,J_n\}$.
For this reason, we begin with a lemma summarizing the basic algebraic properties of these matrices. We omit its proof, which follows easily from the observation that $J_n^2=nJ_n$.

\begin{lemma} \label{le:gaussian-algebra}
Let $x,y,z,w \in \R$, and let $A=xI_n+yJ_n$ and $B=zI_n+wJ_n$. Then
\begin{align*}
AB &= BA, \qquad \text{and} \qquad  e^A = e^x\left(I_n + \frac{e^{yn}-1}{n}J_n\right).
\end{align*}
The matrix $A$ has eigenvalues $x$ and $x+yn$, with multiplicities $n-1$ and $1$, respectively.
\end{lemma}

A well known calculation shows that the Wasserstein distance between two centered Gaussian measures $\gamma_i$ with covariance matrices $\Sigma_i$, for $i=1,2$, is given by
\begin{align*}
\W_2^2(\gamma_1,\gamma_2) &= \tr\big(\Sigma_1+\Sigma_2 - 2(\Sigma_1\Sigma_2)^{1/2}\big).
\end{align*}
See \cite{dowson1982frechet}.
If $\Sigma_1$ and $\Sigma_2$ commute, then this reduces to
\begin{align}
\W_2^2(\gamma_1,\gamma_2) &= \tr\big((\Sigma_1^{1/2} - \Sigma_2^{1/2})^2\big). \label{eq:W2-gaussian}
\end{align}
Now, for $k\in \N$ and $i=1,2$, suppose $\Sigma_i = a_iI_k + b_iJ_k$ for some $a_i > 0$ and $b_i \in \R$. Noting that $\Sigma_i$ has eigenvalues $a_i$ and $a_i+b_ik$, with respective multiplicities $k-1$ and $1$, we may apply \eqref{eq:W2-gaussian} along with a simultaneous diagonalization of $\Sigma_1$ and $\Sigma_2$ to find
\begin{align}
\W_2^2(\gamma_1,\gamma_2) &= (k-1)\big(a_1^{1/2} - a_2^{1/2}\big)^2 + \big((a_1+b_1k)^{1/2} - (a_2+b_2k)^{1/2}\big)^2. \label{eq:W2-gaussian-ourcase}
\end{align}

Let us turn to the setting of Example \ref{ex:gaussian-dyn}. The $n$-particle SDE system takes the form
\begin{align*}
dX^i_t = -\bigg( aX^i_t + \frac{b}{n-1}\sum_{j \neq i} X^j_t\bigg)dt + dW^i_t, \ \ X^i_0=0, \quad i=1,\ldots,n.
\end{align*}
In vector form, letting $\bm{X}_t=(X^1_t,\ldots,X^n_t)$ and $\bm{W}_t=(W^1_t,\ldots,W^n_t)$, we may write this as 
\begin{align*}
d\bm{X}_t = - A\bm{X}_tdt +  d\bm{W}_t, \quad \text{where} \quad A_n := aI_n + \frac{b}{n-1} J_n.
\end{align*}
The solution of this multivariate Ornstein-Uhlenbeck equation is expressed as
\begin{align*}
\bm X_t = \int_0^t e^{-A_n(t-s)}\,d\bm{W}_s.
\end{align*}
Hence, the law $P^{(n,n)}_t$ of $\bm{X}_t$ is a centered Gaussian with covariance matrix 
\begin{align*}
\Sigma^n_t := \int_0^t e^{-2sA_n}\,ds.
\end{align*}
Using Lemma \ref{le:gaussian-algebra} we can write this as 
\begin{align*}
\Sigma^n_t &= \int_0^t e^{-2as}\bigg(I_n + \frac{e^{-\frac{2sbn}{n-1}}-1}{n}J_n\bigg) ds \\
	&= \frac{1}{2a}(1-e^{-2at})I_n + \frac{1}{n}\left(\int_0^t \left( e^{-2s\left(a + \frac{bn}{n-1}\right)} - e^{-2as}\right)ds\right)J_n \\
	&=  \frac{1}{2a}(1-e^{-2at}) I_n + \frac{1}{2n}\bigg(\frac{1 - e^{-2t\left(a + \frac{bn}{n-1}\right)}}{a + \frac{bn}{n-1}} - \frac{1 - e^{-2at}}{a} \bigg)J_n  \\
	&= \frac{1}{2a} (1-e^{-2at})[I_n - c_n(t)J_n],
\end{align*}
where we define
\begin{align*}
v(t) := \frac{1}{2a} (1-e^{-2at}), \qquad c_n(t) := \frac{1}{n}\left(1 - \frac{a}{a + \frac{bn}{n-1}}\frac{1 - e^{-2t\left(a + \frac{bn}{n-1}\right)}}{1 - e^{-2at}} \right).
\end{align*}
Note that $0 \le c_n(t) \le  1/n$.
The $k$-particle marginal $P^{(n,k)}_t$ is a centered Gaussian with covariance matrix $\Sigma^{(n,k)}_t = v(t)[I_k - c_n(t)J_k]$. 
Let $\mu_t$ denote the centered one-dimensional Gaussian measure with variance $v(t)$. Clearly, $P^{(n,k)}_t \to \mu^{\otimes k}_t$ weakly as $n\to\infty$, for each $k$.

Applying the formula \eqref{eq:W2-gaussian-ourcase}, with $a_1=a_2=v(t)$, $b_1=-v(t)c_n(t)$, and $b_2=0$, we find
\begin{align*}
\W_2^2(P^{(n,k)}_t,\mu^{\otimes k}_t) &= v(t)\big( (1 - kc_n(t))^{1/2} - 1\big)^2.
\end{align*}
Note that $((1+x)^{1/2}-1)^2/x^2 \to 1/4$ as $x \to 0$, and that
\begin{align}
\lim_{n\to\infty} nc_n(t) &= 1 - \frac{a}{a+b} \frac{1 - e^{-2t(a + b)}}{1 - e^{-2at}}. \label{limit-cn}
\end{align}
Hence, if $n\to\infty$ and $k/n\to 0$, then $kc_n(t) \to 0$, and so
\begin{align*}
 (n/k)^2 \W_2^2(P^{(n,k)}_t,\mu^{\otimes k}_t) &= \left( n c_n(t) \right)^2\frac{\W_2^2(P^{(n,k)}_t,\mu^{\otimes k}_t )}{ (kc_n(t))^2} \\
	&\to \left(1 - \frac{a}{a+b} \frac{1 - e^{-2t(a + b)}}{1 - e^{-2at}}\right)^2 \frac{1}{8a}(1-e^{-2at})  .
\end{align*}
In fact, this limit is uniform over $t \ge 0$, since \eqref{limit-cn} is.
Finally, the claim \eqref{ex:Gaussian-claim} follows from the simple inequality $\W_2^2(P^{(n,k)}_t,\mu^{\otimes k}_t) \le \W_2^2(P^{(n,k)},\mu^{\otimes k})$ for $t \in [0,T]$.

\section{Proofs for pairwise interactions} \label{se:proofs-dyn}

Working in the setting of Section \ref{se:intro-dyn}, this section is devoted to the proofs of Theorems \ref{th:intro-dynamic} and \ref{th:intro-dynamic-rev}. We begin with three short sections summarizing a projection lemma needed to derive the BBGKY hierarchy, some additional notation,  and the needed properties of relative entropy.

\subsection{A projection lemma} 

The following lemma will later let us identify the dynamics of $k$ particles out of the $n$-particle system. It is well known; cf.\ \cite[Corollary 3.11]{brunick2013mimicking}.

\begin{lemma} \label{le:proj-pathdep}
Suppose a filtered probability space $(\Omega,\F,\FF,\PP)$ supports an $\FF$-Brownian motion $W$ of dimension $k$, a continuous $k$-dimensional process $Y$, and an $\FF$-progressively measurable process $(b_t)_{t\ge 0}$ satisfying $\E\int_0^T|b_t|\,dt < \infty$ and also
\begin{align*}
Y_t = Y_0 + \int_0^t b_s\,ds + W_t, \qquad t \in [0,T].
\end{align*}
Let $\FF^Y=(\F^Y_t)_{t \in [0,T]}$ denote the filtration generated by $(Y_t)_{t \in [0,T]}$.
Let $\widehat{b} : [0,T] \times \C_T^k \to \R^k$ be a progressively measurable function  such that
\begin{align*}
\widehat{b}(t,Y) := \E\big[b_t\,|\,\F^Y_t\big], \qquad a.s., \ \ a.e. \ t  \in [0,T],
\end{align*}
Then there exists an $\FF^Y$-Brownian motion $\widehat{W}$ of dimension $k$ such that
\begin{align*}
Y_t = Y_0 + \int_0^t \widehat{b}(s,Y)\,ds + \widehat{W}_t, \qquad a.s., \ \ a.e. \ t  \in [0,T],
\end{align*}
\end{lemma}
\begin{proof}
This is a classical argument in filtering theory, cf.\ \cite[Theorem VI.8.4]{rogers2000diffusions}:
The process
\begin{align*}
\widehat{W}_t := Y_t - Y_0 - \int_0^t \widehat{b}(s,Y)\,ds = W_t + \int_0^t \big(b_s - \E\big[b_s\,|\,\F^Y_s\big]\big)\,ds
\end{align*}
is easily seen to be an $\FF^Y$-Brownian motion using L\'evy's criterion.
\end{proof}

\begin{remark} \label{re:optionalprojection}
In the setting of Lemma \ref{le:proj-pathdep}, the existence of a such a progressively  measurable function $\widehat{b}$, i.e., a progressively measurable version of the conditional expectation, follows from standard measure-theoretic arguments such as  \cite[Proposition 5.1]{brunick2013mimicking}.
\end{remark}

\subsection{Additional notation for measures on path space}

For any $k \in \N$, $Q \in \P(\C_T^k)$, and $t \in [0,T]$,  denote by  $Q[t]$ the projection to $\C_t^k$, i.e., the pushforward of $Q$ by the restriction map $x \mapsto x|_{[0,t]}$. In particular, $Q[T]=Q$. Recalling that $Q_t \in \P(\R^k)$ denotes the time-$t$ marginal, and that $\C_0^k$ is identified with $\R^k$, we have $Q[0]=Q_0$.

We will systematically make the following minor abuse of notation. Let $\overline{b} : [0,T] \times \C_T^d \to \R^d$ be progressively measurable, as defined in Section \ref{se:pathspacenotation}). Then, for each $t \in [0,T]$, we may unambiguously view $\overline{b}(t,\cdot)$ as a Borel measurable function on $\C_t^d$, rather than $\C_T^d$, via the identification $\overline{b}(t,x|_{[0,t]}):=\overline{b}(t,x)$ for $x \in \C_T^d$.
In particular, for $Q \in \P(\C_T^d)$, we have $\langle Q[t],\overline{b}(t,\cdot)\rangle = \langle Q,\overline{b}(t,\cdot)\rangle$.

\subsection{Relative entropy} \label{se:entropyestimates}

This section summarizes the properties and calculations involving relative entropy that we will need. We will use repeatedly the \emph{chain rule},
\begin{align}
H\big(m^1(dx)K_x^1(dy) \,|\, m^2(dx)K_x^2(dy)\big) &= H(m^1 \,|\, m^2 ) + \int_{E_1} H(K_x^1\,|\,K_x^2) \, m^1(dx), \label{def:chainrule}
\end{align}
for (disintegrated) probability measures $(m^i(dx)K_x^i(dy))_{i=1,2}$ on a product space $E_1 \times E_2$. See \cite[Theorem 2.6]{budhiraja-dupuis} for a proof. Combined with nonnegativity of relative entropy, it follows that
\begin{align}
H\big(m^1(dx)K_x^1(dy) \,|\, m^2(dx)K_x^2(dy)\big) &\ge \int_{E_1} H(K_x^1\,|\,K_x^2) \, m^1(dx). \label{def:chainrule2}
\end{align}
We will also make use use of the so-called \emph{data processing inequality}, which states that $H(\nu \circ f^{-1}\,|\,\nu' \circ f^{-1}) \le H(\nu\,|\,\nu')$ for any probability measures $\nu,\nu'$ on a common measurable space and any measurable function $f$ into another measurable space; this follows easily from Jensen's inequality for conditional expectation.
As a consequence, we have the following simple but useful inequalities:
\begin{align}
H(P^{(k)}[s]\,|\, \mu^{\otimes k}[s]) &\le H(P^{(k)}[t]\,|\, \mu^{\otimes k}[t]), \qquad 0 \le s \le t \le T, \ 1 \le k \le n \label{ineq:dataproc1} \\
H(P^{(k)}[t]\,|\, \mu^{\otimes k}[t]) &\le H(P^{(k+1)}[t]\,|\, \mu^{\otimes (k+1)}[t]), \qquad t \in [0,T], \ 1 \le k < n. \label{ineq:dataproc2}
\end{align}
Note that \eqref{ineq:dataproc1} is true for measures on path space but fails in general for the time-marginals; that is, \eqref{ineq:dataproc1} is not true if $P^{(k)}[\cdot]$ and $\mu^{\otimes k}[\cdot]$ are replaced on both sides by $P^{(k)}_\cdot$ and $\mu^{\otimes k}_\cdot$, respectively.
Lastly, we recall the classical Pinsker's inequality: For probability measures $\nu$ and $\nu'$ on a common measurable space $\Omega$, and for a bounded measurable function $f : \Omega \to \R$, 
\begin{align*}
\langle \nu-\nu',f\rangle^2 \le 2\|f\|_\infty^2 H(\nu\,|\,\nu').
\end{align*}
This extends immediately to vector-valued functions: If $f : \Omega \to \R^d$ is measurable and bounded, then
\begin{align}
\begin{split}
|\langle \nu - \nu',f\rangle|^2 &= \sup_{u \in \R^d,\,|u|=1}\langle \nu - \nu',f \cdot u\rangle^2 \le \sup_{u \in \R^d,\,|u|=1} 2\|f \cdot u\|_\infty^2 H(\nu\,|\,\nu') \\
	&\le 2\||f|^2\|_\infty H(\nu\,|\,\nu').
\end{split} \label{ineq:Pinsker}
\end{align}

The following simple rewriting of the key assumption \eqref{asmp:dyn-transp} will be useful:

\begin{lemma} \label{le:keyasmp-rewrite}
Let $b : [0,T] \times \C_T^d \times \C_T^d \to \R^d$ be progressively measurable.
Let $\mu \in \P(\C_T^d)$ and $0 < \transpconst < \infty$. Then \eqref{asmp:dyn-transp} holds if and only if
\begin{align}
|\langle \mu[t] - \nu, b(t,x,\cdot)\rangle|^2 \le  \transpconst H(\nu\,|\,\mu[t]), \ \ \forall t \in (0,T), \ x \in \C_T^d, \ \nu \in \P(\C_t^d). \label{asmp:dyn-transp-equiv}
\end{align}
\end{lemma}
\begin{proof}
Suppose \eqref{asmp:dyn-transp-equiv} holds. Let $\nu \in \P(\C_T^d)$. Then, for $(t,x) \in (0,T) \times \C_T^d$, we use the progressive measurability of $b$ followed by \eqref{asmp:dyn-transp-equiv} and the monotonicity of entropy \eqref{ineq:dataproc1} to get
\begin{align*}
|\langle \mu  - \nu , b(t,x,\cdot)\rangle|^2 = |\langle \mu[t]  - \nu[t] , b(t,x,\cdot)\rangle|^2 &\le \transpconst H(\nu[t]\,|\,\mu[t]) \le \transpconst H(\nu\,|\,\mu),
\end{align*}
which proves \eqref{asmp:dyn-transp}. Conversely, suppose \eqref{asmp:dyn-transp} holds. Let $t \in (0,T)$. Define the kernel $\C_t^d \ni x \mapsto K_x \in \P(\C_T^d)$ via the disintegration $\mu(dx)=\mu[t](dx|_{[0,t]})K_{x|_{[0,t]}}(dx)$. That is, $K_{x|_{[0,t]}}$ is the conditional law of $x$ given $x|_{[0,t]}$ under $\mu$. Let $\nu \in \P(\C_t^d)$, and define $\overline\nu \in \P(\C_T^d)$ by $\overline\nu(dx)=\nu(dx|_{[0,t]})K_{x|_{[0,t]}}(dx)$. Then $\overline\nu[t]=\nu$, and so progressive measurability of $b$ and \eqref{asmp:dyn-transp} imply
\begin{align*}
|\langle \mu[t] - \nu, b(t,x,\cdot)\rangle|^2 &= |\langle \mu - \overline\nu, b(t,x,\cdot)\rangle|^2 \le \transpconst H(\overline\nu\,|\,\mu),
\end{align*}
for each $x \in \C_T^d$.
By the chain rule of relative entropy \eqref{def:chainrule}, we have
\begin{align*}
H(\overline\nu\,|\,\mu) &= H(\overline\nu[t]\,|\,\mu[t]) + \int_{\C_t^d}H(K_x\,|\,K_x)\,\mu[t](dx) = H(\nu\,|\,\mu[t]).
\end{align*}
This proves \eqref{asmp:dyn-transp-equiv}.
\end{proof}

As a final preparation, we state a version of a classical entropy estimate for diffusions which follows from Girsanov's theorem, though care is needed when working with unbounded coefficients. We give the proof in Appendix \ref{ap:entropyestimates}, borrowing (ii) from \cite{leonard2012girsanov}.

\begin{lemma} \label{le:entropy-diffusions}
Let $k \in \N$, and let $b^1,b^2 : [0,T] \times \C^k_T \to \R^k$ be progressively measurable. 
Suppose for each $i=1,2$ that the SDE
\begin{align*}
dZ^i_t = b^i(t,Z^i)dt +  dW^i_t, \qquad t \in [0,T],
\end{align*}
admits a weak solution, and let $P^i \in \P(\C^k_T)$ denote its law. Assume the SDE$(b^2)$ is well-posed in the sense of Definition \ref{def:SDEwellposed}.
\begin{enumerate}[(i)]
\item If
\begin{align}
\PP\bigg(\int_0^T|b^1(t,Z^i)-b^2(t,Z^i)|^2dt < \infty\bigg)=1, \quad i=1,2. \label{asmp:finiteentropy1}
\end{align}
then for $0 \le t \le T$ we have
\begin{align*}
H(P^1[t]\,|\,P^2[t]) &\le H(P^1_0\,|\,P^2_0) + \frac12 \E\int_0^t|b^1(s,Z^1)-b^2(s,Z^1)|^2\,ds.
\end{align*}
\item If $H(P^1\,|\,P^2) < \infty$,
then for $0 \le t \le T$ we have
\begin{align}
H(P^1[t]\,|\,P^2[t]) &= H(P^1_0\,|\,P^2_0) + \frac12 \E\int_0^t|b^1(s,Z^1)-b^2(s,Z^1)|^2\,ds. \label{eqn:entropyidentity1}
\end{align}
\item If
\begin{align}
\E\int_0^T|b^1(t,Z^i)-b^2(t,Z^i)|^2dt < \infty, \quad i=1,2, \label{asmp:finiteentropy2}
\end{align}
then \eqref{eqn:entropyidentity1} holds for $0 \le t \le T$.
\end{enumerate}
\end{lemma}

\begin{remark} \label{re:vlasov2}
By a simple change of variables, this entropy estimate clearly works for any non-degenerate noise. Precisely, if the identity diffusion coefficient $dW^i_t$ were replaced by $\sigma dW^i_t$, for a non-degenerate matrix $\sigma \in \R^{k \times k}$, then the estimates would be in terms of  $|\sigma^{-1}(b^1-b^2)|^2$  rather than $|b^1-b^2|^2$. Moreover, this remains valid even for degenerate $\sigma$, with $\sigma^{-1}$ replaced by the (Moore–Penrose) pseudoinverse, as long $b^1-b^2$ takes values in the range (column space) of $\sigma$. This is why our arguments work also for the kinetic case mentioned in Remark \ref{re:vlasov1}, which leads to drifts and diffusion coefficients of the form
\begin{align*}
b^i(t,x,v) = \begin{pmatrix} v_t \\ \tilde{b}^i(t,x,v) \end{pmatrix}, \qquad 
\sigma = \begin{pmatrix} 0 & 0 \\ 0 & I \end{pmatrix},
\end{align*}
for $i=1,2$, and we see indeed that $b^1-b^2=(0, \tilde{b}^1-\tilde{b}^2)^\top$ always belongs to the range of $\sigma$.
\end{remark}

\subsection{Proof of Theorem \ref{th:intro-dynamic}}  \label{se:proofs:main}

We now apply the general principles of Lemmas \ref{le:proj-pathdep} and \ref{le:entropy-diffusions} to derive our initial entropy estimate.
Throughout the rest of this section, we abbreviate
\begin{align}
H^k_t := H(P^{(k)}[t] \,|\, \mu^{\otimes k}[t]), \label{abbrev:H^k1}
\end{align}
for $1 \le k \le n$ and $t \in [0,T]$. As in Section \ref{se:intro-dyn}, $(X^1,\ldots,X^n)$ denotes the solution of the SDE \eqref{def:introSDE-nonMarkov} with law $P^{(n)}$. Let $\F^k_t = \sigma((X^1_s,\ldots,X^k_s)_{s \le t})$ denote the natural filtration of the first $k$ coordinates.

\begin{lemma} \label{le:init-entropybound-dyn}
Grant the assumptions of Theorem \ref{th:intro-dynamic}, and let $1 \le k \le n$. Then $t \mapsto H^k_t$ is absolutely continuous, and for a.e.\ $t \in [0,T]$ we have
\begin{align}
\begin{split}
\frac{d}{dt}H^k_t \le &\frac{k}{ (n-1)^2}\E\Bigg[ \Bigg|\sum_{j =2}^k \Big(b(t,X^1,X^j) - \langle \mu,b(t,X^1,\cdot)\rangle  \Big)\Bigg|^2\Bigg]   \\
&\quad + \frac{k(n-k)^2}{ (n-1)^2}\E\left[ \left|\E\big[b(t,X^1,X^n) \,|\, \F^k_t \big] - \langle \mu,b(t,X^1,\cdot)\rangle \right|^2\right].
\end{split} \label{def:init-entropybound-dyn}
\end{align}
Moreover, with $M$ defined as in \eqref{asmp:dyn-moment}, we have
\begin{align}
H^n_T \le H^n_0 + MnT. \label{def:init-entropybound-dyn2}
\end{align}
\end{lemma}
\begin{proof} {\ }

\noindent\textit{Step 1.} We first prove \eqref{def:init-entropybound-dyn2}, by applying Lemma \ref{le:entropy-diffusions}(iii).
Thanks to  the well-posedness asumption (1) of Theorem \ref{th:intro-dynamic},
we must check only the integrability condition \eqref{asmp:finiteentropy2}, which we will do using the assumptions \eqref{asmp:dyn-transp} and \eqref{asmp:dyn-moment}.
Recall first that, by Lemma \ref{le:integ-transp-equiv}, the assumption \eqref{asmp:dyn-transp} implies an exponential square-integrability, which in particular implies
\begin{align}
\int_{\C_T^d \times \C_T^d}\int_0^T \left|  b(t,x^1,x^2) - \langle \mu,b(t,x^1,\cdot)\rangle \right|^2 \, dt \, \mu^{\otimes 2}(dx^1,dx^2) < \infty. \label{pf:entest-dyn1}
\end{align}
To check \eqref{asmp:finiteentropy2} in our context requires showing that
\begin{align*}
\sum_{i=1}^n \int_{(\C_T^d)^n} \int_0^T \bigg|\frac{1}{n-1}\sum_{j \neq i} b(t,x^i,x^j) - \langle \mu,b(t,x^i,\cdot)\rangle\bigg|^2 \,dt \, P^{(n)}(dx^1,\ldots,dx^n) &< \infty, \\
\sum_{i=1}^n \int_{(\C_T^d)^n} \int_0^T \bigg|\frac{1}{n-1}\sum_{j \neq i} b(t,x^i,x^j) - \langle \mu,b(t,x^i,\cdot)\rangle\bigg|^2 \,dt \, \mu^{\otimes n}(dx^1,\ldots,dx^n) &< \infty.
\end{align*}
By exchangeability of $P^{(n)}$ and $\mu^{\otimes n}$, and using convexity of $x \mapsto x^2$, it suffices to check that
\begin{align*}
n \int_{(\C_T^d)^2} \int_0^T \big|b(t,x^1,x^2) - \langle \mu,b(t,x^1,\cdot)\rangle\big|^2 \,dt \, P^{(2)}(dx^1,dx^2) &< \infty, \\
n \int_{(\C_T^d)^2} \int_0^T \big|b(t,x^1,x^2) - \langle \mu,b(t,x^1,\cdot)\rangle\big|^2 \,dt \, \mu^{\otimes 2}(dx^1,dx^2) &< \infty.
\end{align*}
These two claims follow from the assumption \eqref{asmp:dyn-moment} and \eqref{pf:entest-dyn1}, respectively.
We may thus use the entropy formula from Lemma \ref{le:entropy-diffusions}(iii) along with exchangeability to find
\begin{align*}
H^n_T &= H^n_0 + \frac12 \E \sum_{i=1}^n\int_0^T \left|\frac{1}{n-1}\sum_{j \neq i} b(t,X^i,X^j) - \langle \mu,b(t,X^i,\cdot)\rangle \right|^2dt \\
	&\le H^n_0 + \frac12 nMT.
\end{align*}

{\ }

\noindent\textit{Step 2.} Next, we prove \eqref{def:init-entropybound-dyn} for $k < n$.
We begin by applying Lemma \ref{le:proj-pathdep}.
Consider a progressively measurable function $\widehat{b}^k_i : [0,T] \times (\C_T^d)^k \to \R^d$ such that 
\begin{align}
\widehat{b}^k_i(t,X^1,\ldots,X^k) = \E\big[b(t,X^i,X^n) \,|\, \F^k_t \big], \ \ a.s., \ a.e. \ t \in [0,T]. \label{def:bhat-dyn}
\end{align}
Such a function exists by Remark \ref{re:optionalprojection}.
Exchangeability of $P^{(n)}$ and progressively measurability of $(b_0,b)$ imply
\begin{align*}
\E &\left[b_0(t,X^i) + \frac{1}{n-1}\sum_{j \neq i}b(t,X^i,X^j) \,\Big|\, \F^k_t \right] \\
	&\quad = b_0(t,X^i) + \frac{1}{n-1}\sum_{j \neq i, \, j \le k}b(t,X^i,X^j) + \frac{n-k}{n-1}\widehat{b}^k_i(t,X^1,\ldots,X^k).
\end{align*}
By Lemma \ref{le:proj-pathdep}, we see that $(X^1,\ldots,X^k)$ solves the SDE system
\begin{align}
dX^i_t = \! \bigg(b_0(t,X^i) + \frac{1}{n-1}\sum_{j \neq i, \, j \le k}b(t,X^i,X^j) + \frac{n-k}{n-1}\widehat{b}^k_i(t,X^1,\ldots,X^k) \bigg)dt + d\widehat{W}^i_t, \label{pf:def:SDE-Pk}
\end{align}
for some independent Brownian motions $\widehat{W}^1,\ldots,\widehat{W}^k$.
We showed in Step 1 that $H^n_T < \infty$, and thus $H^k_T \le H^n_T < \infty$ according to \eqref{ineq:dataproc2}.
We may then apply Lemma \ref{le:entropy-diffusions}(ii), noting that the linearized SDE satisfied by $\mu^{\otimes k}$ is well-posed by assumption (1) of the theorem. We deduce the claimed absolute continuity of $t \mapsto H^k_t$ and, for a.e.\ $t \in [0,T]$, the identity
\begin{align*}
\frac{d}{dt}H^k_t &= \frac{1}{2}\sum_{i=1}^k \E\Bigg[ \Bigg|\frac{1}{n-1}\sum_{j \neq i, \, j \le k}b(t,X^i,X^j) + \frac{n-k}{n-1}\widehat{b}^k_i(t,X^1,\ldots,X^k) - \langle \mu,b(t,X^i,\cdot)\rangle \Bigg|^2\Bigg].
\end{align*}
Rearrange and apply the elementary inequality $\tfrac12(x+y)^2 \le x^2+y^2$ to get
\begin{align*}
\frac{d}{dt}H^k_t \le \, & \sum_{i=1}^k \E\Bigg[ \bigg|\frac{1}{n-1}\sum_{j \neq i, \, j \le k}\big(b(t,X^i,X^j) - \langle \mu,b(t,X^i,\cdot)\rangle  \big)\bigg|^2\Bigg]  \\
&\quad + \sum_{i=1}^k \E\Bigg[ \bigg|\frac{n-k}{n-1}\big(\E\big[b(t,X^i,X^n) \,|\, \F^k_t \big] - \langle \mu,b(t,X^i,\cdot)\rangle \big) \bigg|^2\Bigg].
\end{align*}
Complete the proof of \eqref{def:init-entropybound-dyn} by using exchangeability.
\end{proof}

\begin{remark} \label{re:BBGKY}
The equation \eqref{pf:def:SDE-Pk} can be seen as a non-Markovian form of the BBGKY hierarchy. For each $k$, it describes the marginal law $P^{(k)}$ in terms of the  $k$-particle SDE system governed by the drift $\widehat{b}^k_i$ defined in \eqref{def:bhat-dyn}. This drift depends on the conditional law of particle $k+1$ given the first $k$. In other words, the $k$-particle marginal dynamics depend on the $(k+1)$-particle marginal dynamics, hence the term \emph{hierarchy}.
\end{remark}

With these preparations, we are now ready for the main line of the proof of Theorem \ref{th:intro-dynamic}. Recall the definition of $H^k_t$ in \eqref{abbrev:H^k1}. We begin from \eqref{def:init-entropybound-dyn}, for a given $k < n$. Use convexity of $x\mapsto x^2$ to deduce the crude bound
\begin{align}
\frac{k}{(n-1)^2}\E\Bigg[ \Bigg|\sum_{j =2}^k \bigg(b(t,X^1,X^j) - \langle \mu,b(t,X^1,\cdot)\rangle  \bigg)\Bigg|^2\Bigg] \le \frac{k(k-1)^2}{(n-1)^2}M. \label{pf:dyn-crude}
\end{align}
Use this along with the trivial bound $(n-k)/(n-1) \le 1$ in \eqref{def:init-entropybound-dyn} to get
\begin{align}
\frac{d}{dt}H^k_t \le &\frac{k(k-1)^2}{(n-1)^2}M + k\E\left[ \left|\E\big[b(t,X^1,X^n) \,|\, \F^k_t \big] - \langle \mu,b(t,X^1,\cdot)\rangle \right|^2 \right]. \label{pf:dyn1}
\end{align}
To estimate the last term, let $P^{(k+1|k)}_{X^1,\ldots,X^k}[t](dx^{k+1})$ denote a version of the regular conditional law of $X^n|_{[0,t]}$ given $(X^1,\ldots,X^k)|_{[0,t]}$. The equivalent form \eqref{asmp:dyn-transp-equiv} of the assumed transport-type inequality yields
\begin{align*}
\big|\E\big[b(t,X^1,X^n) \,|\, \F^k_t \big] - \langle \mu,b(t,X^1,\cdot)\rangle \big|^2 &= \big| \langle P^{(k+1|k)}_{X^1,\ldots,X^k}[t] - \mu[t],b(t,X^1,\cdot)\rangle \big|^2 \\
	&\le \transpconst H\big( P^{(k+1|k)}_{X^1,\ldots,X^k}[t] \,\big|\, \mu[t] \big), \ \ a.s.
\end{align*}
Moreover, the chain rule for relative entropy \eqref{def:chainrule} implies
\begin{align*}
\E \, H\big( P^{(k+1|k)}_{X^1,\ldots,X^k}[t] \,\big|\, \mu[t] \big) = H^{k+1}_t - H^k_t.
\end{align*}
(Indeed, apply \eqref{def:chainrule} by identifying $m^1 \to P^{(k)}[t]$, $m^2 \to \mu^{\otimes k}[t]$, $K^1_x \to P^{(k+1|k)}_{x^1,\ldots,x^k}[t]$, and $K^2_x \to \mu[t]$.)
Returning to \eqref{pf:dyn1}, we thus deduce
\begin{align*}
\frac{d}{dt}H^k_t \le \frac{k(k-1)^2}{(n-1)^2}M + \transpconst k \big(H^{k+1}_t - H^k_t \big).
\end{align*}
By Gronwall's inequality, we have
\begin{align*}
H^k_t &\le e^{-\transpconst k t} H^k_0 + \int_0^t e^{-\transpconst k (t-s)}\left(\frac{k(k-1)^2}{(n-1)^2}M + \transpconst k H^{k+1}_s \right) ds.
\end{align*}
Now, if we iterate this inequality $n-k$ times, and use in the last iteration the fact that $t \mapsto H_t^n$ is non-decreasing by \eqref{ineq:dataproc1}, we find 
\begin{align}
H^k_T &\le \sum_{\ell=k}^{n-1} \left[B_k^\ell(T)H^\ell_0  +  \frac{(\ell-1)^2 M}{\transpconst (n-1)^2 } A_k^\ell(T)\right] + A_k^{n-1}(T) H^n_T, \label{pf:dyn-est1}
\end{align}
where we define
\begin{align}
\begin{split}
A_k^\ell(t_k) &:= \bigg(\prod_{j=k}^{\ell} \transpconst j \bigg) \int_0^{t_k}\int_0^{t_{k+1}}\cdots\int_0^{t_{\ell}} e^{- \sum_{j=k}^{\ell} \transpconst j(t_j - t_{j+1}) } dt_{\ell+1}\cdots dt_{k+2}dt_{k+1}, \\
B_k^\ell(t_k) &= \bigg(\prod_{j=k}^{\ell-1}  \transpconst j \bigg) \int_0^{t_k}\int_0^{t_{k+1}}\cdots\int_0^{t_{\ell-1}} e^{ -  \transpconst \ell t_{\ell} - \sum_{j=k}^{\ell-1}  \transpconst j (t_j - t_{j+1}) } dt_{\ell}\cdots dt_{k+2}dt_{k+1},
\end{split} \label{def:ABterms}
\end{align}
for $t_k \ge 0$ and $n \ge \ell \ge k \ge 1$, with $B_k^{k}(t) := e^{-\transpconst k t}$.

The assumption \eqref{asmp:dyn-init} along with \eqref{def:init-entropybound-dyn2} together imply $H^n_T \le C_0 + MnT$. Hence,
\begin{align}
H^k_T &\le \frac{C_0}{n^2}\sum_{\ell=k}^{n-1} \ell^2 B_k^\ell(T) +  \frac{ M}{\transpconst (n-1)^2 } \sum_{\ell=k}^{n-1} (\ell-1)^2 A_k^\ell(T) + A_k^{n-1}(T)(C_0 + MnT ), \label{pf:dyn-est2}
\end{align}
This is essentially the sharpest estimate we have obtained, and the rest of the proof is devoted to its simplification. 
In particular, a decent amount of work is required to simplify and estimate the quantities $A_k^\ell(t)$ and $B_k^\ell(t)$, and the following lemma summarizes the computations needed for the proof at hand.

\begin{lemma} \label{le:AB-estimate}
For integers $\ell > k \ge 1$ and real numbers $t > 0$, we have
\begin{align}
A_k^\ell(t) &\le \exp\left( -2 (\ell+1) \left( e^{-\transpconst t} - \frac{k}{\ell+1} \right)_+^2\right),  \label{est:Atail}
\end{align}
Moreover, for integers $r \ge 0$, 
\begin{align}
\sum_{\ell=k}^\infty \ell^r A_k^\ell(t) &\le \frac{(k+r)!}{(k-1)!} \frac{e^{\transpconst (r+1) t} - 1}{r+1}, \label{est:Asum} \\
\sum_{\ell=k}^\infty \ell^2 B_k^\ell(t) &\le k(k+1) e^{ 2\transpconst t} \le 2k^2e^{2\transpconst t}. \label{est:Bsum}
\end{align}
\end{lemma}

This follows from the more general Proposition \ref{pr:ABestimate}, proven in Section \ref{se:ABestimates}, applied with $a=\transpconst k$ and $b=\transpconst$.
Next, apply \eqref{est:Asum} with $r=2$ to get
\begin{align*}
\sum_{\ell=k}^{n-1}  (\ell-1)^2 A_k^\ell(T) &\le \sum_{\ell=k}^\infty  \ell^2 A_k^\ell(T) \le \frac13 k(k+1)(k+2) (e^{3\transpconst T} - 1) \\
	&\le \frac53 k^3 (e^{3\transpconst T} - 1) \le 5 \transpconst T k^3 e^{3\transpconst T}.
\end{align*}
Noting also that $(\ell-1)/(n-1) \le \ell/n$, we thus estimate the second term of \eqref{pf:dyn-est2}  by $5MT\frac{k^3}{ n^2}e^{3\transpconst T}$.
Lastly, using \eqref{est:Atail} to estimate the third term of  \eqref{pf:dyn-est2} and putting it all together, 
\begin{align}
H^k_T &\le \frac{1}{n^2}e^{3\transpconst T}( 2C_0k^2  +  5MTk^3 ) + ( C_0 + MnT ) \exp\left( -2 n \left( e^{-\transpconst T} - \frac{k}{n} \right)_+^2\right). \label{pf:dyn-est3}
\end{align}

This is not the desired result, because it is $O(k^3/n^2)$ instead of the announced $O(k^2/n^2)$. The idea at this point is to use \eqref{pf:dyn-est3} in order to improve the crude estimate of \eqref{pf:dyn-crude}.
Specifically, applying \eqref{pf:dyn-est3} with $k=3$, we find
\begin{align*}
H^3_T \le \frac{9}{n^2}e^{3\transpconst T}( 2C_0  + 15MT ) + ( C_0 +  MnT ) \exp\left( -2 n \left( e^{-\transpconst T} - \frac{3}{n} \right)_+^2\right).
\end{align*} 
In addition, for $n \ge 6 e^{\transpconst T}$ we have 
\begin{align*}
n \exp\bigg( -2 n \bigg( e^{-\transpconst T} - \frac{3}{n} \bigg)_+^2\bigg) &\le n\exp\bigg( -\frac12 n e^{- 2\transpconst T}\bigg) \le \frac{48 e^{6 \transpconst T}}{n^2},
\end{align*}
with the last step using the inequality $e^{-x} \le 6/x^3$ for $x > 0$. Hence, 
\begin{align}
H^3_T \le \frac{\widehat{C}_0}{n^2}, \quad \text{where} \quad \widehat{C}_0 &:= 9 e^{3\transpconst T}( 2C_0  + 15 MT ) + 48(C_0 +  MT ) e^{6 \transpconst T} \label{pf:dyn-est4} \\
	&\ \le 183(C_0+MT)e^{6\transpconst T}. \nonumber
\end{align}
We may now improve \eqref{pf:dyn-crude} as follows. Instead of using convexity of $x \mapsto x^2$, we expand the square and use exchangeability to get the exact identity
\begin{align*}
&\frac{k}{(n-1)^2}\E\Bigg[ \Bigg|\sum_{j =2}^k \bigg(b(t,X^1,X^j) - \langle \mu,b(t,X^1,\cdot)\rangle  \bigg)\Bigg|^2\Bigg] =  \\
&\quad = \frac{k(k-1)(k-2)}{(n-1)^2} \E\left[\left(b(t,X^1,X^2) - \langle \mu,b(t,X^1,\cdot)\rangle\right) \cdot \left(b(t,X^1,X^3) - \langle \mu,b(t,X^1,\cdot)\rangle\right)\right] \\
&\quad\qquad +  M\frac{k(k-1)}{(n-1)^2}.
\end{align*}
To handle the first term, let $P^{(3|2)}_{X^1,X^2}$ denote a version of the conditional law of $X^3$ given $(X^2,X^1)$. Condition on $(X^1,X^2)$ and use Cauchy-Schwarz followed by the assumption \eqref{asmp:dyn-transp} to get
\begin{align*}
\E &\left[ \left(b(t,X^1,X^2) - \langle \mu,b(t,X^1,\cdot)\rangle\right) \cdot \left(b(t,X^1,X^3) - \langle \mu,b(t,X^1,\cdot)\rangle\right)\right] \\
	&= \E \left[\left(b(t,X^1,X^2) - \langle \mu,b(t,X^1,\cdot)\rangle\right) \cdot \langle P^{(3|2)}_{X^1,X^2} - \mu,b(t,X^1,\cdot)\rangle \right] \\
	&\le M^{1/2} \E \left[ \big|\langle P^{(3|2)}_{X^1,X^2} - \mu,b(t,X^1,\cdot)\rangle \big|^2\right]^{1/2} \\
	&\le M^{1/2} \left(\transpconst \E \left[ H\big(P^{(3|2)}_{X^1,X^2}\,\big|\,\mu \big) \right] \right)^{1/2}.
\end{align*}
Using the inequality \eqref{def:chainrule2} coming from the chain rule for relative entropy, along with the estimate $H^3_T \le \widehat{C}_0/n^2$ from \eqref{pf:dyn-est4}, we bound this further by
\begin{align*}
M^{1/2}\left(\transpconst  H^3_T \right)^{1/2} \le \frac{1}{n}\sqrt{M \transpconst  \widehat{C}_0}.
\end{align*}
Putting it together, we have thus shown the following bound:
\begin{align}
\frac{k}{(n-1)^2}&\E\Bigg[ \Bigg|\sum_{j =2}^k \bigg(b(t,X^1,X^j) - \langle \mu,b(t,X^1,\cdot)\rangle  \bigg)\Bigg|^2\Bigg] \nonumber \\
	&\le M\frac{k(k-1)}{(n-1)^2} + \sqrt{M \transpconst \widehat{C}_0}\frac{k(k-1)(k-2)}{ n(n-1)^2} \nonumber \\
	&\le \left(2M + \sqrt{M \transpconst \widehat{C}_0} \right)\frac{k^2}{n^2}, \label{pf:dyn-est5}
\end{align}
with the last step using the simple inequalities $n \le 2(n-1)$ and $(k-1)/(n-1) \le k/n$.
This replaces the first term on the right-hand side of \eqref{pf:dyn1}. The second term remains the same and is bounded as before by $\transpconst k (H^{k+1}_t-H^k_t)$, which yields 
\begin{align*}
\frac{d}{dt}H^k_t &\le \left(2M + \sqrt{M \transpconst \widehat{C}_0} \right)\frac{k^2}{ n^2} + \transpconst k (H^{k+1}_t-H^k_t).
\end{align*}
Apply Gronwall's inequality and iterate to get the following variant of \eqref{pf:dyn-est1}:
\begin{align}
H^k_T &\le \sum_{\ell=k}^{n-1} \left[B_k^\ell(T)\frac{C_0\ell^2}{n^2}  +  \left(2M + \sqrt{M \transpconst \widehat{C}_0} \right)\frac{\ell }{\transpconst n^2 } A_k^\ell(T)\right] + A_k^{n-1}(T) H^n_T. \label{pf:dyn-est1'}
\end{align}
The point is that now we have merely $(\ell-1)$ in the second term, rather than $(\ell-1)^2$.

We handle the first term of \eqref{pf:dyn-est1'} again using \eqref{est:Bsum}.
To handle the second term of \eqref{pf:dyn-est1'}, we apply \eqref{est:Asum} with $r=1$ now (instead of $r=2$ as above) to get
\begin{align}
\sum_{\ell=k}^{n-1}  \ell A_k^\ell(T) &\le  \frac12 k(k+1) (e^{2\transpconst T} - 1) \le k^2 (e^{2\transpconst T} - 1)  \le 2\transpconst T k^2 e^{2\transpconst T}.  \label{pf:dyn-Abound1}
\end{align}
To handle the final term of \eqref{pf:dyn-est1'}, apply \eqref{def:init-entropybound-dyn} for $k=n$, noting that the second term vanishes and using \eqref{pf:dyn-est5} to estimate the first term:
\begin{align*}
H^n_T &\le H^n_0 + T\left(2M + \sqrt{M \transpconst \widehat{C}_0}\right) \le C_0 + T\left(2M + \sqrt{M \transpconst \widehat{C}_0}\right),
\end{align*}
where the last step used the assumption \eqref{asmp:dyn-init}.
Putting it together, from \eqref{pf:dyn-est1'}  we get
\begin{align}
H^k_T &\le  \frac{2C_0 k^2}{n^2} e^{2 \transpconst T}  +  2 T\left(2M + \sqrt{M \transpconst \widehat{C}_0} \right)\frac{k^2}{ n^2} e^{2\transpconst T} + A_k^{n-1}(T) \left(C_0 +  2MT + T\sqrt{M \transpconst \widehat{C}_0}\right) \nonumber \\
	&\le  \left(2C_0 + 2T\left(2M + \sqrt{M \transpconst \widehat{C}_0} \right) \right)e^{2\transpconst T }\frac{k^2}{n^2} \nonumber \\
	&\qquad + \left(C_0 + 2MT + T\sqrt{M \transpconst \widehat{C}_0}\right)\exp\left( -2 n \left( e^{-\transpconst T} - \frac{k}{n} \right)_+^2\right). \label{pf:dyn-est2'}
\end{align}
Finally, we simplify the form of the constants in \eqref{pf:dyn-est2'}. For the second, use Young's inequality:
\begin{align*}
C_0 +  2MT  + T\sqrt{M \transpconst \widehat{C}_0} &\le C_0 + (2 + 8\transpconst)MT + \frac{1}{32}\widehat{C}_0 \\
	&\le C_0 + (2+8\transpconst) MT + \frac{183}{32}(C_0+MT)e^{6\transpconst T} \\
	&\le e^{6 \transpconst T}\left( 7 C_0 +  8\transpconst M T + 8 M T \right) \le C.
\end{align*}
The first constant is  bounded similarly:
\begin{align*}
2C_0 + 2 T \left(2M + \sqrt{M \transpconst \widehat{C}_0} \right) &\le 2C_0 +  2 T \left(2M +  8 \transpconst M  + \frac{1}{32}  \widehat{C}_0 \right) \\
	&\le 2C_0 + 4 M T + 16 \transpconst M T  + \frac{183}{16}(C_0+MT)e^{6\transpconst T} \\
	&\le e^{6 \transpconst T}\left(14 C_0 + 16 MT + 16 \transpconst M T \right) \le 2C.
\end{align*}
{\ } \vskip-1.06cm \hfill\qedsymbol

\begin{remark}
In estimating the second term of \eqref{pf:dyn1}, it is natural to try truncating $b$ to be bounded, so that Pinsker's inequality may be used in the subsequent estimate, without any need for the transport-type inequality assumed in Theorem \ref{th:intro-dynamic}. But this does not seem to work: Assuming merely the uniform integrability condition like
\begin{align*}
M(r) &:= \sup_{n \in \N}\sup_{t \in [0,T]}\int_{\C_T^d \times \C_T^d} |b(t,\cdot)|^21_{\{|b(t,\cdot)|^2 > r\}} \,d(P^{(2)} + P^{(1)} \times \mu) \to 0, \ \ \text{ as } r \to \infty,
\end{align*}
one might hope to remove any truncation error after first sending $n\to\infty$ (say, for fixed $k$). Keeping track of the extra error term, the following estimate can be shown for each $r > 0$:
\begin{align*}
H^{k}_T &\le \sum_{\ell=k}^{n-1} \left[B_k^\ell(T)H^{\ell}_0  +  \frac{(\ell-1)^2 M}{r (n-1)^2 } A_k^\ell(T) + \frac{M(r)}{r}  A_k^\ell(T)\right] + A_k^{n-1}(T) H^{n}_T,
\end{align*}
where $A_k^\ell$ and $B_k^\ell$ are defined in exactly the same manner except with $\transpconst = 3r$. For fixed $r$, all but the second to last term will vanish as $n \to \infty$. But $\sum_{\ell=k}^{n-1}A_k^\ell(T)$ is of order $e^{3rT}$, and we will thus need $M(r)e^{3rT} \to 0$ as $r\to\infty$ in order to get this remaining term to vanish. This requires a subexponential tail bound for $|b|^2$, which essentially leads back to an exponential integrability assumption like in Lemma \ref{le:integ-transp-equiv}(b).
\end{remark}

\subsection{Proof of Theorem \ref{th:intro-dynamic-rev}}  \label{se:proofs:rev}

Throughout this section, we define
\begin{align*}
H^k_t := H(\mu^{\otimes k}[t] \,|\, P^{(k)}[t]),
\end{align*}
for $1 \le k \le n$ and $t \in [0,T]$. Let $Y^1,\ldots,Y^n$ denote i.i.d.\ processes with law $\mu$. Thus, for some independent Brownian motions $W^1,\ldots,W^n$, we have
\begin{align*}
dY^i_t = \left(b_0(t,Y^i) + \langle \mu,b(t,Y^i,\cdot)\rangle \right)dt + dW^i_t, \qquad i=1,\ldots,n.
\end{align*}
Define the conditional expectation functions $\widehat{b}^k_i$ exactly as in \eqref{def:bhat-dyn}. We begin again with an entropy estimate:

\begin{lemma} \label{le:init-entropybound-dyn-rev}
Grant the assumptions of Theorem \ref{th:intro-dynamic-rev}, and let $1 \le k \le n$. 
Then $t \mapsto H^k_t$ is absolutely continuous, and for a.e.\ $t \in [0,T]$ we have
\begin{align}
\begin{split}
\frac{d}{dt}H^k_t \le &\frac{k(k-1)}{(n-1)^2} \transpconst  + \frac{k(n-k)^2}{(n-1)^2}\E\left[ \left|\widehat{b}^k_1(t,Y^1,\ldots,Y^k) - \langle \mu,b(t,Y^1,\cdot)\rangle \right|^2\right],
\end{split} \label{def:init-entropybound-dyn-rev}
\end{align}
where we recall that $\transpconst:=2\||b|^2\|_\infty$.
Moreover, 
\begin{align}
H^k_t \le C_0 + \transpconst T. \label{def:init-entropybound-dyn2-rev}
\end{align}
\end{lemma}
\begin{proof}
We follow very closely the proof of Lemma \ref{le:init-entropybound-dyn}. To justify an application of Lemma \ref{le:entropy-diffusions}, we must be careful to check is that the SDE \eqref{pf:def:SDE-Pk} satisfied by the reference measure $P^{(k)}$ is well-posed. This is the SDE$(\widetilde{b})$, where, for each $i=1,\ldots,k$, the function $\widetilde{b} : [0,T] \times (\C_T^d)^k \to (\R^d)^k$ has $i^\text{th}$ component given by
\begin{align*}
(t,x^1,\ldots,x^k) \mapsto b_0(t,x^i) + \frac{1}{n-1}\sum_{j \neq i, \, j \le k}b(t,x^i,x^j) + \frac{n-k}{n-1}\widehat{b}^k_i(t,x^1,\ldots,x^k).
\end{align*}
We know by assumption that the SDE($b_0^{\otimes k}$) is well-posed. Since $b$ is bounded by assumption, the well-posedness of the SDE($\widetilde{b}$) follows from a standard application of Girsanov's theorem.

Following the first part of the proof of Lemma \ref{le:init-entropybound-dyn}, but with $Y$ in place of $X$,
we deduce the claimed absolute continuity from Lemma \ref{le:entropy-diffusions}(ii), along with the a.e.\ inequality
\begin{align*}
\begin{split}
\frac{d}{dt}H^k_t \le &\frac{k}{(n-1)^2}\E\Bigg[ \Bigg|\sum_{j =2}^k \big(b(t,Y^1,Y^j) - \langle \mu,b(t,Y^1,\cdot)\rangle  \big)\Bigg|^2\Bigg]   \\
&\quad + \frac{k(n-k)^2}{(n-1)^2}\E\left[ \left| \widehat{b}^k_1(t,Y^1,\ldots,Y^k) - \langle \mu,b(t,Y^1,\cdot)\rangle \right|^2\right].
\end{split} 
\end{align*}
Expanding the square in the first term, the off-diagonal terms vanish because $Y^1,\ldots,Y^k$ are i.i.d.\ with law $\mu$. 
Note for $j \neq 1$ that
\begin{align*}
\E\left[\big| b(t,Y^1,Y^j) - \langle \mu,b(t,Y^1,\cdot)\rangle  \big|^2\right] \le \||b|^2\|_\infty \le 2\||b|^2\|_\infty = \transpconst.
\end{align*}
 This yields \eqref{def:init-entropybound-dyn-rev}. Take $k=n$ in \eqref{def:init-entropybound-dyn-rev}, noting that the second term vanishes, to get $(d/dt)H^n_t \le \transpconst$ and thus $H^n_T \le H^n_0 + \transpconst T$. Assumption \eqref{asmp:dyn-init-rev} yields $H^n_0 \le C_0$, and the fact that $H^k_t$ is increasing in both $k$ and $t$ gives \eqref{def:init-entropybound-dyn2-rev}.
\end{proof}

We are now ready for the main line of the proof of Theorem \ref{th:intro-dynamic-rev}, which is similar to that of Theorem \ref{th:intro-dynamic}. We begin from \eqref{def:init-entropybound-dyn-rev}.
To estimate the second term therein, recall that $P^{(k+1|k)}_{X^1,\ldots,X^k}[t](dx^{k+1})$ is a version of the regular conditional law of $X^n|_{[0,t]}$ given $(X^1,\ldots,X^k)|_{[0,t]}$. The definition of $\widehat{b}^k_i$ from \eqref{def:bhat-dyn} and Pinsker's inequality (in the multivariate form of \eqref{ineq:Pinsker}) yield
\begin{align*}
\big|\widehat{b}^k_1(t,Y^1,\ldots,Y^k) - \langle \mu,b(t,Y^1,\cdot)\rangle \big|^2 &= | \langle P^{(k+1|k)}_{Y^1,\ldots,Y^k}[t] - \mu[t],b(t,Y^1,\cdot)\rangle|^2 \\
	&\le \transpconst H\big( \mu[t]  \,\big|\, P^{(k+1|k)}_{Y^1,\ldots,Y^k}[t] \big), \ \ a.s.,
\end{align*}
where we recall $\transpconst:=2\||b|^2\|_\infty$.
Moreover, the chain rule for relative entropy \eqref{def:chainrule} implies
\begin{align*}
\E \, H\big( \mu[t] \,\big|\, P^{(k+1|k)}_{Y^1,\ldots,Y^k}[t] \big) = H^{k+1}_t - H^k_t.
\end{align*}
(Indeed, apply \eqref{def:chainrule} by identifying $m^1 \to \mu^{\otimes k}[t]$, $m^2 \to P^{(k)}[t]$, and $K^1_x \to \mu[t]$, and $K^2_x \to P^{(k+1|k)}_{x^1,\ldots,x^k}[t]$.)
Apply this in \eqref{def:init-entropybound-dyn-rev} to get, for a.e.\ $t \in [0,T]$,
\begin{align*}
\frac{d}{dt}H^k_t \le \transpconst k \frac{ k-1}{(n-1)^2} +  \transpconst k \big(H^{k+1}_t - H^k_t \big).
\end{align*}
By Gronwall's inequality, and using $(k-1)/(n-1) \le k/n$ and $n \le 2(n-1)$, we have 
\begin{align*}
H^k_t &\le e^{-\transpconst k t} H^k_0 + \transpconst k  \int_0^t e^{-\transpconst k (t-s)}\left(2\frac{k}{n^2}  + H^{k+1}_s \right) ds.
\end{align*}
Iterate this inequality $n-k$ times  to get
\begin{align}
H^k_T &\le \sum_{\ell=k}^{n-1} \left[B_k^\ell(T)H^\ell_0  +  2\frac{ \ell }{ n^2} A_k^\ell(T)\right] + A_k^{n-1}(T)H^n_T, \label{pf:dyn-est1-rev}
\end{align}
where $A_k^\ell(T)$ and $B_k^\ell(T)$ are defined as in \eqref{def:ABterms}.

We handle the first term of \eqref{pf:dyn-est1-rev} using \eqref{est:Bsum}, which gave $\sum_{\ell=k}^{n-1} \ell^2 B_k^\ell(T) \le 2 k^2 e^{2 \transpconst T }$.
We handle the second term of \eqref{pf:dyn-est1-rev} as in \eqref{pf:dyn-Abound1} to get  $\sum_{\ell=k}^{n-1}  \ell A_k^\ell(T) \le 2\transpconst T \frac{k^2}{n^2} e^{2\transpconst T}$.
For the last term in \eqref{pf:dyn-est1-rev}, use \eqref{def:init-entropybound-dyn2-rev} to get $H^n_T \le C_0 + \transpconst T$.
Putting it together, from \eqref{pf:dyn-est1-rev} we deduce
\begin{align*}
H^k_T &\le 2C_0 \frac{k^2}{n^2}e^{2 \transpconst T} +  4\transpconst T \frac{k^2}{ n^2}e^{2\transpconst T} + A_k^{n-1}(T)(C_0+ \transpconst T ).
\end{align*}
Finally, use the estimate on $A_k^{n-1}(T)$ from \eqref{est:Atail} to get
\begin{align*}
H^k_T &\le 2(C_0 + 2\transpconst T) e^{2\transpconst T} \frac{k^2}{n^2} + (C_0+ \transpconst T ) \exp\left( -2 n\left( e^{-\transpconst T} - \frac{k}{n} \right)_+^2\right).
\end{align*}
{\ } \vskip-1.15cm \hfill\qedsymbol

{ \ }

\begin{remark} \label{re:relaxingboundedness}
The assumption that $b$ is bounded in Theorem \ref{th:intro-dynamic-rev} could likely be generalized to the case where the transport-type inequality
\begin{align}
\big|\big\langle \mu[t] - P^{(k+1|k)}_{x^1,\ldots,x^k}[t], b(t,x,\cdot)\big\rangle \big|^2 \le  \transpconst H(\mu[t]\,|\,P^{(k+1|k)}_{x^1,\ldots,x^k}[t]), \qquad \forall t,x,\nu,x^1,\ldots,x^k, \label{asmp:dyn-transp-rev}
\end{align}
holds for some $\transpconst < \infty$, as well as square-integrability as  in \eqref{asmp:dyn-moment} but with $\mu^{\otimes 2}$ in place of $P^{(2)}$.
This condition \eqref{asmp:dyn-transp-rev}, however, seems difficult to check in practice.
In addition, there is a technical subtlety here which is not present in the setting of Theorem \ref{th:intro-dynamic}. Our use of Girsanov's theorem to estimate relative entropy (see Lemma \ref{le:entropy-diffusions}) requires that the SDE satisfied by the \emph{reference measure} is well-posed. Here, $P^{(k)}$ is the reference measure, and its SDE involves the conditional expectation functions $\widehat{b}^k_i$ defined in \eqref{def:bhat-dyn}, which makes it is difficult to identify natural sufficient conditions for well-posedness other than boundedness.
\end{remark}

\section{Estimates of iterated exponential integrals}  \label{se:ABestimates}

Fix $a,b > 0$ throughout this section.
For $t_0 \ge 0$ and $n \ge \ell \ge 1$, define 
\begin{align*}
A_\ell(t_0) &:= \bigg(\prod_{j=0}^{\ell}(a+bj) \bigg) \int_0^{t_0}\int_0^{t_1}\cdots\int_0^{t_\ell} e^{- \sum_{j=0}^{\ell} (a+bj)(t_j - t_{j+1}) } dt_{\ell+1}\cdots dt_{2}dt_{1}, \\
B_\ell(t_0) &= \bigg(\prod_{j=0}^{\ell-1}(a+bj) \bigg) \int_0^{t_0}\int_0^{t_{1}}\cdots\int_0^{t_{\ell-1}} e^{ - (a+b\ell)t_{\ell} - \sum_{j=0}^{\ell-1} (a+bj)(t_j - t_{j+1}) } dt_{\ell}\cdots dt_{2}dt_{1},
\end{align*}
with the convention that $B_0(t) := e^{-at}$.
Iterated integrals of this form appeared in the arguments in Section \ref{se:proofs-dyn} with $a=\transpconst k$ and $b=\transpconst$ for $k \in \N$ and $\transpconst > 0$, and they will appear again in Section \ref{se:proofs-inf} with $a=\transpconst k$ and $b=\transpconst \bar\ell$ for $\bar\ell \in \N$. Recall that $x_+:=\max(x,0)$.

\begin{proposition} \label{pr:ABestimate}
Let $t > 0$.
For integers $\ell \ge 0$ we have
\begin{align*}
A_\ell(t) &\le \exp\left(-2(\ell+(a/b)+2)\left(e^{-bt} - \frac{(a/b)}{\ell+(a/b)+1} \right)_+^2 \right).
\end{align*}
For integers $p \ge 0$ we have
\begin{align*}
\sum_{\ell=0}^\infty ((a/b)+\ell)^p A_\ell(t) &\le \frac{e^{b(1+p)t} - 1}{1+p} \prod_{i=0}^p((a/b)+i), \\
\sum_{\ell=0}^\infty ((a/b)+\ell)^{1+p} B_\ell(t) &\le be^{b(1+p)t}\prod_{i=0}^p ((a/b)+i).
\end{align*}
\end{proposition}

To prove this, we rewrite $A_\ell(t)$ as a tail probability involving a Beta distribution (Section \ref{se:transform-to-beta}) and exploit a subgaussian estimate (Section \ref{se:beta-subgaussian}). With these preparations, the remaining calculations are given in Section \ref{se:pf:ABest}.

\subsection{Transformation to a beta distribution} \label{se:transform-to-beta}

The starting point is the following observation. Abbreviate $c_j = a+bj$ for $j \ge 0$, and define the exponential density functions $h_j(t) := c_j e^{-c_jt}1_{[0,\infty)}(t)$. Then $A_\ell$ and $B_\ell$ can be expressed in terms of convolutions:
\begin{align*}
A_\ell(t) &= \int_0^t  h_0 * h_1 * \cdots  * h_\ell(s)\, ds, \qquad B_\ell(t) = \frac{1}{c_\ell} h_0 * h_1 * \cdots  * h_\ell(t).
\end{align*}
To express this in probabilistic terms, let $Z_j$ be an exponential random variable with parameter $c_j$, with $(Z_j)_{j \in \N}$ independent.
Then
\begin{align}
A_\ell(t) &= \PP\bigg( \sum_{j=0}^\ell Z_j \le t\bigg), \qquad B_\ell(t) = \frac{1}{a + b \ell}\frac{d}{dt}A_\ell(t). \label{eq:AB-exp-representation}
\end{align}
Recall in the following that for $\alpha,\beta > 0$ the distribution Beta($\alpha$,$\beta$) on $[0,1]$ is described by the density function
\begin{align*}
\frac{\Gamma(\alpha+\beta)}{\Gamma(\alpha)\Gamma(\beta)} x^{\alpha-1} (1-x)^{\beta-1}, \qquad x \in [0,1],
\end{align*}
where $\Gamma(z)=\int_0^\infty u^{z-1}e^{-u}\,du$ for $z > 0$ denotes the Gamma function.

\begin{lemma} \label{le:Atransform-Beta}
Let $t \ge 0$, and let $\ell \ge 0$ be an integer.
Let $Y \sim \mathrm{Beta}(a/b, \ell+1)$. Then $A_\ell(t) = \PP(Y \ge e^{-bt})$. That is,
\begin{align}
A_\ell(t) &= \frac{ \Gamma(\ell+(a/b)+1)}{\ell!\, \Gamma(a/b)} \int_{e^{-bt}}^1 x^{(a/b)-1} (1 - x)^\ell dx. \label{eq:A-beta-expr}
\end{align}
\end{lemma}
\begin{proof}
A computation detailed in \cite[Theorem 2.1]{akkouchi2008convolution} gives
\begin{align*}
h_0 * \cdots * h_\ell(t) &= \sum_{j=0}^\ell \frac{\prod_{i=0}^{\ell} c_i}{\prod_{i=0,\,i \neq j}^{\ell}(c_i-c_j)}e^{-c_j t}.
\end{align*}
Note for each $j$ that
\begin{align*}
\prod_{i=0,\,i \neq j}^{\ell}(c_i-c_j) &=  b^{\ell} \prod_{i=0,\,i \neq j}^{\ell} (i-j) = b^{\ell}(-1)^{j} j! (\ell-j)!,
\end{align*}
and the identity $\Gamma(z+1)=z\Gamma(z)$ yields
\begin{align*}
\prod_{i=0}^{\ell} c_i &= \prod_{i=0}^{\ell} (a+bi) = b^{\ell+1} \prod_{i=0}^{\ell} ( (a/b) + i) = b^{\ell+1} \frac{\Gamma(\ell+(a/b)+1)}{\Gamma(a/b)}.
\end{align*}
We thus find 
\begin{align*}
h_0 * \cdots * h_\ell(t) &= b \frac{ \Gamma(\ell+(a/b)+1)}{ \Gamma(a/b)}\sum_{j=0}^\ell \frac{1}{j! (\ell-j)!}  (-1)^{j} e^{-(a+bj) t} \\
	&= be^{-a t}\frac{ \Gamma(\ell+(a/b)+1)}{\ell!\, \Gamma(a/b)}\sum_{j=0}^\ell \binom{\ell}{j}  (-e^{-bt})^{j} \\
	&= b e^{-a t}\frac{ \Gamma(\ell+(a/b)+1)}{\ell!\, \Gamma(a/b)}(1-e^{-bt})^{\ell} .
\end{align*}
Changing variables via $x = e^{- b s}$,  
\begin{align*}
A_\ell(t) &= \int_0^t h_0 * \cdots * h_\ell(s) ds \\
	&= \frac{ \Gamma(\ell+(a/b)+1)}{\ell!\, \Gamma(a/b)} \int_0^t  b(e^{- bs})^{a/b} (1 - e^{-bs})^\ell ds \\
	&= \frac{ \Gamma(\ell+(a/b)+1)}{\ell!\, \Gamma(a/b)} \int_{e^{-bt}}^1 x^{(a/b)-1} (1 - x)^\ell dx.
\end{align*}
{\ } \vskip-.8cm
\end{proof}

\begin{remark}
In the case where $a/b=k$ is a positive integer, so that $c_j=b(k+j)$, an alternative proof of Lemma \ref{le:Atransform-Beta} is available using \emph{R\'enyi's representation} of exponential order statistics \cite{renyi1953theory}. Note that $c_j=b(j+k)$, so that $E_{j+k}:=b(j+k)Z_j$ are i.i.d.\ $\sim$ Exp(1). We deduce from \eqref{eq:AB-exp-representation} that $A_\ell(\cdot)$ is the cumulative distribution function of $\frac{1}{b} \sum_{j=k}^{\ell+k} \frac{E_j}{j}$. R\'enyi's result implies that $\sum_{j=k}^{\ell+k} \frac{E_j}{j}$ has the same law as the $k^\text{th}$ largest order statistic of a sample of $\ell+k$ independent Exp(1)'s. 
By a simple transformation, it follows that $A_\ell(t) = \PP(U_{(\ell+k,k)} \ge e^{-b t})$, where $U_{(\ell+k,k)}$ is the $k^\text{th}$ smallest order statistic of $\ell+k$ independent uniforms. Finally, it is well known that $U_{(\ell+k,k)} \sim$ Beta($k$,$\ell+1$).
\end{remark}

\subsection{A subgaussian estimate for the Beta distribution} \label{se:beta-subgaussian}

We next state a result from \cite{marchal2017sub}, which gives a tractable bound for the subgaussian constant of the Beta distribution, as well as the exact but less explicit value which we will not need.

\begin{lemma} \cite[Theorem 2.1]{marchal2017sub} \label{le:beta-tailbound}
Let $Y \sim \mathrm{Beta(}\alpha,\beta\mathrm{)}$. Then, for all $\lambda \in \R$,
\begin{align*}
\log\E\big[e^{\lambda(Y-\E[Y])}\big] &\le \frac{\lambda^2}{8(\alpha+\beta+1)}.
\end{align*}
Moreover, for all $t \in \R$,
\begin{align*}
\PP(Y > t) &\le \exp\left(- 2(\alpha+\beta+1)\left(t - \frac{\alpha}{\alpha+\beta}\right)_+^2\right).
\end{align*}
\end{lemma}
\begin{proof}
The first claim is due to \cite[Theorem 2.1]{marchal2017sub}. The second claim is trivial for $t \le \frac{\alpha}{\alpha+\beta} = \E[Y]$, and for $t > \E[Y]$ it follows from the first claim by a standard Chernoff bound.
\end{proof}

\subsection{Proof of Proposition \ref{pr:ABestimate}} \label{se:pf:ABest}

Combine Lemma \ref{le:Atransform-Beta} and the estimate of Lemma \ref{le:beta-tailbound} to get the first claim of Proposition \ref{pr:ABestimate}.
To estimate the summations, we will compute exactly the quantity
\begin{align*}
\sum_{\ell=0}^\infty \left(\prod_{i=1}^p((a/b)+\ell+i) \right) A_\ell(t).
\end{align*}
To this end, we first note two computations that will be useful along the way.
First,  the identity $\Gamma(z+1)=z\Gamma(z)$ implies
\begin{align}
\left(\prod_{i=1}^p(\ell+(a/b)+i) \right)\Gamma(\ell+(a/b)+1) &= \Gamma(\ell+(a/b)+p+1). \label{pf:gammaidentity2}
\end{align}
For $r \ge 0$ and $0 \le y < 1$, use the definition of the Gamma function and Fubini's theorem to get a second useful identity,
\begin{align}
\sum_{\ell=0}^\infty \frac{\Gamma(\ell+r+1)}{\ell!}y^\ell &= \int_0^\infty  \sum_{\ell=0}^\infty \frac{1}{\ell!} u^{\ell+r}e^{-u}y^\ell \, du \nonumber \\
	&= \int_0^\infty u^{r} e^{-(1-y)u} \, du \nonumber \\
	&= (1-y)^{-(1+r)}\int_0^\infty u^{r} e^{-u} \, du \nonumber \\
	&= (1-y)^{-(1+r)}\Gamma(r+1). \label{pf:gammaidentity1}
\end{align}

Now use Fubini's theorem along with the formulas \eqref{eq:A-beta-expr}, \eqref{pf:gammaidentity2}, and then \eqref{pf:gammaidentity1} with $r=p+(a/b)$ and $y=1-x$:
\begin{align*}
\sum_{\ell=0}^\infty  &\left(\prod_{i=1}^p(\ell+(a/b)+i) \right) A_\ell(t) \\
	&= \sum_{\ell=0}^\infty \left(\prod_{i=1}^p(\ell+(a/b)+i) \right) \frac{ \Gamma(\ell+(a/b)+1)}{\ell! \Gamma(a/b)} \int_{e^{-bt}}^1 x^{(a/b)-1} (1 - x)^\ell \,dx \\
	&= \sum_{\ell=0}^\infty \frac{ \Gamma(\ell+(a/b)+p+1)}{\ell! \Gamma(a/b)} \int_{e^{-bt}}^1 x^{(a/b)-1} (1 - x)^\ell \,dx \\
	&= \frac{\Gamma(1+p+(a/b))}{\Gamma(a/b)}\int_{e^{-bt}}^1 x^{(a/b)-1} x^{-(1+ p +(a/b))}\, dx \\
	&= \frac{\Gamma(1+p+(a/b))}{\Gamma(a/b)}\frac{e^{b(1+p)t} - 1}{1+p}.
\end{align*}
This immediately implies
\begin{align*}
\sum_{\ell=0}^\infty ((a/b)+\ell)^p A_\ell(t) &\le \frac{\Gamma(1+p+(a/b))}{\Gamma(a/b)}\frac{e^{b(1+p)t} - 1}{1+p} .
\end{align*}

We can perform similar computations with $B$ in place of $A$, by recalling that $B_\ell(t) = \frac{1}{a + b \ell}\frac{d}{dt}A_\ell(t)$. Applying the above, we get
\begin{align*}
\sum_{\ell=0}^\infty (a+b\ell)^{1+p} B_\ell(t) &\le \sum_{\ell=0}^\infty \left(\prod_{i=0}^{p}(\ell+(a/b)+i) \right) B_\ell(t) \\
	&= \frac{d}{dt}\sum_{\ell=0}^\infty \left(\prod_{i=1}^{p}(\ell+(a/b)+i) \right) A_\ell(t) \\
	&= \frac{d}{dt}\left(\frac{\Gamma(1+p+(a/b))}{\Gamma(a/b)}\frac{e^{b(1+p)t} - 1}{1+p}\right) \\
	&= \frac{\Gamma(1+p+(a/b))}{\Gamma(a/b)} be^{b(1+p)t}.
\end{align*}
To complete the proof, simplify the previous estimates by using the identity
\begin{align*}
\frac{\Gamma(1+p+(a/b))}{\Gamma(a/b)} &= \prod_{i=0}^p ((a/b)+i).    
\end{align*}

\section{Proofs of corollaries for pairwise interactions} \label{se:proofs:corollaries}

This section gives the proofs of Corollaries \ref{co:bounded} and \ref{co:Lipschitz-dyn} and Theorem \ref{th:sublinear}.
These all follow fairly quickly from Theorem \ref{th:intro-dynamic}, except for the claims about well-posedness of the McKean-Vlasov equations, deferred to Section \ref{se:wellposedness}.

We will make some use of the weighted Pinsker inequality of \cite[Theorem 2.1(ii)]{bolleyvillani}, both in this section and in Appendix \ref{ap:transp-proof}. It states that, for probability measures $\nu$ and $\nu'$ on a common measurable space, and for a measurable real-valued function $f$ thereon, 
\begin{align*}
\langle \nu-\nu',f\rangle^2 \le 2\left(1 + \log \int e^{f^2}\,d\nu'\right) H(\nu \,|\,\nu' ). 
\end{align*}
This extends to $\R^d$-valued functions $f$ via the same argument as in \eqref{ineq:Pinsker}:
\begin{align}
|\langle \nu-\nu',f\rangle|^2 \le 2\left(1 + \log \int e^{|f|^2}\,d\nu'\right) H(\nu \,|\,\nu' ). \label{ineq:weightedPinsker}
\end{align}

\subsection{Proof of Corollary \ref{co:bounded}} \label{se:proofs:bounded}

We first check the claimed well-posedness of \eqref{def:introSDE-nonMarkov} and \eqref{def:introMV-nonMarkov}.  A standard argument using Girsanov's theorem and boundedness of $b$ shows that the well-posedness of \eqref{def:introSDE-nonMarkov} is equivalent to that of the SDE($b_0^{\otimes n}$), which holds by assumption. Existence and uniqueness for the McKean-Vlasov equation \eqref{def:introMV-nonMarkov} follow from Proposition \ref{pr:MVwellposed-bounded}.

Corollary \ref{co:bounded} will follow from Theorem \ref{th:intro-dynamic} as soon as we check the assumptions (1--3) therein. The well-posedness assumption (1) follows from the boundedness of $b$ and the assumed well-posedness of the SDE($b_0^{\otimes k}$) for each $k$. Clearly the square-integrability assumption (2) holds with $M \le 2\||b|^2\|_\infty$. Finally, the transport inequality (3) holds with $\transpconst = 2\||b|^2\|_\infty$ by Pinsker's inequality \eqref{ineq:Pinsker}. \hfill \qedsymbol

\subsection{Proof of Corollary \ref{co:Lipschitz-dyn}}  \label{se:proofs:Lipschitz}

The well-posedness of the SDE \eqref{def:introSDE-nonMarkov} and the  McKean-Vlasov SDE \eqref{def:introMV-nonMarkov} are both classical under the Lipschitz assumption; see \cite[Theorem 1.1]{sznitman1991topics} for the latter, noting that the proof adapts without change to the path-dependent setting.
The Lipschitz assumption \eqref{asmp:Lipschitz-dyn} also implies that assumption (1) of Theorem \ref{th:intro-dynamic} holds, as the function $\overline{b}(t,x)$ defined therein is Lipschitz in $x$  (and thus so is $\overline{b}^{\otimes k}(t,x)$ for each $k$).

It remains to check assumptions (2) and (3) of Theorem \ref{th:intro-dynamic}.
First, the facts that $\overline{b}$ is $L$-Lipschitz and $\mu_0$ satisfies the quadratic transport inequality \eqref{def:dyn-init-T2H} together imply that $\mu$ itself satisfies a quadratic transport inequality, as is proven in Proposition \ref{pr:T2-randominit}. Precisely, we have
\begin{align}
\W_2^2(\nu,\mu) \le \eta H(\nu\,|\,\mu), \qquad \forall \nu \in \P(\C_T^d), \label{pf:dyn-T2H-1}
\end{align}
where $\eta := 3(\eta_0 \vee 2T)e^{3TL^2}$.
This lets us estimate the constant $\transpconst$. Indeed, since $b(t,x,\cdot)$ is $L$-Lipschitz, so is $u \cdot b(t,x,\cdot)$ for each unit vector $u \in \R^d$. The inequality $\W_1 \le \W_2$ combined with Kantorovich duality then implies that, for $\nu \in \P(\C_T^d)$,
\begin{align*}
\big|\langle \mu-\nu, b(t,x,\cdot)\rangle\big|^2 &= \sup_{u \in \R^d, \, |u|=1}\langle \mu-\nu, u \cdot b(t,x,\cdot)\rangle^2 \le L^2 \eta H(\nu\,|\,\mu).
\end{align*}
Hence, we may take $\transpconst = L^2 \eta$.

Next, we estimate the constant $M$. Let $(X^1,\ldots,X^n)$ denote the solution of the SDE \eqref{def:introSDE-nonMarkov}.
A classical argument yields a second moment bound: Start with
\begin{align*}
|X^i_t| &\le |X^i_0| + \int_0^t \bigg|b_0(s,X^i) + \frac{1}{n-1}\sum_{j \neq i} b(s,X^i,X^j)\bigg| \, ds +  |W^i_t| \\
	&\le |X^i_0| + L\int_0^t\bigg(\|X^i\|_s + \frac{1}{n-1}\sum_{j \neq i}\|X^j\|_s\bigg)ds + M_0 + |W^i_t|.
\end{align*}
Square both sides and average to get
\begin{align*}
\frac{1}{n}\sum_{i=1}^n \|X^i\|_t^2 \le 4|X^i_0|^2 + 16L^2T \int_0^t \frac{1}{n}\sum_{i=1}^n \|X^i\|_s^2\,ds + 4M_0^2 + 4 \frac{1}{n}\sum_{i=1}^n \|W^i\|_t^2.
\end{align*}
Apply Gronwall's inequality and take expectations to get
\begin{align*}
\E[\|X^1\|_t^2] &= \E\left[\frac{1}{n}\sum_{i=1}^n \|X^i\|_t^2\right] \le 4e^{16L^2T}\left( \E|X^1_0|^2 + M_0^2 + 4d T \right),
\end{align*}
where we used Doob's inequality to get $\E \|W^1\|_t^2 \le 4\E |W^1_T|^2 = 4dT$.
A similar argument for the McKean-Vlasov equation yields
\begin{align*}
\int_{\C_T^d} \|x\|_t^2\,\mu(dx) &\le 4e^{16L^2T}\left( \int_{\R^d} |x_0|^2\,\mu_0(dx) + M_0^2 + 4d T \right).
\end{align*}
Finally, note for all $(t,x,y)$ that
\begin{align*}
\big|b(t,x,y) - \langle \mu,b(t,x,\cdot)\rangle\big| &\le \int_{\C^d_T} | b(t,x,y)-b(t,x,y')|\,\mu(dy') \le L\int_{\C^d_T}  \|y-y'\|_t \,\mu(dy'),
\end{align*}
which implies
\begin{align*}
M &\le L^2\int_{\C^d_T}\int_{\C^d_T} \|y-y'\|_T^2 \,P^{(1)}(dy)\,\mu(dy') \\
	&\le 2L^2 \bigg(\int_{\C_T^d} \|x\|_t^2\,\mu(dx) + \E\|X^1\|_T^2\bigg) \\
	&\le 8L^2e^{16TL^2}\bigg( \int_{\R^d}|x|^2\,\mu_0(dx) + \E|X^1_0|^2 + 2M_0^2 + 8dT\bigg).
\end{align*}
With the constants $\transpconst$ and $M$ estimated, we may now apply Theorem \ref{th:intro-dynamic} to get the second claimed inequality. To get the first, we  simply use the fact that the transport inequality \eqref{pf:dyn-T2H-1} tensorizes (see \cite[Proposition 1.9]{gozlan-leonard}), in the sense that 
\begin{align*}
\W_2^2(\nu,\mu^{\otimes k}) \le \eta H(\nu\,|\,\mu^{\otimes k}), \qquad \forall \nu \in \P((\C^d_T)^k),
\end{align*}
and note that $\eta=\transpconst/L^2$. \hfill\qedsymbol

\subsection{Proof of Theorem \ref{th:sublinear}} \label{se:proofs:sublinear}

We first check the claimed well-posedness of \eqref{def:introSDE-nonMarkov} and \eqref{def:introMV-nonMarkov}. Existence for \eqref{def:introSDE-nonMarkov} follows from linear growth and Girsanov's theorem \cite[3.5.16]{karatzas-shreve}. Uniqueness also follows from Girsanov, e.g., using \cite[Theorem 7.7]{liptser-shiryaev}.
Well-posedness of the McKean-Vlasov equation follows from Proposition \ref{pr:MVwellposed-sublinear}. Moreover, Proposition \ref{pr:MVwellposed-sublinear} also ensures that condition (3) of Theorem \ref{th:intro-dynamic}  holds.

The entropy estimate of Theorem \ref{th:sublinear} will follow from Theorem \ref{th:intro-dynamic} as soon as we check the assumptions (1) and (2) therein. As above, the linear growth of $b$ yields the well-posedness assumption (1). We next check condition (2). The standard argument using linear growth and Gronwall's inequality, as in the proof of Corollary \ref{co:Lipschitz-dyn}, yields
\begin{align*}
\int_{\C_T^d} \|x\|_T^2\,P^{(1)}(dx) &\le 4e^{16K^2T}\left(\int_{\R^d}|x|^2\,P^{(1)}_0(dx) + K^2T^2 + 4dT\right), \\
\int_{\C_T^d} \|x\|_T^2\,\mu(dx) &\le 4e^{16K^2T}\left(\int_{\R^d}|x|^2\mu_0(dx) + K^2T^2 + 4dT\right).
\end{align*}
Next, note for all $(t,x,y)$ that 
\begin{align}
\big|b(t,x,y) - \langle \mu,b(t,x,\cdot)\rangle\big| &\le \int_{\C^d_T} | b(t,x,y)-b(t,x,y')|\,\mu(dy') \nonumber \\
	&\le K\bigg(1 + \|y\|_T + \int_{\C^d_T}  \|y'\|_T\, \mu(dy')\bigg), \label{pf:sublinear2}
\end{align}
which implies
\begin{align*}
M &\le 3K^2\bigg(1 + \int_{\C^d_T}  \|y\|_T^2 \, P^{(1)}(dy) + \int_{\C^d_T}  \|y\|_T^2 \, \mu(dy)\bigg) \\
	&\le 12K^2e^{16K^2T}\bigg( 1 +  \int_{\R^d}|x|^2\,P^{(1)}_0(dx) + \int_{\R^d}|x|^2\,\mu_0(dx) + 2K^2T^2 + 8dT\bigg).
\end{align*}
This yields condition (2) of Theorem \ref{th:intro-dynamic}, which we may now apply to yield the claimed upper bound on $H(P^{(k)}\,|\,\mu^{\otimes k})$.

Lastly, we prove the inequality $\W_1^2(P^{(k)},\mu^{\otimes k}) \le Ck H(P^{(k)}\,|\,\mu^{\otimes k})$. Note first that \eqref{ineq:expmoment:sublin} from Proposition \ref{pr:MVwellposed-sublinear} is the well known integral criterion (see \cite[Theorem 2.3]{djellout2004transportation} or \cite[Proposition 6.3]{gozlan-leonard}) which implies that there exists $C > 0$ such that
\begin{align*}
\W_1^2(\nu,\mu) &\le C H(\nu\,|\,\mu), \qquad \text{for all } \nu \in \P(\C_T^d).
\end{align*}
This inequality tensorizes (see \cite[Proposition 1.9]{gozlan-leonard}) to yield
\begin{align*}
\W_{1,\ell_1}^2(\nu,\mu^{\otimes k}) &\le C k H(\nu\,|\,\mu^{\otimes k}), \qquad \text{for all } k \in \N, \ \ \nu \in \P((\C_T^d)^k),
\end{align*}
where $\W_{1,\ell_1}$ is the Wasserstein distance on $(\C_T^d)^k$ defined using the $\ell_1$-norm, as in Remark \ref{re:transport-rate}. Recalling that $\W_1$ is defined in terms of the usual ($\ell_2$) norm of $(\C_T^d)^k\cong \C_T^{dk}$, which is bounded from above by the $\ell_1$-norm, this completes the proof.\hfill \qedsymbol

\section{Some well-posedness results for McKean-Vlasov equations} \label{se:wellposedness}

This section proves some results on well-posedness of McKean-Vlasov equations with irregular coefficients, used in Section \ref{se:proofs:corollaries}.
There has been significant recent efforts in this direction, such as \cite{mishura2016existence,rockner2018well}, though none of them appear to cover exactly our results below. 

The first result, Proposition \ref{pr:MVwellposed-bounded}, is nearly covered by existing results, with the exception that $b_0$ is not required to be bounded. Our proof is inspired by the approach of \cite[Theorem 2.4]{lacker2018strong}.
As usual, we assume implicitly that $b_0$ and $b$ are progressively measurable.

\begin{proposition} \label{pr:MVwellposed-bounded}
Suppose the SDE$(b_0)$ is well-posed in the sense of Definition \ref{def:SDEwellposed} and $b$ is bounded. 
Then the McKean-Vlasov SDE \eqref{def:introMV-nonMarkov} admits a unique in law weak solution from any initial distribution.
\end{proposition}
\begin{proof}
Fix an initial law $\mu_0 \in \P(\R^d)$.
For $\mu \in \P(\C_T^d)$ define $\overline{b}_\mu(t,x) = b_0(t,x) + \langle \mu,b(t,x,\cdot)\rangle$. A standard argument using Girsanov's theorem shows that the SDE$(\overline{b}_\mu)$ is well-posed for each $\mu$, because the SDE$(b_0)$ is. Let $\Phi(\mu) \in \P(\C_T^d)$ denote the law of the unique solution of the SDE$(\overline{b}_\mu)$, started from $\mu_0$.
To prove the claim, we must show that $\Phi$ has a unique fixed point.
Using Pinsker's inequality \eqref{ineq:Pinsker} and Lemma \ref{le:entropy-diffusions}, for $t \in [0,T]$ we have
\begin{align*}
\|\Phi(\mu)[t] - \Phi(\nu)[t]\|_{\mathrm{TV}}^2 &\le 2H(\Phi(\mu)[t]\,|\,\Phi(\nu)[t]) \\
	&= \int_{\C_t^d}\int_0^t \left|\overline{b}_\mu(s,x) - \overline{b}_\nu(s,x)\right|^2\,ds\,\mu[t](dx) \\
	&= \int_{\C_t^d}\int_0^t \left|\langle \mu[s] - \nu[s],b(s,x,\cdot)\rangle\right|^2\,ds\,\mu[t](dx) \\
	&\le \||b|^2\|_\infty \int_0^t \|\mu[s] - \nu[s]\|_{\mathrm{TV}}^2\,ds.
\end{align*}
A standard Picard iteration completes the proof.
\end{proof}

The next result, tailored to the setting of Theorem \ref{th:sublinear}, is somewhat more novel. The closest prior result seems to be \cite[Theorem 2]{mishura2016existence}, which requires the stronger uniform linear growth condition $\sup_{t,y}|b(t,x,y)| \le c(1+|x|)$.

\begin{proposition} \label{pr:MVwellposed-sublinear}
Suppose $(b_0,b)$ satisfy the assumption \eqref{condition:lingrowth} of Theorem \ref{th:sublinear}.
Let $\mu_0 \in \P(\R^d)$ satisfy $\int_{\R^d} e^{c_0|x|^2}\mu_0(dx) < \infty$ for some $c_0 > 0$.
Then the McKean-Vlasov SDE \eqref{def:introMV-nonMarkov} admits a unique in law weak solution with initial distribution $\mu_0$.
Moreover, there exist $c,\transpconst > 0$ such that the solution $\mu \in \P(\C_T^d)$ satisfies
\begin{align}
\int_{\C_T^d} e^{c\|x\|_T^2}\mu(dx) &< \infty, \quad \text{and}  \label{ineq:expmoment:sublin} \\
\sup_{(t,x) \in (0,T) \times \C_T^d}|\langle \mu - \nu, b(t,x,\cdot)\rangle|^2 &\le  \transpconst H(\nu\,|\,\mu), \ \ \forall \nu \in \P(\C_T^d). \label{ineq:transp:sublin}
\end{align}
\end{proposition}
\begin{proof} Fix an initial law $\mu_0 \in \P(\R^d)$ satisfying the stated condition.

{\ } \vskip-.2cm

\noindent\textit{Uniqueness.} We first prove that there is at most one solution. Let $X^i\sim\mu^i$ for $i=1,2$ be two solutions, 
\begin{align*}
dX^i_t = \left(b_0(t,X^i) + \langle \mu^i,b(t,X^i,\cdot)\rangle\right)dt + dW^i_t, \ \ X^i_0 \sim \mu_0, \ \ \mu^i=\mathrm{Law}(X^i).
\end{align*}
Using the assumption \eqref{condition:lingrowth}, for $t \in [0,T]$ we have
\begin{align*}
\|X^i\|_t \le |X_0^i| + KT + K\int_0^t (\|X^i\|_s + \E\|X^i\|_s)\,ds + \|W^i\|_t.
\end{align*}
Take expectations and apply Gronwall to get $\E\|X^i\|_T \le e^{2KT}(\E|X^i_0| + 4dT)$, where we used also $\E \|W^i\|_T^2 \le 4\E |W_T^i|^2 = 4dT$. Use Gronwall again to get also
\begin{align*}
\|X^i\|_T \le e^{KT}\left(|X^n_0| + KT + KT\E\|X^i\|_T +  \|W^i\|_T\right).
\end{align*}
Since $\E e^{c\|W\|_T^2} < \infty$ for sufficiently small $c > 0$ by Fernique's theorem, the assumption on $\mu_0$ ensures that we may find $c < 0$ small enough and $C < \infty$ that
\begin{align}
\int_{\C_T^d} e^{c\|x\|_T^2}\,\mu^i(dx) = \E\, e^{c\|X^i\|_T^2} \le C, \label{pf:wellposed-sublin3.0}
\end{align}
for $i=1,2$.
Next, note for $m \in \P(\C_T^d)$ that  the assumption \eqref{condition:lingrowth} yields
\begin{align}
\big|b(t,x,y) - \langle m,b(t,x,\cdot)\rangle\big| &\le \int_{\C^d_T} | b(t,x,y)-b(t,x,y')| \, m(dy') \nonumber \\
	&\le K\left(1 + \|y\|_T + \int_{\C^d_T}  \|y'\|_T \, m(dy')\right), \label{pf:wellposed-sublin3.1}
\end{align}
for all  $(t,x,y)$.
Use Lemma \ref{le:entropy-diffusions} and the weighted Pinsker inequality \eqref{ineq:weightedPinsker} to get
\begin{align}
H(\mu^1[t]\,|\,\mu^2[t]) &= \frac12\int_{\C_T^d}\int_0^t \left| \langle \mu^1 - \mu^2,b(s,x,\cdot)\rangle\right|^2\,ds \\
	&\le \epsilon^{-1}(1+ R_\epsilon ) \int_0^t H(\mu^1[s]\,|\,\mu^2[s]) \,ds, \label{pf:wellposed-sublin3.2}
\end{align}
for each $\epsilon > 0$, where we define
\begin{align*}
R_\epsilon := \sup_{t \in [0,T],\,x \in \C_T^d}\log\int \exp\left(\epsilon \big|b(t,x,y) - \langle \mu^2,b(t,x,\cdot)\rangle\big|^2\right)\,\mu^2(dy).
\end{align*}
If we choose $\epsilon = c/K^2$, then from \eqref{pf:wellposed-sublin3.1} and \eqref{pf:wellposed-sublin3.0} we deduce that $R_\epsilon < \infty$. Thus, applying Gronwall's inequality to \eqref{pf:wellposed-sublin3.2} shows that $H(\mu^1[t]\,|\,\mu^2[t]) =0$ for all $t$, and so $\mu^1=\mu^2$.
The finiteness of $R_\epsilon$, along with Lemma \ref{le:integ-transp-equiv}, implies that any solution must satisfy \eqref{ineq:transp:sublin}.

{\ }

\noindent\textit{Existence.} To prove existence of a solution, we truncate $b$, apply Proposition \ref{pr:MVwellposed-bounded}, and take limits. Note first that the linear growth assumption ensures that the SDE($b_0$) is well-posed. Indeed, existence follows from linear growth and Girsanov's theorem, by  an argument due to Bene\u{s} \cite[Corollary 3.5.16]{karatzas-shreve}. Uniqueness also follows from Girsanov, e.g., by using \cite[Theorem 7.7]{liptser-shiryaev} to identify the (unique) form of the Radon-Nikodym derivative of Wiener measure with respect to the law of the solution. Let $\overline\mu\in\P(\C_T^d)$ denote the law of the unique solution of the SDE($b_0$) started from $\mu_0$.

{\ }

\noindent\textit{Step 1.} We first define the truncations and establish some key integrability estimates.
Define $b^n:= (b \wedge n) \vee (-n)$ for each $n \in \N$. 
By Proposition \ref{pr:MVwellposed-bounded} and boundedness of $b^n$, there exists a unique solution $X^n \sim \mu^n$ of the McKean-Vlasov equation
\begin{align*}
dX^n_t = \left(b_0(t,X^n) + \langle \mu^n,b^n(t,X^n,\cdot)\rangle\right)dt + dW^n_t, \ \ X^n_0 \sim \mu_0, \ \ \mu^n=\mathrm{Law}(X^n).
\end{align*}
Since $b^n$ satisfies the same (first) sublinearity inequality \eqref{condition:lingrowth}, we have as in \eqref{pf:wellposed-sublin3.0} that
there exist $c,C_1 > 0$ such that
\begin{align}
\sup_n\int_{\C_T^d} e^{c\|x\|_T^2}\,\mu^n(dx)  \le C_1. \label{pf:wellposed-sublin3}
\end{align}
Similarly, $\overline\mu$ satisfies the same estimate.
For later use, we note also that if $\epsilon := c/K^2$ then \eqref{pf:wellposed-sublin3} and \eqref{pf:wellposed-sublin3.1} imply that
\begin{align}
R  := \sup_{n \in \N}\sup_{t \in [0,T],\,x \in \C_T^d}\log\int \exp\left(\epsilon \big|b(t,x,y) - \langle \mu^n,b(t,x,\cdot)\rangle\big|^2\right)\,\mu^n(dy) < \infty. \label{pf:wellposed-sublin-Rfin}
\end{align}
Lastly, we may use Lemma \ref{le:entropy-diffusions} to get 
\begin{align}
\sup_n H(\mu^n\,|\,\overline\mu) &= \frac12 \sup_n \int_{C_T^d}\int_0^T|\langle \mu^n,b^n(t,x,\cdot)\rangle|^2\,dt\,\mu^n(dx) < \infty, \label{pf:wellposed-sublin-Hfin}
\end{align}
which is finite because of the linear growth assumption and the estimate \eqref{pf:wellposed-sublin3}. This will not be needed until later.

{\ }

\noindent\textit{Step 2.} Next, we show that $(\mu^n)_{n \in \N}$ is a Cauchy sequence in the total variation norm, which implies by completeness that it converges.
Using Lemma \ref{le:entropy-diffusions}, we have for each $n,m \in \N$ and $t \in [0,T]$
\begin{align*}
H(\mu^n[t]\,|\,\mu^m[t]) &= \frac12 \int_{C_T^d}\int_0^t|\langle \mu^n,b^n(s,x,\cdot)\rangle - \langle \mu^m,b^m(s,x,\cdot)\rangle|^2\,ds\,\mu^n(dx) \\
	&\le \frac32\int_{C_T^d}\int_0^t|\langle \mu^n,b^n(s,x,\cdot) - b(s,x,\cdot)\rangle|^2 \,ds\,\mu^n(dx) \\
	&\qquad + \frac32\int_{C_T^d}\int_0^t|\langle \mu^m,b^m(s,x,\cdot) - b(s,x,\cdot)\rangle|^2 \,ds\,\mu^n(dx) \\
	&\qquad + \frac32\int_{C_T^d}\int_0^t|\langle \mu^n-\mu^m,b(s,x,\cdot)\rangle|^2 \,ds\,\mu^n(dx).
\end{align*}
By Jensen's inequality, 
\begin{align*}
\int_{\C_T^d} &|\langle \mu^n,b^n(s,x,\cdot)-b(s,x,\cdot)\rangle|^2\,\mu^n(dx)  \\ 
	\qquad &\le \int_{\C_T^d}\int_{\C_T^d}|b(s,x,y)|^21_{\{|b(s,x,y)| \ge n\}}\,\mu^n(dx)\,\mu^n(dy).
\end{align*}
It follows from \eqref{pf:wellposed-sublin3} and the linear growth assumption of \eqref{condition:lingrowth} that the right-hand side vanishes as $n\to\infty$.
Similarly,
\begin{align*}
\int_{\C_T^d} &|\langle \mu^m,b^m(s,x,\cdot)-b(s,x,\cdot)\rangle|^2\,\mu^n(dx)   \\	
\qquad &\le \int_{\C_T^d}\int_{\C_T^d}|b(s,x,y)|^21_{\{|b(s,x,y)| \ge m\}}\,\mu^m(dx)\,\mu^n(dy),
\end{align*}
which vanishes as $n,m\to\infty$ for the same reasons.
On the other hand, with $\epsilon=c/K^2$ and $R$ as in \eqref{pf:wellposed-sublin-Rfin}, the weighted Pinsker inequality \eqref{ineq:weightedPinsker} and the progressive measurability of $b$ yield
\begin{align*}
\int_{C_T^d} & \int_0^t |\langle \mu^n-\mu^m,b(s,x,\cdot)\rangle|^2  \,dt\,\mu^n(dx) \le \frac{2}{\epsilon} (1+R)\int_0^t H(\mu^n[s]\,|\,\mu^m[s])  \,dt.
\end{align*}
Using Gronwall's inequality, we finally deduce that $\lim_{n,m\to\infty}H(\mu^n[t]\,|\,\mu^m[t]) = 0$. The claim that $(\mu^n)_{n \in \N}$ is a $\|\cdot\|_{\mathrm{TV}}$-Cauchy sequence follows from Pinsker's inequality.

{\ }

\noindent\textit{Step 3.} Thanks to Step 2, we know $\|\mu^n -\mu\|_{\mathrm{TV}} \to 0$ for some $\mu \in \P(\C_T^d)$. We next upgrade the convergence to allow test functions with linear growth: If $f : \C_T^d \to \R$ is measurable with $\sup_x |f(x)|/(1+\|x\|_T) < \infty$, then $\langle \mu^n,f\rangle \to \langle \mu,f\rangle$. 
To see this, let $k \in \N$ and $\delta > 0$, and apply the Gibbs variational formula \cite[Proposition 2.3(b)]{budhiraja-dupuis} to get
\begin{align*}
\int_{\{|f| > k\}} |f|\,d\mu^n \le \delta \log \int e^{\delta^{-1} |f| 1_{\{|f| > k\}} } \,d\overline\mu + \delta H(\mu^n\,|\,\overline\mu).
\end{align*}
where we recall that $\overline\mu$ is the law of the solution of the SDE($b_0$) started from $\mu_0$.
The linear growth of $f$ and the estimate \eqref{pf:wellposed-sublin3} ensure that the first term converges to zero as $k\to\infty$, for any $\delta > 0$. Hence, sending $k\to\infty$ and then $\delta \to 0$, the uniform bound on $H(\mu^n\,|\,\overline\mu)$ from \eqref{pf:wellposed-sublin-Hfin} ensures that $\int_{\{|f| > k\}} f\,d\mu^n \to 0$ as $k\to\infty$, uniformly in $n$. The integrability \eqref{pf:wellposed-sublin3} for $\overline\mu$ implies also that $\int_{\{|f| > k\}} |f|\,d\overline\mu \to 0$  as $k\to\infty$. Writing
\begin{align*}
\langle \mu^n -\mu,f\rangle &= \int_{\{|f| \le k\}} f\,d(\mu^n-\mu) + \int_{\{|f| > k\}} f\,d(\mu^n-\mu),
\end{align*}
the right hand side vanishes by sending  $n\to\infty$ and then $k\to\infty$.

{\ }

\noindent\textit{Step 4.} We complete the proof by showing that $\mu$ is the law of a solution to the McKean-Vlasov equation. To do so, we study the corresponding \emph{martingale problem}, in the sense of Stroock-Varadhan \cite{stroock-varadhan}, and we refer to \cite{stroock-varadhan} for the now-classical definitions and connection to weak solutions of SDEs.
Using the McKean-Vlasov equation solved by  $\mu^n$, in martingale problem form, we know that 
\begin{align}
\begin{split}
0 &= \int_{\C_T^d}h(x|_{[0,s]})\Bigg[\varphi(x_t) - \varphi(x_s) \\
	&\qquad\quad - \int_s^t \left(\left(b_0(u,x) + \langle \mu^n,b^n(u,x,\cdot)\rangle\right) \cdot \nabla\varphi(x_u) + \frac12\Delta\varphi(x_u)\right)du\Bigg]\,\mu^n(dx),
\end{split} \label{pf:wellposed-sublin2.1}
\end{align}
for each $0 \le s < t \le T$, each $n \in \N$, each bounded continuous function $h : \C_s^d \to \R$, and each smooth function $\varphi : \R^d \to \R$ of compact support. We wish to send $n\to\infty$.

We claim first that, for each $t \ge 0$,
\begin{align}
\lim_{n\to\infty} \int_{\C_T^d} |\langle \mu^n, b^n(t,x,\cdot)\rangle - \langle \mu, b(t,x,\cdot)\rangle| \,\mu^n(dx) \to 0. \label{pf:wellposed-sublin2}
\end{align}
To see this, first note
using Jensen's inequality and the definition of $b^n$ that
\begin{align*}
\int_{\C_T^d} & |\langle \mu^n,b^n(t,x,\cdot)-b(t,x,\cdot)\rangle|\,\mu^n(dx)  \le  \int_{\C_T^d}\int_{\C_T^d}|b(t,x,y)| 1_{\{|b(t,x,y)| \ge n\}}\,\mu^n(dx)\,\mu^n(dy),
\end{align*}
which converges to zero thanks to the linear growth assumption and \eqref{pf:wellposed-sublin3}.
Hence, to prove \eqref{pf:wellposed-sublin2}, it suffices to show that $\langle \mu^n,|a_n|\rangle \to 0$, where we set $a_n(x) := \langle \mu^n-\mu, b(t,x,\cdot)\rangle$. The result of Step 3 and the linear growth assumption on $b$ imply that $a_n\to 0$ pointwise. Using the assumption \eqref{condition:lingrowth}, we have
\begin{align*}
|a_n(x)| &\le \int_{\C_T^d}\int_{\C_T^d} \left|b(t,x,y) - b(t,x,y')\right|\mu^n(dy)\mu(dy') \\
	&\le K\left(1 + \int_{\C_T^d}\|y\|_T \mu^n(dy) + \int_{\C_T^d}\|y'\|_T \mu(dy')\right) < \infty.
\end{align*}
For each $\delta > 0$, the Gibbs variational formula \cite[Proposition 2.3(b)]{budhiraja-dupuis} again yields
\begin{align*}
\langle \mu^n,|a_n|\rangle \le \delta \log\int e^{\delta^{-1}|a_n|}\,d\overline\mu + \delta \sup_k H(\mu^k\,|\,\overline\mu).
\end{align*}
For fixed $\delta > 0$, the first term vanishes as $n\to\infty$ by the bounded convergence theorem.
Recalling \eqref{pf:wellposed-sublin-Hfin}, we may send $n\to\infty$ and then $\delta \to 0$ to find $\langle \mu^n,|a_n|\rangle \to 0$.

With \eqref{pf:wellposed-sublin2} now established, it follows from \eqref{pf:wellposed-sublin2.1} that
\begin{align*}
0 &= \lim_{n\to\infty}\int_{\C_T^d}h(x|_{[0,s]})\Bigg[\varphi(x_t) - \varphi(x_s) \\
	&\qquad\qquad\quad - \int_s^t \left(\left(b_0(u,x) + \langle \mu,b(u,x,\cdot)\rangle\right) \cdot \nabla\varphi(x_u) + \frac12\Delta\varphi(x_u)\right)du\Bigg]\,\mu^n(dx).
\end{align*}
Using the linear growth assumption and the result of Step 3, we deduce
\begin{align*}
0 &= \int_{\C_T^d}h(x|_{[0,s]})\Bigg[\varphi(x_t) - \varphi(x_s) \\
	&\qquad\quad - \int_s^t \left(\left(b_0(u,x) + \langle \mu,b(u,x,\cdot)\rangle\right) \cdot \nabla\varphi(x_u) + \frac12\Delta\varphi(x_u)\right)du\Bigg]\,\mu(dx).
\end{align*}
This holds for all $(t,s,h,\varphi)$ as above, which implies that $\mu$ is a solution of the martingale problem corresponding to the McKean-Vlasov equation.
\end{proof}

\begin{remark} \label{re:MVuniqueness}
A more general principle is at work in the uniqueness part of Propositions \ref{pr:MVwellposed-bounded} and \ref{pr:MVwellposed-sublinear}.
Precisely, suppose $\mu \in \P(\C_T^d)$ is a solution of the McKean-Vlasov equation \eqref{def:introMV-nonMarkov} satisfying conditions (1) and (3) of Theorem \ref{th:intro-dynamic}. Suppose another solution $\widetilde\mu$ satisfies $\widetilde\mu_0=\mu_0$, and the function $[0,T] \times \C_T^d \ni (t,x) \mapsto |\langle \widetilde\mu - \mu,b(t,x,\cdot)\rangle|$ belongs to $L^2(dt \otimes (\mu+\widetilde\mu))$. 
Then the well-posedness assumption (1) justifies an application of Lemma \ref{le:entropy-diffusions}(iii), and  the transport-type inequality (3) and progressive measurability of $b$ imply
\begin{align*}
H(\widetilde\mu[t]\,|\,\mu[t]) &= \frac12 \int_0^t \int_{\R^d} |\langle \widetilde\mu - \mu,b(s,x,\cdot)\rangle|^2\,\mu(dx)\,ds \le \frac{\transpconst}{2}\int_0^t H(\widetilde\mu[s]\,|\,\mu[s])\,ds,
\end{align*}
for $t \in [0,T]$. Gronwall's inequality implies $\widetilde\mu=\mu$. 
\end{remark}

\section{Proofs for infinite-range interactions} \label{se:proofs-inf}

This section is devoted to the proof of Theorem \ref{th:intro-infrange}, which we first briefly preview. It follows a similar scheme as in the case of pairwise interactions, as outlined in Section \ref{se:ideas} and implemented in Section \ref{se:proofs-dyn}. In the finite-range case, where $s_\ell=0$ for all $\ell > r$ for some $r \in \N$, the main difference is that the final term in \eqref{idea:diffeq-entropy} becomes $H^{k+r}_t-H^k_t$ instead of $H^{k+1}_t-H^k_t$. In the infinite-range case, we truncate the range of interactions, and this is where the parameter $\ell$ and the term $\tailfn_0(\ell)$ in \eqref{infrange-est} come from.

Throughout this section, for integers $k < n$ we write $(k,n]:=\{k+1,k+2,\ldots,n\}$. When $k=0$, we write simply $[n]:=(0,n]$. The Cartesian product is denoted as usual $(k,n]^\ell$ for integers $\ell \ge 1$, and we write $(k,n]^\ell_{\neq}$ for the subset consisting of those vectors with distinct entries.
Note of course that $(k,n]^{\ell}_{\neq} = \emptyset$ if $\ell > n-k$ by the pigeonhole principle.
Note also that $|[k]^\ell|=k^\ell$ and $|[k]^{\ell}_{\neq}| = (k)_{\ell}$ for $\ell \le k$, where we use the falling factorial notation 
\[
(x)_{\ell} := x(x-1)\cdots(x-\ell+1) = \prod_{i=0}^{\ell-1}(x-i), \qquad x \in \R.
\]
We state here for later use a fairly standard estimate which can be interpreted as comparing sampling schemes \emph{with} versus \emph{without} replacement, cf.\ \cite{diaconis1980finite} or \cite[Theorem 3.1]{golse2013empirical}.

\begin{lemma} \label{le:combin1}
Let $\ell, n \in \N$, and let $-1 \le a_S \le 1$ for $S \in [n]^{\ell}$.  Then
\begin{align*}
\Bigg|\frac{1}{n^{\ell}}\sum_{S \in [n]^\ell } a_S - \frac{1}{(n-1)_{\ell}}\sum_{S \in (1,n]^{\ell}_{\neq} } a_S\Bigg| \le \frac{\ell(\ell+1)}{n}.
\end{align*}
\end{lemma}
\begin{proof}
If $\ell \ge n$, then $(1,n]_{\neq}^\ell=\emptyset$, and the claim follows from the trivial estimate
\begin{align*}
\Bigg|\frac{1}{n^{\ell}}\sum_{S \in [n]^\ell } a_S \Bigg| \le 1 \le \frac{\ell}{n}.
\end{align*}
If $\ell \le n-1$, then 
\begin{align*}
\Bigg|\frac{1}{n^{\ell}}\sum_{S \in [n]^\ell } a_S - \frac{1}{(n-1)_{\ell}}\sum_{S \in (1,n]^{\ell}_{\neq} } a_S\Bigg| &\le \left|\frac{1}{n^\ell} - \frac{1}{(n-1)_{\ell}}\right|\Bigg| \sum_{S \in (1,n]^{\ell}_{\neq}} a_S\Bigg| + \Bigg|\frac{1}{n^{\ell}}\sum_{S \in [n]^\ell \setminus (1,n]^{\ell}_{\neq}} a_S\Bigg| \\
	&\le 2\left( 1 - \frac{ (n-1)_{\ell} }{n^\ell} \right) = 2\left(1 - \prod_{i=0}^{\ell-1} \left( 1- \frac{1+i}{n}\right)\right) \\
	&\le 2\sum_{i=0}^{\ell-1}\frac{1+i}{n} = \frac{\ell(\ell+1)}{n}.
\end{align*}
The last line used the well known inequality 
\begin{align*}
1-\prod_{i=1}^{m}(1-c_i) \le \sum_{i=1}^m c_i, \quad \text{for } c_1,\ldots,c_m \in [0,1],
\end{align*}
which is precisely the inequality $f(0) - f(1) \le -f'(0)$ applied to the function $f(t) = \prod_{i=1}^{m}(1-t c_i)$ which is convex on $[0,1]$.
\end{proof}

\subsection{Estimating the full entropy} \label{se:infrange-fullentropy}

Abbreviate $H^k_t := H_t(P^{(k)}[t] \,|\, \mu^{\otimes k}[t])$ for each $k\in [n]$ and $t \in [0,T]$.
Recall that $(X^1,\ldots,X^n) \sim P^{(n)}$ solves the SDE system \eqref{def:mainSDE-finrange}. 

First, we will apply Lemma \ref{le:entropy-diffusions}(iii). To justify this, define $\overline{b} : [0,T] \times \C_T^d \to \R^d$ by
\begin{align*}
\overline{b}(t,x) = b_0(t,x) + b(t,x,\mu).
\end{align*}
Then, since $\overline{b}-b_0$ is bounded and the SDE($b_0^{\otimes n}$) is well-posed by assumption, a standard argument using Girsanov's theorem shows that that the SDE($\overline{b}^{\otimes n}$) is well-posed.
Hence, we may apply Lemma \ref{le:entropy-diffusions}(iii), recalling the definition of $\bar{s}_0$ from \eqref{def:tailfn}, to get
\begin{align}
H^n_T &= H^n_0 + \frac12 \sum_{i=1}^n \E\int_0^T \big| b(t,X^i,L^n) - b(t,X^i,\mu) \big|^2 dt \le n(C_0 + T\bar{s}_0^2/2). \label{pf:Hnbound-infrange}
\end{align}
Indeed, the last step used \eqref{def:powerseries} to deduce $|b(t,X^i,L^n) - b(t,X^i,\mu)| \le \bar{s}_0$, noting that the assumption $|b_\ell|\le 2^{-\ell}$ implies
\begin{align}
|\langle (L^n-\mu)^{\otimes \ell},b_\ell\rangle| \le 2^{-\ell} \|(L^n-\mu)^{\otimes \ell}\|_{\mathrm{TV}} \le 2^{-\ell} \|L^n-\mu\|_{\mathrm{TV}}^\ell \le 1. \label{pf:infrange-TVbound}
\end{align}

\subsection{Main proof}

Now, fix $k \in \N$ with $1 \le k < n/2$.
We will apply Lemma \ref{le:proj-pathdep} to identify a convenient representation of $P^{(k)}$. 
Let $\F^k_t = \sigma((X^1_s,\ldots,X^k_s)_{s \le t})$ denote the natural filtration of the first $k$ coordinates.
By Lemma \ref{le:proj-pathdep}, we have that $(X^1,\ldots,X^k)$ solves the SDE system
\begin{align*}
dX^i_t = \bigg(b_0(t,X^i) + \E[b(t,X^i,L^n)\,|\,\F^k_t]\bigg)dt + d\widehat{W}^i_t,
\end{align*}
for independent Brownian motions $\widehat{W}^1,\ldots,\widehat{W}^k$.
We then deduce, using Lemma \ref{le:entropy-diffusions} and exchangeability, that $t \mapsto H^k_t$ is absolutely continuous, and for a.e.\ $t \in [0,T]$
\begin{align*}
\frac{d}{dt}H^k_t &=\frac12 \sum_{i=1}^k \E\big|  \E[b(t,X^i,L^n)\,|\,\F^k_t] - b(t,X^i,\mu)\big|^2 \\
	&= \frac{k}{2}\E\big|  \E[b(t,X^1,L^n) - b(t,X^1,\mu)\,|\,\F^k_t]\big|^2 \\
	&= \frac{k}{2}\E\bigg| \sum_{\ell=1}^{\infty} \E\big[\big\langle (L^n-\mu)^{\otimes \ell} ,b_\ell(t,X^1,\cdot)\big\rangle\,\big|\,\F^k_t\big] \bigg|^2.
\end{align*}
We next truncate the infinite range of interactions, so that we may handle the finite-range part as in the proof of Theorem \ref{th:intro-dynamic}.
Let $\bar\ell \in \N$ with $\bar\ell < n/2$ and $\bar\ell \le n-k$. Then 
\begin{align*}
\frac{d}{dt}H^k_t  &\le k \E\bigg| \sum_{\ell=1}^{\bar\ell} \E\big[\big\langle (L^n-\mu)^{\otimes \ell} ,b_\ell(t,X^1,\cdot)\big\rangle\,\big|\,\F^k_t\big] \bigg|^2  \\
	&\quad + k \E\bigg| \sum_{\ell=\bar\ell+1}^{\infty} \E\big[\big\langle (L^n-\mu)^{\otimes \ell} ,b_\ell(t,X^1,\cdot)\big\rangle\,\big|\,\F^k_t\big] \bigg|^2  .
\end{align*}
Using $|b_\ell| \le 2^{-\ell}$ as in \eqref{pf:infrange-TVbound}, we may bound the second term by
\begin{align*}
 k \bigg(\sum_{\ell=\bar\ell+1}^{\infty}  s_\ell\bigg)^2 =  k \bar{s}_0^2 \tailfn_0^2(\bar\ell),
\end{align*}
For the first term, we expand the tensor power to write
\begin{align*}
\big\langle(L^n-\mu)^{\otimes \ell},b_\ell(t,X^1,\cdot)\big\rangle = \frac{1}{n^\ell}\sum_{S \in [n]^\ell} Z_S, \qquad Z_S := \Big\langle \bigotimes_{i \in S} (\delta_{X^i}-\mu), b_\ell(t,X^1,\cdot)\Big\rangle,
\end{align*}
which is well defined without regard to the order of the tensor product because the function $b_\ell(t,X^1,\cdot)$ is symmetric.
Since $|Z_S| \le 1$ (argued again as in \eqref{pf:infrange-TVbound}), Lemma \ref{le:combin1} implies
\begin{align*}
\Bigg|\frac{1}{n^{\ell}}\sum_{S \in [n]^\ell }Z_S - \frac{1}{(n-1)_{\ell}}\sum_{S \in (1,n]^\ell_{\neq} } Z_S\Bigg| \le \frac{ \ell(\ell+1)}{n} \le \frac{2\ell^2}{n},
\end{align*}
for each $\ell \ge 1$, where we used also $\ell+1 \le 2\ell$. Hence,
\begin{align}
\frac{d}{dt}H^k_t  &\le k \bar{s}_0^2 \tailfn_0^2(\bar\ell) + 2k\bigg( \sum_{\ell=1}^{\bar\ell} s_\ell \frac{2\ell^2}{n} \bigg)^2 + 2k \E \bigg| \sum_{\ell=1}^{\bar\ell} \frac{s_\ell}{(n-1)_{\ell}}\sum_{S \in (1,n]^{\ell}_{\neq} } \E[ Z_S\,|\,\F^k_t]  \bigg|^2 \nonumber \\
	&\le  k \bar{s}_0^2 \tailfn_0^2(\bar\ell) + \frac{8k}{n^2}\bar{s}_2^2 + R_1  + R_2, \label{pf:infrange-ent1}
\end{align}
where we define
\begin{align*}
R_1 &:= 4k \E \bigg| \sum_{\ell=1}^{\bar\ell} \frac{s_\ell}{(n-1)_{\ell}}\sum_{S \in (1,k]^\ell_{\neq}  } \E[ Z_S\,|\,\F^k_t] \bigg|^2, \\
R_2 &:= 4k \E \bigg| \sum_{\ell=1}^{\bar\ell} \frac{s_\ell}{(n-1)_{\ell}}\sum_{S \in (1,n]^\ell_{\neq} \setminus   (1,k]_{\neq}^\ell }  \E[ Z_S\,|\,\F^k_t] \bigg|^2.
\end{align*}
To estimate $R_1$, use use $|Z_S| \le 1$ and $(k-1)_\ell/(n-1)_\ell \le k/n$ for $\ell \ge 1$ to get
\begin{align}
R_1 &\le 4k\bigg(\sum_{\ell=1}^{\bar\ell} s_\ell \frac{(k-1)_\ell}{(n-1)_\ell}\bigg)^2 \le \frac{4 k^3}{n^2}\bar{s}_0^2. \label{pfinfrange-R1bound}
\end{align}
To estimate $R_2$, first use convexity of $x\mapsto x^2$ to get
\begin{align*}
R_2 &\le 4k\bigg( \sum_{\ell=1}^{\infty}  s_\ell  \bigg) \sum_{\ell=1}^{\bar\ell}\frac{s_\ell}{(n-1)_{\ell}}\sum_{S \in (1,n]^\ell_{\neq} \setminus   (1,k]_{\neq}^\ell } \E \big| \E[ Z_S\,|\,\F^k_t]\big|^2.
\end{align*}
We next use Pinsker's inequality and the chain rule for relative entropy \eqref{def:chainrule}. Let $1 \le \ell \le \bar\ell$ and $S \in (1,n]^\ell_{\neq} \setminus   (1,k]_{\neq}^\ell$ be arbitrary, and choose $i^* \in (k,n] \cap S$.
Let $P_S$ denote a version of the conditional law of $X^{i^*}$ given $\F^{S,i^*}_t$, the $\sigma$-algebra generated by $(X^j|_{[0,t]})_{j \in [k] \cup S \setminus \{i^*\}}$.
The key point is that, \emph{because $S$ contains $i^*$ at only one index}, we may use the identity $P_S=\E[\delta_{X_{i^*}}\,|\,\F^{S,i^*}_t]$ to deduce
\begin{align*}
\E[ Z_S\,|\,\F^{S,i^*}_t] &= \E\Big[\Big\langle \bigotimes_{i \in S} (\delta_{X^i}-\mu), b_\ell(t,X^1,\cdot)\Big\rangle \,\Big|\, \F^{S,i^*}_t\Big] \\
	&= \Big\langle (P_S - \mu ) \otimes \bigotimes_{i \in S \setminus \{i^*\}} (\delta_{X^i}-\mu), b_\ell(t,X^1,\cdot)\Big\rangle.
\end{align*}
Pinsker's inequality yields (since $|b_\ell| \le 2^{-\ell}$)
\begin{align*}
\langle P_S - \mu, b_\ell(t,X^1,(X^i)_{i \in S \setminus \{i^*\}},\cdot)\rangle \le 2^{1-\ell}H(P_S[t]\,|\,\mu[t]).
\end{align*}
Thus, since $|S \setminus\{i^*\}|=\ell-1$ and the norm of a tensor product is the product of the norms,
\begin{align*}
|\E[ Z_S\,|\,\F^{S,i^*}_t]| \le 2^{1-\ell}H(P_S[t]\,|\,\mu[t]) \bigg\| \bigotimes_{i \in S \setminus \{i^*\}} (\delta_{X^i}-\mu) \bigg\|_{\mathrm{TV}} \le H(P_S[t]\,|\,\mu[t]).
\end{align*}
Finally, since $\F^k_t \subset \F^{S,i^*}_t$, we deduce
\begin{align*}
\E\big|\E[ Z_S\,|\,\F^k_t]\big|^2 &\le \E\big|\E[ Z_S\,|\,\F^{S,i^*}_t] \big|^2 \le \E H(P_S[t]\,|\,\mu[t]) \le H^{k+\ell}_t-H^{k+\ell-1}_t \le H^{k+\bar\ell}_t-H^k_t,
\end{align*}
where we used the chain rule \eqref{def:chainrule} and the monotonicity of entropy \eqref{ineq:dataproc2} in the last two steps.
Using the trivial bound $|(1,n]^\ell_{\neq} \setminus   (1,k]_{\neq}^\ell| \le |(1,n]^\ell_{\neq}| = (n-1)_\ell$, we get
\begin{align}
R_2 &\le 4k \bar{s}_0^2 \big( H^{k+\bar\ell}_t - H^k_t\big). \label{pfinfrange-R2bound}
\end{align}

\subsection{Iterating}

Combine \eqref{pf:infrange-ent1} with \eqref{pfinfrange-R1bound} and \eqref{pfinfrange-R2bound} to get
\begin{align*}
\frac{d}{dt}H^k_t  &\le a_k + c_k \big( H^{k+\bar\ell}_t - H^k_t\big),
\end{align*}
where we define
\begin{align*}
a_k &:= k\bar{s}_0^2\tailfn_0^2(\bar\ell) + \frac{8k}{n^2}\bar{s}_2^2 + \frac{4k^3}{n^2} \bar{s}_0^2, \qquad c_k := 4k\bar{s}_0^2.
\end{align*}
By Gronwall's inequality, we get
\begin{align*}
H^k_t &\le e^{-c_kt}H^k_0 + \int_0^t e^{-c_k(t-s)}\left(  a_k + c_k H^{k+\bar\ell}_s \right)ds, \qquad t \in [0,T].
\end{align*}
For $\bar{r} \ge 1$ satisfying $k+ \bar{r} \bar\ell < n/2$, we may iterate this $\bar{r}$ times to get
\begin{align}
H^k_T &\le \sum_{r=0}^{\bar{r}-1} B_k^r(T)H^{k+r\bar\ell}_0 + \sum_{r=0}^{\bar{r}-1} \frac{a_{k+r\bar\ell}}{c_{k+r\bar\ell}} A_k^r(T) + A_k^{\bar{r}-1}(T)H^{k+\bar{r}\bar\ell}_T, \label{pf:infrange-key1}
\end{align}
where we define
\begin{align*}
A_k^r(t_0) &:= \bigg(\prod_{j=0}^{r}c_{k+j\bar\ell} \bigg) \int_0^{t_0}\int_0^{t_1}\cdots\int_0^{t_r} e^{- \sum_{j=0}^{r} c_{k+j\bar\ell}(t_j - t_{j+1}) } dt_{r+1}\cdots dt_{2}dt_{1}, \\
B_k^\ell(t_0) &= \bigg(\prod_{j=0}^{r-1}c_{k+j\bar\ell} \bigg) \int_0^{t_0}\int_0^{t_{1}}\cdots\int_0^{t_{r-1}} e^{ - c_{k+r\bar\ell}t_{r} - \sum_{j=0}^{r-1} c_{k+j\bar\ell}(t_j - t_{j+1}) } dt_{r}\cdots dt_{2}dt_{1},
\end{align*}
for $t_k \ge 0$, with $B_k^{k}(t) := e^{-c_k t}$. 

Next, note that 
\begin{align*}
\frac{a_{k}}{c_{k}} &= \frac14 \tailfn_0^2(\bar\ell) + \frac{2\bar{s}_2^2}{n^2\bar{s}_0^2} + \frac{ k^2 }{n^2 }.
\end{align*}
Use the assumption \eqref{asmp:infrange-init} to control the first term in \eqref{pf:infrange-key1}, along with the estimate \eqref{pf:Hnbound-infrange} for the last term, to get
\begin{align}
\begin{split}
H^k_T &\le \frac{C_0}{n^2}\sum_{r=0}^{\bar{r}-1} (k+r\bar\ell)^3 B_k^r(T) + \left(\frac14 \tailfn_0^2(\bar\ell) + \frac{2\bar{s}_2^2}{n^2\bar{s}_0^2}\right)\sum_{r=0}^{\bar{r}-1} A_k^r(T) \\
	&\quad  + \frac{1}{n^2 }\sum_{r=0}^{\bar{r}-1}  (k+r\bar\ell)^2 A_k^r(T) + A_k^{\bar{r}-1}(T)n\left(C_0 + T\bar{s}_0^2/2 \right).
\end{split} \label{pf:infrange-key3}
\end{align}
Note that this holds for integers $\bar\ell,\bar{r} \ge 1$ satisfying $k+ \bar{r} \bar\ell \le n/2$.

Next, apply Proposition \ref{pr:ABestimate} with $a= 4\bar{s}_0^2 k$ and $b= 4\bar{s}_0^2\bar\ell$ to get
\begin{align*}
A_k^{\bar{r}-1}(T) &\le \exp\left(-2(\bar{r}+(k/\bar\ell)+1)\left(e^{-4\bar{s}_0^2 \bar\ell T} - \frac{(k/\bar\ell)}{\bar{r}+(k/\bar\ell)} \right)_+^2 \right), 
\end{align*}
as well as
\begin{align*}
\sum_{r=0}^{\infty} (k+r\bar\ell)^2 A_k^r(T) &\le \frac{e^{12 \bar{s}_0^2 \bar\ell T} - 1}{3} (k/\bar\ell)((k/\bar\ell)+1)((k/\bar\ell)+2) \le \frac{2k^3}{\bar\ell}e^{12 \bar{s}_0^2 \bar\ell T}, \\
\sum_{r=0}^\infty A_k^r(T) &\le (k/\bar\ell)  e^{4\bar{s}_0^2 \bar\ell T}, \\
\sum_{r=0}^{\infty} (k+r\bar\ell)^3 B_k^r(T) &\le  4\bar{s}_0^2 k ((k/\bar\ell)+1)((k/\bar\ell)+2) e^{12\bar{s}_0^2 \bar\ell T} \le 24\bar{s}_0^2 k^2 e^{12\bar{s}_0^2 \bar\ell T} .
\end{align*}
We used $(k/\bar\ell) \le k$ in the first and last lines.
Apply these inequalities in \eqref{pf:infrange-key3} to get
\begin{align*}
H^k_T &\le 24 C_0\bar{s}_0^2\frac{k^3}{n^2} e^{12\bar{s}_0^2 \bar\ell T} + \frac{2k^3}{\bar\ell n^2} e^{12\bar{s}_0^2 \bar\ell T} + \left(\frac14 \tailfn_0^2(\bar\ell) + \frac{2\bar{s}_2^2}{n^2\bar{s}_0^2}\right)  \frac{k}{\bar\ell} e^{4\bar{s}_0^2 \bar\ell T} \\
	&\quad +  n\left(C_0 +  \bar{s}_0^2 T/2  \right)\exp\bigg(-2(\bar{r}+(k/\bar\ell)+1)\bigg(e^{-4 \bar{s}_0^2  \bar\ell T} - \frac{(k/\bar\ell)}{\bar{r}+(k/\bar\ell) } \bigg)_+^2 \bigg).
\end{align*}
Note that $1 \le (\bar{s}_2/\bar{s}_0)^2$.
Choose $\bar{r} = \lfloor ((n/2)-k)/\bar\ell \rfloor$, so that $k+\bar{r}\bar\ell \le n/2$ as was required above. Then $\bar{r} + 1 \ge ((n/2)-k)/\bar\ell$, and we get 
\begin{align}
\begin{split}
H^k_T &\le 24 C_0\bar{s}_0^2\frac{k^3}{n^2} e^{12\bar{s}_0^2 \bar\ell T} +  4 \frac{ \bar{s}_2^2 }{\bar{s}_0^2} \frac{k^3}{\bar\ell n^2} e^{12\bar{s}_0^2 \bar\ell T}  + \frac{k}{\bar\ell} \tailfn_0^2(\bar\ell) e^{4\bar{s}_0^2 \bar\ell T} \\
	&\quad +  n\left(C_0 + \bar{s}_0^2 T  \right)\exp\bigg(-  \frac{n}{\bar\ell} \bigg(e^{-4 \bar{s}_0^2 \bar\ell T} - \frac{k}{2n } \bigg)_+^2 \bigg).
\end{split} \label{pf:infrange-mainbound}
\end{align}
This completes the proof of Theorem \ref{th:intro-infrange}. \hfill \qedsymbol

\begin{remark} \label{re:inf-selfimprovement}
Having now an estimate on $H^k_T$, one might return to the estimate of $R_1$ and estimate the cross-terms more effectively upon expanding the square, as we did in Theorem \ref{th:intro-dynamic}. This may allow an improvement of the $k^3$ term in \eqref{pf:infrange-mainbound} to at best $k^2$, but the dependence on $\bar\ell$ in \eqref{pf:infrange-mainbound} makes this intricate.
\end{remark}

\appendix

\section{Proof of Lemma \ref{le:entropy-diffusions}} \label{ap:entropyestimates}

The claim (ii) follows immediately from \cite[Theorem 2.3]{leonard2012girsanov}, after noting that the \emph{uniqueness condition} assumed therein is implied by our notion of well-posedness given in Definition \ref{def:SDEwellposed}.
Let us prove first that (iii) follows from (i) and (ii). To see this, note first that the assumption \eqref{asmp:finiteentropy2} is stronger than \eqref{asmp:finiteentropy1}, so the conclusion of (i) holds. Note that $H(P^1[t]\,|\,P^2[t]) \ge H(P^1_0\,|\,P^2_0)$ for all $t \in [0,T]$, by the data processing inequality (recalled at the beginning of Section \ref{se:entropyestimates}). Hence, if $H(P^1_0\,|\,P^2_0)=\infty$, then $H(P^1[t]\,|\,P^2[t])=\infty$, and \eqref{eqn:entropyidentity1} holds trivially. If $H(P^1_0\,|\,P^2_0)<\infty$, then we deduce from (i) and the assumption \eqref{asmp:finiteentropy2} that $H(P^1\,|\,P^2) < \infty$, and we are then in a position to apply (ii).

Hence, it remains to prove (i), and it suffices to do so for $t=T$. The argument is similar to \cite[Proposition 1]{lehec2013representation}, which covers the case $b^2\equiv 0$. Let $X=(X_t)_{t \in [0,T]}$ denote the canonical process (coordinate map) on $\C_T^k$, and let $\FF=(\F_t)_{t \in [0,T]}$ denote its natural (unaugmented) filtration. Abbreviate $b^i_t:=b^i(t,X)$. Under $P^i$, the process
\begin{align*}
W^i_t := X^i_t - X^i_0 - \int_0^t b^i_s\,ds
\end{align*}
is a Brownian motion.
Suppose first that $\int_0^T|b^1_t-b^2_t|^2dt \in L^\infty(P^1)$. By Novikov's criterion and Girsanov's theorem \cite[Corollary 3.5.13 and Theorem 3.5.1]{karatzas-shreve}, the process
\begin{align*}
M_t := \exp\left(\int_0^t(b^2_s-b^1_s) \cdot dW^1_s - \frac12\int_0^t|b^1_s-b^2_s|^2ds\right), \quad t \in [0,T],
\end{align*}
is then a martingale under $P^1$, and the process
\begin{align*}
\widetilde{W}^1_t := W^1_t + \int_0^t(b^1_s-b^2_s) \, ds, \quad t \in [0,T],
\end{align*}
defines a Brownian motion under the probability measure defined by $M_T dP^1$. By well-posedness of the SDE$(b^2)$, we have $M_T dP^1 = dP^2$. Thus, letting $\E^i$ denote expectation under $P^i$, we have
\begin{align*}
H(P^1\,|\,P^2) = -\frac12 \E^1[\log M_T] = \frac12 \E^1\int_0^T|b^1_t-b^2_t|^2\,dt,
\end{align*}
proving the claim.

To handle the general case, define the stopping times
\begin{align*}
\tau_n := T \wedge \inf\left\{t \in [0,T] : \int_0^t|b^1_s-b^2_s|^2ds \ge n\right\},
\end{align*}
where $\inf\emptyset := \infty$ as usual.
By well-posedness of the SDE$(b^2)$, there exists for each $(t,z) \in [0,T) \times \C^k_T$ a unique in law solution of the SDE
\begin{align*}
d\widehat{X}_s &= b^2(s,\widehat{X})ds + d\widehat{W}_s,  \text{ for } s \in (t,T], \quad \text{ and } \quad \widehat{X}_s = z_s,  \text{ for } s \in [0,t],
\end{align*}
and we write $P^2_{t,z} \in \P(\C_T^k)$ for its law. Set also $P_{T,z}:=\delta_z$ for $z \in \C_T^d$. With $M$ defined as above, the process $M_{\cdot \wedge \tau_n}$ is a martingale under $P^1$.
The process
\begin{align*}
W^1_t + \int_0^{t \wedge \tau_n}(b^1_s-b^2_s)ds, \quad t \in [0,T],
\end{align*}
is a Brownian motion under the probability measure defined by $M_{\tau_n} dP^1$.
Define $\F_{\tau_n} := \sigma(X_{s \wedge \tau_n} : s \in [0,T])$. (Note by \cite[Lemma 1.3.3]{stroock-varadhan} that this agrees with the usual definition of filtration at a stopping time.)
It is tempting to claim that the well-posedness of the SDE$(b^2)$ implies that $M_{\tau_n} dP^1 = dP^2$ on $\F_{\tau_n}$, but this is not immediate; under $M_{\tau_n} dP^1$, the canonical process $X$ solves the SDE($b^2$) but only on the interval $[0,\tau_n]$, not all of $[0,T]$, and we must prove that uniqueness carries over to this random sub-interval.
To do this, we follow  \cite[Theorem 6.1.2]{stroock-varadhan},
to construct a probability measure on $\C_T^k$ which essentially follows $M_{\tau_n} dP^1$ up to time $\tau_n$ and then switches to follow the SDE($b^2$) thereafter. As a consequence of Lemma \ref{le:Ptx-measurability} below, $P^2_{\tau_n,X_{\cdot \wedge \tau_n}}(A)$ is $\F_{\tau_n}$-measurable for each Borel set $A \subset \C_T^k$. We may then uniquely define $Q_n \in \P(\C_T^k)$ by requiring that $dQ_n=M_{\tau_n} dP^1$ on $\F_{\tau_n}$ and that $Q_n(A \,|\,\F_{\tau_n}) = P^2_{\tau_n,X_{\cdot \wedge \tau_n}}(A)$ a.s.\ for each Borel set $A$. Then, using the final conclusion of \cite[Theorem 6.1.2]{stroock-varadhan}, one may check using L\'evy's criterion that
\begin{align*}
X_t - \int_0^t b^2(s,X)\,ds, \quad t \in [0,T],
\end{align*}
is a Brownian motion under $Q_n$. This implies $Q_n=P^2$, and in particular $M_{\tau_n} dP^1=dQ_n=dP^2$ on $\F_{\tau_n}$. Hence,
\begin{align*}
H(P^1|_{\F_{\tau_n}} \,|\, P^2|_{\F_{\tau_n}}) &= -\E^1[\log M_{\tau_n}] = \frac12\E^1\int_0^{\tau_n}|b^1_t - b^2_t|^2dt \le \frac12\E^1\int_0^T|b^1_t - b^2_t|^2dt.
\end{align*}
The assumption \eqref{asmp:finiteentropy1} implies that $P^i(\tau_n \to T)=1$, which in turn  implies that $\|P^i|_{\F_{\tau_n}} - P^i\|_{\mathrm{TV}}\to 0$, for $i=1,2$. Send $n\to\infty$ and use lower semicontinuity of relative entropy. \hfill \qedsymbol

\begin{lemma} \label{le:Ptx-measurability}
Let $b : [0,T] \times \C_T^k \to \R^k$ be progressively measurable, and suppose the SDE$(b)$ is well-posed in the sense of Definition \ref{def:SDEwellposed}. 
For each  $(t,z) \in [0,T) \times \C^k_T$, let $P_{t,z} \in \P(\C_T^k)$ denote the law of the unique solution of the SDE
\begin{align*}
dX_s &= \overline{b}(s,X)ds + dW_s,  \text{ for } s \in (t,T], \quad \text{ and } \quad X_s = z_s,  \text{ for } s \in [0,t],
\end{align*}
Let $P_{T,z}:=\delta_z$ for $z \in \C_T^d$.
Then the map $[0,T] \times \C_T^k \ni (t,z) \mapsto P_{t,z} \in \P(\C_T^k)$ is progressively measurable, with $\P(\C_T^k)$ equipped with the Borel $\sigma$-field of the topology of weak convergence.
\end{lemma}
\begin{proof}
It follows from well-posedness that $P_{t,z}$ is non-anticipative in the sense that $P_{t,z} =P_{t,z'}$ for any $t \in [0,T]$  and $z,z' \in \C_T^k$ satisfying $z'|_{[0,t]}=z|_{[0,t]}$. We must show also that $(t,z) \mapsto P_{t,z}$ is Borel measurable.
It suffices by \cite[Theorem 14.12]{kechris2012classical} to show the Borel measurability of the graph
\begin{align*}
G := \left\{(t,x,P_{t,x}) : t \in [0,T], \, x \in \C_T^k\right\} \subset [0,T] \times \C_T^k \times \P(\C_T^k).
\end{align*} 
For each $t \in [0,T)$, define a map $\Phi_t : \C_T^k \to \C_T^k$ by
\begin{align*}
\Phi_t(x)(s) := \sqrt{\frac{T}{T-t}}\bigg(x_{t+s\frac{T-t}{T}}-x_t - \int_t^{t+s\frac{T-t}{T}} \overline{b}(r,x)\,dr\bigg), \quad s \in [0,T].
\end{align*}
Fix countable convergence-determining sets $\{f_n : n \in \N\}$ and $\{g_n : n \in \N\}$ of bounded continuous functions $f_n : \C_T^k \to \R$ and $g_n : \R^k \to \R$.  
Define $S(t,x)$ for each $(t,x)\in[0,T) \times \C_T^k$ as the set of $Q \in \P(\C_T^k)$ satisfying 
\begin{enumerate}[(a)]
\item $\int_{\C_T^k} f_n(z_{\cdot \wedge t})\, Q(dz) = f_n(x_{\cdot \wedge t})$ for all $n \in \N$,
\item $Q \circ \Phi_t^{-1}$ equals Wiener measure on $\C_T^k$, and
\item For every $n,m \in \N$ and rational numbers $s,u \in [t,T]$ with $s < u$,
\begin{align*}
\int_{\C_T^d}  & f_n(z_{\cdot \wedge s}) g_m ( \Phi_t(z)(r)-\Phi_t(z)(s))\,Q(dz) \\
	&= \int_{\C_T^k} f_n(z_{\cdot \wedge s})\,Q(dz) \int_{\C_T^d}g_m ( \Phi_t(z)(r)-\Phi_t(z)(s))\,Q(dz).
\end{align*}
\end{enumerate}
Let $S(T,x)=\{\delta_x\}$ for $x \in \C_T^k$.
Note that (a) is equivalent to the statement that $Q\{z \in \C_T^k : z_s=x_s \ \forall s \le t\}=1$, whereas (b) and (c) express the statement that $\Phi_t$ is a Brownian motion under $Q$, as it has the law of Wiener measure and independent increments with respect to the filtration generated by the coordinate process on $\C_T^k$. The well-posedness assumption that $P_{t,x}$ is in fact the \emph{unique} element of $S(t,x)$, for each $(t,x)$. The graph $G$ is thus precisely the set of $(t,x,Q) \in [0,T] \times \C_T^d \times \P(\C_T^d)$ satisfying the constraints specified in (a--c).
Note that $(t,x) \mapsto \Phi_t(x)$ is Borel measurable, and $\int_{\C_T^k} f_n(z_{\cdot \wedge t})\, Q(dz)$ and $f_n(x_{\cdot \wedge t})$ are continuous in $(t,x,Q)$. Thus, $G$ can be expressed as the intersection of countably many inverse images of Borel sets under various Borel maps, and we deduce that $G$ is Borel.
\end{proof}

\section{Proof of Lemma \ref{le:integ-transp-equiv}} \label{ap:transp-proof}

We first prove the implication (b) $\Rightarrow$ (a), with $\transpconst = 2(1+R)/\kappa$. Let $\nu \in \P(\C_T^d)$ and $(t,x) \in [0,T] \times \C_T^d$. For any unit vector $u \in \R^d$, we apply the weighted Pinsker inequality \eqref{ineq:weightedPinsker} with $f(y) :=  \kappa^{1/2} \, \left(b(t,x,y)-\langle\mu,b(t,x,\cdot)\rangle\right)$ to get
\begin{align*}
\left|  \langle \mu-\nu, b(t,x,\cdot)\rangle \right|^2 \le \frac{2}{\kappa} \left(1 + \log \int e^{ \kappa \left| b(t,x,y)-\langle\mu,b(t,x,\cdot)\rangle  \right|^2 }\,\mu(dy)\right) H(\nu\,|\,\mu).
\end{align*}
This is exactly (a).

We next prove that (a) $\Rightarrow$ (b), using a form of Marton's argument \cite{marton1986simple}. Fix $t \in (0,T)$ and $x \in \C_T^d$. Fix a coordinate $i\in \{1,\ldots,d\}$. Let $f_i(y)=b_i(t,x,y)-\langle\mu,b_i(t,x,\cdot)\rangle$ for $y \in \C_T^d$, noting that $\langle \mu,f_i\rangle=0$. For any Borel set $A \subset \C_T^d$ with $\mu(A) > 0$, we may apply (a) with $\nu \in \P(\C_T^d)$ defined by  $d\nu/d\mu=1_A/\mu(A)$ to get
\begin{align*}
\left|\frac{1}{\mu(A)}\int_A f_i\,d\mu \right|^2 = |\langle \nu-\mu,f_i\rangle|^2 &\le \transpconst H(\nu\,|\,\mu) = - \transpconst \log \mu(A).
\end{align*}
(Note that $f_i \in L^1(\nu)$ because $b(t,x,\cdot) \in L^1(\mu)$.)
We claim that $\mu(f_i \ge r) \le e^{-r^2/\transpconst}$ for all $r > 0$. This is trivially true if $\mu(f_i \ge r) = 0$.
If $\mu(f_i \ge r) > 0$, then applying the above with $A= \{f_i \ge r\}$ yields
\begin{align*}
 r^2 \le \bigg|\frac{1}{\mu(f_i \ge r)}\int_{\{f_i \ge r\}} f_i\,d\mu \bigg|^2 \le - \transpconst \log \mu(f_i \ge r).
\end{align*}
Similarly, $\mu(f_i \le -r) \le e^{-r^2/\transpconst}$. By a union bound,
\begin{align*}
\mu\big( |b(t,x,\cdot)-\langle\mu,b(t,x,\cdot)\rangle| \ge r\big) \le \sum_{i=1}^d \mu(|f_i|^2 \ge r^2/d) \le 2d e^{-r^2/\transpconst d}.
\end{align*}
Integrating this tail bound leads to (b), with constants $(\kappa,R)$ depending only on $(\transpconst,d)$. \hfill\qedsymbol

\section{Quadratic transport inequalities for Lipschitz SDEs}

The following proposition extends a well known idea to the case of random initial positions. Prior results, for non-random initial positions, may be found in \cite{djellout2004transportation,pal2012concentration,ustunel2012transportation}.

\begin{proposition} \label{pr:T2-randominit}
Let $k \in \N$. Suppose $b : [0,T] \times \C^k_T \to \R^k$ is progressively measurable, and 
\begin{align*}
\sup_{t \in [0,T]}|b(t,x)-b(t,y)| \le L\|x-y\|_t, \qquad \forall x,y \in \C^k_T.
\end{align*}
Assume also that $M_0 := \int_0^T|b(t,0)|^2dt < \infty$.
Let $P_0 \in \P(\R^k)$ have finite second moment, and let $P \in \P(\C^k_T)$ be the law of the unique solution of the SDE
\begin{align*}
dX_t = b(t,X)dt +  dW_t, \qquad X_0 \sim P_0.
\end{align*}
If there exists $C_0 < \infty$ such that
\begin{align}
\W_2^2(P_0,\nu) \le C_0 H(\nu\,|\,P_0), \quad \forall \nu \in \P(\R^k), \label{ap:asmp:Tineq_0}
\end{align}
then we have
\begin{align*}
\W_2^2(P,Q) \le 3(C_0 \vee 2T)e^{3TL^2} H(Q\,|\,P), \quad \forall Q \in \P(\C^k_T)
\end{align*}
\end{proposition}
\begin{proof}
We follow the ideas of \cite[Section 5]{djellout2004transportation} and \cite{ustunel2012transportation}, with some adaptations to account for the random initial state.
Let $X$ denote the canonical process on $\C^k_T$, so that 
\begin{align}
W_t := X_t - X_0 - \int_0^t b(s,X)\,ds, \quad t \in [0,T],  \label{pf:T2-B1}
\end{align}
defines a $P$-Brownian motion.
Let $\F_t=\sigma(X_s:s \le t)$ denote the canonical filtration on $\C^k_T$.
Let $Q \in \P(\C^k_T)$ satisfy $H(Q\,|\,P) < \infty$. We may assume that $dQ/dP$ is strictly positive; indeed, we may first study $dQ_\epsilon/dP_\epsilon = (\epsilon+1)^{-1}(\epsilon + dQ/dP)$ and then send $\epsilon \downarrow 0$ to deduce the general case from the strictly positive case, as in \cite[Proof of Theorem 9.1]{ustunel2012transportation}.

Let $\pi$ denote an optimal coupling for $\W_2^2(P_0,Q_0)$. We may identify $\pi$ as the pushforward of $Q_0 \otimes \mathrm{Leb}$ through the map $\R^d \times [0,1] \ni (x,u) \mapsto (\xi(x,u),x) \in \R^d \times \R^d$, where $\xi : \R^k \times [0,1] \to \R^k$ is a measurable function, and $\mathrm{Leb}$ denotes the uniform measure on $[0,1]$.

A standard argument using the martingale representation theorem (see \cite[Theorem 2.1]{leonard2012girsanov} for a general version) shows that there exists a progressively measurable process $\alpha : [0,T] \times \C^k_T \to \R^k$ such that $\E^P\int_0^T|\alpha(t,X)|^2dt < \infty$, and 
\begin{align*}
\frac{dQ}{dP} &= \frac{dQ_0}{dP_0}(X_0)\, \exp\left(\int_0^T\alpha(t,X) \cdot dW_t - \frac12 \int_0^T|\alpha(t,X)|^2dt\right), \ \ a.s.
\end{align*}
Letting $(Q^x)_{x \in \R^k}$ and $(P^x)_{x \in \R^k}$ denote the conditional measures of $Q$ and $P$ given the initial position, we have also
\begin{align*}
\frac{dQ^x}{dP^x} &= \exp\left(\int_0^T\alpha(t,X) \cdot dW_t - \frac12 \int_0^T|\alpha(t,X)|^2dt\right), \ \ a.s., \ \ x \in \R^k.
\end{align*}
Under $Q$, the process
\begin{align}
\widehat{W}_t := W_t - \int_0^t \alpha(s,X)ds, \quad t \in [0,T], \label{pf:T2-B2}
\end{align}
is a Brownian motion. Now, enlarge the probability space by setting $\Omega = \C^k_T \times [0,1]$, equipped with the probability measure $\overline{Q}=Q \otimes \mathrm{Leb}$ and the filtration $\overline{\F}_t := \F_t \otimes \B[0,1]$, where $\B[0,1]$ denotes the Borel $\sigma$-field on $[0,1]$.
All random variables defined on $\C^k_T$ extend in the obvious way to $\Omega$. Let $U:\Omega \to [0,1]$ denote the second coordinate map.
On this enlarged space, $\widehat{W}$ remains a $\overline{Q}$-Brownian motion. Let $\xi=\xi(X_0,U)$ be the map given in the paragraph above, noting that $\xi$ is $\overline{\F}_0$-measurable and satisfies
\begin{align*}
\overline{Q} \circ \xi^{-1} = \mu_0, \qquad \overline{Q} \circ (\xi^{-1},X_0) = \pi.
\end{align*}
Let $Y$ be the unique strong solution of the SDE
\begin{align*}
dY_t = b(t,Y)dt + \sigma d\widehat{W}_t, \qquad Y_0 = \xi.
\end{align*} 
Then $\overline{Q} \circ Y^{-1}=P$.
Recalling that $X : \Omega \to \C_T^k$ is the first coordinate map, we deduce that $\overline{Q} \circ (X,Y)^{-1}$ is a coupling of $(Q,P)$. This implies
\begin{align*}
\W_2^2(P,Q) &\le \E^{\overline{Q}}\|X-Y\|_T^2.
\end{align*}
To estimate this quantity, note for all $t \in [0,T]$ that \eqref{pf:T2-B1} and \eqref{pf:T2-B2} imply
\begin{align*}
|X_t-Y_t|^2 
&= \left| X_0 - \xi + \int_0^t\left(b(s,Y) - \alpha(s,X) - b(s,X)\right)ds\right| \\
&\le 3|X_0-\xi|^2 + 3TL^2\int_0^t\|X-Y\|_s^2ds + 3T\int_0^T|\alpha(s,X)|^2ds.
\end{align*}
By Gronwall's inequality,
\begin{align*}
\|X-Y\|_T^2 \le 3e^{3TL^2}\left(|X_0-\xi|^2 + T\int_0^T|\alpha(t,X)|^2dt\right).
\end{align*}
Recall by construction that the joint law of $(\xi,X_0)$ under $\overline{Q}$ is exactly the optimal coupling for $(\mu_0,\nu_0)$. Hence, taking expectations, we find
\begin{align*}
\W_2^2(P,Q) &\le 3e^{3TL^2}\left(\E^{\overline{Q}}|X_0-\xi|^2 + T\E^{\overline{Q}}\int_0^T|\alpha(t,X)|^2dt\right) \\
	&= 3e^{TL^2}\left(\W_2^2(P_0,Q_0) + 2T \int_{\R^k} H(Q^x\,|\,P^x)\,Q_0(dx)\right),
\end{align*}
with the last term coming from Lemma \ref{le:entropy-diffusions}(ii). Use the assumption \eqref{ap:asmp:Tineq_0} and the chain rule for relative entropy \eqref{def:chainrule} to get
\begin{align*}
\W_2^2(P,Q) &\le 3e^{3TL^2}\left(C_0H(Q_0\,|\,P_0) + 2T\int_{\R^k} H(Q^x\,|\,P^x)\,Q_0(dx)\right) \\
	&= 3(C_0 \vee 2T)e^{3TL^2} H(Q\,|\,P).
\end{align*}
\end{proof}

\bibliographystyle{amsplain}
\bibliography{biblio}

\end{document}